\documentclass[english]{amsart}
\usepackage{euscript,graphicx,epstopdf,amscd,amsgen,amsfonts,amssymb,latexsym,
amsmath,amsthm,graphicx,mathrsfs,times,color,overpic,bm}
\usepackage[all]{xy}
\UseRawInputEncoding

\usepackage{xcolor}

\definecolor{JungleGreen}{cmyk}{2,0,0.52,0}
\definecolor{green}{rgb}{0.1,0.6,0.}
\definecolor{blufrance}{rgb}{0.19, 0.55, 0.91}
\definecolor{brickred}{rgb}{1, 0.25, 0.5}
\definecolor{arsenic}{rgb}{0.23, 0.27, 0.29}
\definecolor{amaranth}{rgb}{0.9, 0.17, 0.31}
\definecolor{dgreen}{rgb}{0.0, 0.5, 0.25}
\definecolor{amaranth}{rgb}{0.9, 0.17, 0.31}
\definecolor{americanrose}{rgb}{1.0, 0.01, 0.24}
\definecolor{persianindigo}{rgb}{0.2, 0.07, 0.48}

\newcommand{\spac}[1]{{\quad\text{ #1 }\quad}}
\usepackage[T1]{fontenc}

\newtheorem{mtheo}{Theorem}

\newtheorem{mpro}[mtheo]{Proposition}

\newtheorem{theorem}{Theorem}[section]
\newtheorem{lemma}[theorem]{Lemma}
\newtheorem{mlemma}[theorem]{Main Lemma}
\newtheorem{claim}[theorem]{Claim}

\newtheorem{corollary}[theorem]{Corollary}
\newtheorem{proposition}[theorem]{Proposition}

\newtheorem{notation}[theorem]{Notation}

\theoremstyle{definition}
\newtheorem{definition}[theorem]{Definition}

\newtheorem{remark}[theorem]{Remark}
\newtheorem{assumption}[theorem]{Assumption}

\newtheorem*{remark*}{Remark}

\newcommand{\eqdef}{\stackrel{\scriptscriptstyle\rm def}{=}}
\DeclareMathOperator{\Att}{Att}

\DeclareMathOperator{\var}{var}
\DeclareMathOperator{\id}{id}

\DeclareMathOperator{\diam}{diam}

\DeclareMathOperator{\card}{card}
\DeclareMathOperator{\interior}{int}

\DeclareMathOperator{\cudot}{{\bigcup\hspace{-0.31cm}{\boldsymbol{\cdot}}\hspace{0.15cm}}}
\DeclareMathOperator{\cupdot}{\cudot\limits}
\DeclareMathOperator{\smallcudot}{{\cup\hspace{-0.18cm}{\boldsymbol{\cdot}}\hspace{0.15cm}}}
\DeclareMathOperator{\smallcupdot}{\smallcudot\limits}

\def\cB{\EuScript{W}}
\def\cL{\EuScript{L}}

\def\frakR{\mathfrak{R}}
\def\fX{\mathfrak{X}}

\DeclareMathOperator{\Mod}{Mod}
\DeclareMathOperator{\PCs}{PCS}
\DeclareMathOperator{\Cs}{CS}
\newcommand{\PCS}[1]{\PCs(#1)}
\newcommand{\CS}[1]{\Cs(#1)}

\def\c{{\rm c}}

\def\bE{\mathbb{E}}
\def\bA{\mathbf{A}}
\def\bN{\mathbb{N}}

\def\bZ{\mathbb{Z}}
\def\bR{\mathbb{R}}
\def\bP{\mathbb{P}}
\def\bS{\mathbb{S}}

\def\cA{\mathcal{A}}

\def\fG{\EuScript{G}}
\def\fp{\mathfrak{p}}

\def\fS{\EuScript{S}}
\def\fR{\EuScript{P}}
\def\fP{\EuScript{P}}
\def\fT{\EuScript{T}}
\def\fL{\EuScript{L}}

\def\cO{\EuScript{O}}

\def\cM{\EuScript{M}}

\def\ft{\mathfrak{t}}
\def\fb{\mathfrak{b}}
\def\fB{\mathfrak{B}}

\def\Sub{\mathcal{S}}
\def\Cut{\mathcal{C}}

\def\c{{\rm c}}

\def\D{\lVert F\rVert}

\def\va{{\mathbf a}}

\numberwithin{equation}{section}

\DeclareMathSymbol{\varnothing}{\mathord}{AMSb}{"3F}
\renewcommand{\emptyset}{\varnothing}

\thanks{}

\begin{document}

\title[Variational principle for nonhyperbolic ergodic measures]{Variational principle\\for nonhyperbolic ergodic measures: \\Skew products and elliptic cocycles}
\author[L.~J.~D\'iaz]{L. J. D\'\i az}
\address{Departamento de Matem\'atica PUC-Rio, Marqu\^es de S\~ao Vicente 225, G\'avea, Rio de Janeiro 22451-900, Brazil}
\email{lodiaz@mat.puc-rio.br}
\author[K.~Gelfert]{K.~Gelfert}
\address{Instituto de Matem\'atica Universidade Federal do Rio de Janeiro, Av. Athos da Silveira Ramos 149, Cidade Universit\'aria - Ilha do Fund\~ao, Rio de Janeiro 21945-909,  Brazil}\email{gelfert@im.ufrj.br}
\author[M.~Rams]{M. Rams} \address{Institute of Mathematics, Polish Academy of Sciences, ul. \'{S}niadeckich 8,  00-656 Warszawa, Poland}
\email{rams@impan.pl}

\begin{abstract}
For a large class of transitive non-hyperbolic systems, we construct nonhyperbolic ergodic measures with entropy arbitrarily close to its maximal possible value. The systems we consider are partially hyperbolic with one-dimension central direction for which there are positive entropy ergodic measures whose central Lyapunov exponent is negative, zero, or positive. We construct ergodic measures with zero central Lyapunov exponent whose entropy is positive and arbitrarily close to the topological entropy of the set of points with central Lyapunov exponent zero. This provides a restricted variational principle for nonhyperbolic (zero exponent) ergodic measures.

The result is applied to the setting of $\mathrm{SL}(2,\bR)$ matrix cocycles and provides a counterpart to Furstenberg's classical result: for an open and dense subset of elliptic $\mathrm{SL}(2,\bR)$  cocycles we construct ergodic measures with upper Lyapunov exponent zero and with metric entropy arbitrarily close to the topological entropy of the set of infinite matrix products with subexponential growth of the norm.
\end{abstract}

\begin{thanks}{This research has been supported [in part] by the Coordenaç\~ao de Aperfeiçoamento de Pessoal de N\'ivel Superior - Brasil (CAPES) - Finance Code 001, by CNPq-grants, CNPq Projeto Universal, and INCT-FAPERJ  (Brazil) and by National  Science  Centre  grant 2019/33/B/ST1/00275 (Poland). The authors acknowledge the hospitality of IMPAN, IM-UFRJ, and PUC-Rio. }\end{thanks}


\maketitle
\tableofcontents

\section{Introduction}

Topological entropy, metric entropy, and Lyapunov exponents are key concepts in ergodic theory and thermodynamical formalism to quantify the complexity of  dynamical systems. Several classical results such as the variational principle for entropy \cite{Wal:75} and Ruelle's inequality \cite{Rue:78} provide relations between them. On the other hand, Oseledets' theorem establishes the framework to study the Lyapunov exponents of invariant measures, see \cite{Rue:79}. An ergodic measure is \emph{nonhyperbolic} if its Oseledets splitting has some bundle whose Lyapunov exponent is zero. Otherwise the measure is called \emph{hyperbolic}. In what follows, we use the terms \emph{nonhyperbolic} and \emph{hyperbolic} only for \emph{ergodic measures}. 

A particularly interesting setting are partially hyperbolic systems which, by assumption, have several globally defined, continuous, and invariant subbundles which carry implicitly information about Lyapunov exponents. When investigating nonhyperbolic measures of those systems, it suffices to focus on the ``central bundle'' $E^\c$ which neither displays uniform contraction nor expansion and hence detects nonhyperbolicity. 
We study settings where $E^\c$ is one-dimensional and nonhyperbolic measures are robustly present and essential: they exist, some of them have positive entropy, and these two properties hold also for small perturbations of the dynamics. Moreover, the systems are genuinely nonhyperbolic as they display simultaneously nonhyperbolicity and hyperbolicity of different types in the central bundle: there exist ergodic measures $\mu$ for which the central-Oseledets exponent $\chi^\c(\mu)$ (relative to $E^\c$) is negative, zero or positive. The nonhyperbolic nature of these systems is also reflected by the fact that specification-like properties are not satisfied and that inside this class robust heterodimensional cycles occur densely. 

We aim to understand the ``total amount'' of nonhyperbolicity that can be detected on the ergodic level.  Here our focus is on entropy. To put this discussion into a broader context, recall the concept of  topological entropy $h_{\rm top}$ of continuous maps on general sets (not necessarily compact or invariant)  introduced by Bowen \cite{Bow:73}. One of the key results in \cite{Bow:73} is that the entropy of an ergodic measure bounds from below the entropy of the set of its generic points. This  result has an immediate consequence in the study of the ``set of nonhyperbolicity'' when the Lyapunov exponent $\chi^\c(x)$ of a point $x$ is the Birkhoff average of a continuous potential at that point. 
Bowen's result then implies that the entropy of a nonhyperbolic ergodic measure bounds from below the entropy of the set of nonhyperbolic points
\[
	\sup\{h(\mu)\colon \mu\text{ ergodic, }\chi^\c(\mu)=0\}
	\le h_{\rm top}(\cL^\c(0)),
\]
{where}
\[
	\cL^\c(0)
	\eqdef \{x\colon\chi^\c(x)=0\}.
\]
In general, it is unknown if, in terms of entropy, the set $\cL^\c(0)$ can be larger. One of the goals of this paper is to explore this relation. We exhibit chaotic settings where the above is an equality. Let us now be more precise. 

\subsection{Skew products and cocycles}
Given $N\ge2$, consider a finite family $f_i\colon \bS^1\to \bS^1$, $i=1,\ldots,N$, of $C^1$ diffeomorphisms and the associated  step skew product
\begin{equation}\label{eq:sp}
	F\colon \Sigma_N\times \bS^1\to \Sigma_N\times \bS^1,
	\quad
	F(\xi,x) =(\sigma(\xi), f_{\xi_0}(x)),
\end{equation}
where $\Sigma_N=\{1,\ldots,N\}^\bZ$ and $\sigma$ is the usual left shift in this space.
Given $X=(\xi,x)\in \Sigma_N\times\bS^1$, consider its \emph{(fiber) Lyapunov exponent} 
\[
	\chi(X)
	\eqdef
	 \lim_{n\to\pm\infty}\frac{1}{ n}\log\,\lvert (f_{\xi}^n)'(x)\lvert  ,
\]
where for $\xi=(\ldots,\xi_{-1}|\xi_0,\xi_1,\ldots)$ we write
\begin{equation}\label{eq:reference}
	 f_\xi^{-n}\eqdef f_{\xi_{-n}}^{-1}\circ\cdots \circ f_{\xi_{-1}}^{-1}
	 \quad\text{ and }\quad
	 f_\xi^n \eqdef f_{\xi_{n-1}}\circ\cdots\circ f_{\xi_0},
\end{equation}	 
and we assume that both limits $n\to\pm \infty$  exist and coincide. Otherwise we say that the Lyapunov exponent $\chi(X)$ does not exist.
We will analyze the topological entropy of the \emph{level sets of Lyapunov exponents}: given $\alpha\in\bR$ let
\begin{equation}\label{def:L0}
	\cL(\alpha)
	\eqdef \big\{X\in\Sigma_N\times\bS^1\colon \chi(X)=\alpha\big\}.
\end{equation}
Denote by $\cM_{\rm erg}(F)$ the space of $F$-ergodic measures. Given $\mu\in\cM_{\rm erg}(F)$, denote by $h(F,\mu)$ its \emph{metric entropy} and by $\chi(\mu)$ its \emph{Lyapunov exponent}  
 \[
 	\chi(\mu)
	\eqdef \int\log\,\lvert(f_{\xi_0})'(x)\rvert\,d\mu(\xi,x).
\]

In continuation to our introduction above, maps $F$ as above can be viewed as a special case of partially hyperbolic diffeomorphisms whose one-dimensional central bundle $E^\c$ is integrable. Moreover, $E^\c$ is tangent to the circle fibers and $\chi^\c=\chi$. The class $\mathrm{SP}^1_{\rm shyp}(\Sigma_N\times\bS^1)$ of maps $F$ studied here was introduced in \cite{DiaGelRam:17}, see Section \ref{sec:axioms} for its definition. They capture the relevant properties of the so-called robustly nonhyperbolic transitive sets,  reformulating the main properties of those systems in the setting of skew products over the shift of $N$ symbols whose fiber maps are $C^1$ diffeomorphisms of the circle $\bS^1$. Every $F\in \mathrm{SP}^1_{\rm shyp}(\Sigma_N\times\bS^1)$ is transitive and has so-called  contracting and expanding blenders which, by transitivity, are connected; see also the contraction-expansion-rotation examples as introduced in \cite{GorIlyKleNal:05}. For details see the discussion in \cite[Section 8]{DiaGelRam:17}. A paradigmatic setting  arises from the projective action of $2\times2$ matrix-cocycles. In particular, our approach applies to an open and dense subset of the so-called elliptic $\mathrm{SL}(2,\bR)$ cocycles. 

In our setting, by \cite{GorIlyKleNal:05}, besides hyperbolic measures with negative or positive exponent,  there are nonhyperbolic measures.
As a consequence of \cite{BocBonDia:16}, they can be chosen with positive entropy and hence $h_{\rm top}(F,\cL(0))>0$. The arguments in \cite{BocBonDia:16} are based on the construction of a compact invariant set with positive topological entropy consisting only of  points with zero Lyapunov exponent. Hence the existence of nonhyperbolic measures with positive entropy is a consequence of the classical variational principle for entropy \cite{Wal:75}. Though, the construction in \cite{BocBonDia:16} studies a very specific region of the space (the dynamics associated to some robust cycle involving a blender) and presumably the captured entropy  is much smaller than $h_{\rm top}(F,\cL(0))$. A natural question is if there exist nonhyperbolic measures whose entropy is equal or arbitrarily close to $h_{\rm top}(F,\cL(0))$. We answer positively the second question. The notoriously much harder question about the existence of measures maximizing entropy remains open.

On the other hand, by \cite{DiaGelRam:19}, for $F\in \mathrm{SP}^1_{\rm shyp}(\Sigma_N\times\bS^1)$ the closure of the ergodic measures is the union of two Poulsen simplices (corresponding to negative and positive Lyapunov exponent, respectively) which ``glue along'' nonhyperbolic measures. In particular, any nonhyperbolic measure is a weak$\ast$ and entropy-limit of hyperbolic ones. Moreover the spectrum of the exponent $\chi$ is a closed interval containing negative and positive numbers and for every $\alpha\ne0$ it holds
\[
	\sup\{h(F,\mu)\colon\mu\in\cM_{\rm erg}(F),\chi(\mu)=\alpha\}
	= h_{\rm top}(F,\cL(\alpha))
\]
 and
\begin{equation}\label{eq:varpinrc}\begin{split}
	\sup\{h(F,\mu)&\colon \mu\in\cM_{\rm erg}(F),\chi(\mu)=0\}\\
	&\le \lim_{\varepsilon\to0}
		\sup\{h(F,\mu)\colon \mu\in\cM_{\rm erg}(F),		
			\chi(\mu)\in(-\varepsilon,0)\cup(0,\varepsilon)\}\\
	&=h_{\rm top}(F,\cL(0))	.
\end{split}\end{equation}
This shows that measures with ``weak hyperbolicity'' are well inserted in this space and  are key ingredients to describe nonhyperbolic ones. 

The previous analysis is however insufficient to state in ``what amount'' weakly hyperbolic (and nonhyperbolic) measures contribute to the complexity of the dynamics. For example, in general it is unknown if  any term in \eqref{eq:varpinrc} attains the maximal entropy $\log N$.  Rigidity results for partially hyperbolic diffeomorphisms in \cite{TahYan:19} suggest that generically one should expect that high entropy-measures are hyperbolic. Indeed,  also assuming proximality%
\footnote{\emph{Proximality} holds if for every pair of points $x,y\in\bS^1$ there is $\xi\in\Sigma_N$ so that $\lvert f_\xi^n(x)-f_\xi^n(y)\rvert\to0$ and $\lvert f_\xi^{-n}(x)-f_\xi^{-n}(y)\rvert\to0$ as $n\to\infty$.}%
, by \cite[Theorem 2]{DiaGelRam:19}, there are exactly two ergodic measures maximizing entropy 
and they are both hyperbolic (with negative and positive Lyapunov exponent, respectively) and project to the entropy-maximizing Bernoulli measure in $\Sigma_N$. In particular, proximality implies that all terms in \eqref{eq:varpinrc} are   strictly less than $\log N$. 

The main result claims that in our setting there are nonhyperbolic ergodic measures whose entropy is as large as possible. As a consequence, a restricted variational principle holds for those nonhyperbolic measures. 

\begin{mtheo}\label{teo:1}
	For every $F\in \mathrm{SP}^1_{\rm shyp}(\Sigma_N\times\bS^1)$, $N\ge2$, it holds
\[\begin{split}
	h_{\rm top}(F,\cL(0))
	&= \sup\{h(F,\mu)\colon \mu\in\cM_{\rm erg}(F),\chi(\mu)=0\}\\
	&= \lim_{\varepsilon\to0}\sup\{h(F,\mu)
		\colon\mu\in\cM_{\rm erg}(F),
			\chi(\mu)\in(-\varepsilon,0)\}\\
	&= \lim_{\varepsilon\to0}\sup\{h(F,\mu)
		\colon\mu\in\cM_{\rm erg}(F),
			\chi(\mu)\in(0,\varepsilon)\}. 
\end{split}\]	
\end{mtheo}

Theorem \ref{teo:1} follows from a more quantitative result, see Theorem~\ref{teo:2} stated below. 

Let us now draw a consequence for $2\times 2$-matrix cocycles generated by a finite collection of matrizes  $\bA=\{A_1,\ldots,A_N\}$ in $\mathrm{SL}(2, \mathbb{R})^N$, $N\ge2$. The action of any matrix on the projective line $\bP^1$ (which is topologically the circle $\bS^1$) is a very special diffeomorphism. Given $A_i$, we define 
\begin{equation}\label{neq:defsteskecoc}
	f_i =f_{A_i}\colon \bP^1 \to \bP^1, \quad
	f_i (v)
	\eqdef \frac{A_i v}{\lVert A_i v\rVert}
\end{equation}
and denote by $F_\bA$ the associated skew product generated by the maps $f_1,\ldots,f_N$ as in \eqref{eq:sp}. 
Note that the spectrum of Lyapunov exponents of the cocycle $\bA$ and the fiber Lyapunov spectrum of $F_{\bA}$ are related (see \cite[Section 11]{DiaGelRam:19} for details) and hence our results can be translated to the elliptic cocycles setting. We postpone the details to Section \ref{sec:cocyc}.

Consider the (\emph{upper}) \emph{Lyapunov exponent} of $\xi^+=(\xi_0,\xi_1,\ldots)\in\Sigma_N^+\eqdef\{1,\ldots,N\}^{\bN_0}$ by
\begin{equation}\label{eq:upper}
	\lambda_1(\bA,\xi^+)
	\eqdef\lim_{n\to\infty}\frac1n\log\,
		\lVert A_{\xi_{n-1}}\circ\cdots\circ A_{\xi_0}\rVert
\end{equation}
whenever this limit exists. Given an ergodic measure $\nu^+$ with respect to the left shift $\sigma^+$ on $\Sigma_N^+$, by the subadditive ergodic theorem, almost surely it holds
\[
	\lambda_1(\bA,\xi^+)
	= \lambda_1(\bA,\nu^+)
	\eqdef \lim_{n\to\infty}\int\frac1n\log\,
		\lVert A_{\eta_{n-1}}\circ\cdots\circ A_{\eta_0}\rVert\,d\nu^+(\eta^+).
\]
 Note that $\lambda_1(\bA,\nu^+)\ge0$. Moreover, given any (nondegenerate) Bernoulli measure $\fb^+$, Furstenberg's theorem \cite{Fur:63} states that, assuming that the semi-group generated by $\bA$ is not relatively compact and there is no finite set $\emptyset\ne L\subset\bP^1$ such that $A(L)=L$ for every $A\in \bA$, then $\lambda_1(\bA,\fb^+)>0$.%
 \footnote{Furstenberg's result states the dichotomy ``positive Lyapunov exponent versus rigid dynamics''. As we are, by hypotheses,  in a non-rigid context, this implies always positive exponent.}

Bernoulli measures are rather specific ergodic measures and, besides Lyapunov-maxi\-mizing ones, very little is known about the ergodic theory of measures in this context. The following result complements this line of research (see also \cite{BocRam:16}) and can be read as a study of Lyapunov-minimizing measures 
 of matrix cocycles in our non-rigid context.
Recall that $\bA$ is \emph{elliptic} if its associated multiplicative semigroup  contains some \emph{elliptic} element $R$ (i.e., the absolute value of the  trace of $R$ is less than $2$). The set  $\mathfrak{E}_N$ of elliptic cocycles is open. In \cite{DiaGelRam:19} it is introduced an open and dense subset $\mathfrak{E}_{N,\rm shyp}$ of $\mathfrak{E}_N$, the so-called \emph{elliptic cocycles having some hyperbolicity}. The key property is that if $\bA\in \mathfrak{E}_{N,\rm shyp}$ then $F_{\bA}\in \mathrm{SP}^1_{\rm shyp}(\Sigma_N\times\bP^1)$, see Section \ref{sec:cocyc} for details. Also observe that we are in the case of proximality, which implies that entropy is positive and less than $\log N$.

Analogously to \eqref{def:L0} define the set of nonhyperbolic matrix concatenations by
\begin{equation}\label{eq:orquic}
	\cL_\bA^+(0)
	\eqdef \{\xi^+\in\Sigma_N^+\colon \lambda_1(\bA,\xi^+)=0\}.
\end{equation}

\begin{mtheo}\label{teo:3}
	For every $N\ge2$ and every $\bA$ in the open and dense subset $\mathfrak{E}_{N,\rm shyp}$ of $\mathfrak{E}_N$ it holds
\[
	0
	<h_{\rm top}(\sigma^+,\cL^+_\bA(0))
	= \sup\{h(\sigma^+,\nu^+)\colon \nu^+\in\cM_{\rm erg}(\sigma^+)
		,\lambda_1(\bA,\nu^+)=0\}
	<\log N	.
\]	
\end{mtheo}

Let us now discuss the tools to prove the above results and describe their context. 

\subsection{Nonhyperbolic measures: constructions and tools}

When dealing with nonhyperbolic ergodic measures, one major problem is that, at the current state of the art, there are very few general tools available as there are for hyperbolic ones (for example Pesin theory \cite{BarPes:07}). 
Among the few tools available to deal with nonhyperbolic measures are the so-called invariance principles in the spirit of Furstenberg's result \cite{Fur:63} and also \cite{Led:86,Cra:90,AviVia:10}. This principle is very well adapted to dynamics arising from cocycles and, in very rough terms, states that if the fiber Lyapunov exponent is zero then the fiber dynamics carries some transversally invariant structure. Though, these tools apply only to base measures which have a local product structure. In general, it is unknown if measures maximizing \eqref{eq:varpinrc} fall into this  category. 

An alternative approach is the explicit construction of nonhyperbolic measures. Naively, one can think of taking a weak$\ast$ limit of hyperbolic ergodic measures with central exponents approaching zero. Though, it is in general not guaranteed that the limit measure is ergodic and \emph{nontrivial} (i.e., with uncountable support).
The control of entropy of the limit measure is another issue.%
\footnote{By \cite[Corollary 1.2]{DiaFis:11} (see also \cite{CowYou:05}), in our setting, the entropy map is upper semi-continuous.}
 This approach was implemented and improved in several steps. It was initiated by the so-called \emph{GIKN construction} in \cite{GorIlyKleNal:05} for circle fiber-skew products. It was generalized first in \cite{KleNal:07} for certain partially hyperbolic diffeomorphisms and thereafter in \cite{DiaGor:09,BonDiaGor:10,WanZha:20} for nonhyperbolic homoclinic classes.%
\footnote{Our focus here is on as-large-as-possible entropy. The GIKN construction can be adapted and extended to produce nonhyperbolic measures with zero entropy and full support (see \cite{BonDiaGor:10,BocBonDia:18,BonZha:19}). 
The method in \cite{BocBonDia:16} was modified in \cite{BonDiaKwi:} to get nonhyperbolic measures with positive entropy and also full support. It was adapted also in \cite{BocRam:16} to deal with matrix cocycles. 
The constructions in this paper lay the foundations to construct nonhyperbolic measures with entropy as large as possible and also full support, following the ideas in \cite{BonDiaGor:10,BonDiaKwi:}.}
\footnote{Concerning nonhyperbolic measures with several zero Lyapunov exponents (that is, a higher-dimensional central bundle), the state of the art is very incipient, see results in \cite{BocBonDia:14} for iterated function systems and in \cite{WanZha:20} for some nonhyperbolic homoclinic classes.} 

Our construction is naively inspired by the GIKN method in \cite{GorIlyKleNal:05} that we proceed to sketch. This construction starts from an appropriate sequence of periodic orbits $\mathcal{O}_n$, say, expanding in the fiber direction.  In very rough terms, each periodic orbit has two parts: one ``repeats and shadows'' the previous orbit and the second part is a ``tail''.  There is a balance between both parts. The tail is used to ``spread" the support of the measures and to decrease the Lyapunov exponent by visiting a fiber-contracting region. A general criterium in \cite{GorIlyKleNal:05} (see also Proposition \ref{pro:mainrussian}) guarantees ergodicity of any limit measure. 

As stated in \cite{KwiLac:}, any limit measure of a GIKN construction involving a limit of a sequence of periodic measures has zero entropy. Thus, this method is not useful for our purpose since we aim for positive (maximal) entropy. We extend the GIKN approach of repeating and tailing to a much broader context to enable to capture positive entropy.  In our construction we replace periodic measures by specifically chosen ``Bernoulli-suspended measures on horseshoes'' carrying enough entropy. This choice is based on a ``skeleton property'' of hyperbolic measures with negative Lyapunov exponent. More precisely, we choose sufficiently many orbit pieces capturing the ergodic properties of the measure in finite time (compare also \cite[Section 4]{DiaGelRam:17}) to ``transfer a substantial amount'' of its entropy to the constructed nonhyperbolic measure (see Theorem \ref{teo:2}). The skeletons provide the ``repeat''-part. We combine this approach with ideas in the proof of \cite[Theorem 5 item 2]{DiaGelRam:17} which states how much entropy from measures with negative exponent can ``carry over'' to those with positive exponent. This provides the ``tail''-part. To prove Theorem \ref{teo:1}, we will consider measures with exponent close to zero and entropy close to the target entropy, though our construction is general. 

The verification of ergodicity of the limit measure (see \cite{GorIlyKleNal:05} and Proposition \ref{pro:mainrussian}) relies on the control of Birkhoff averages simultaneously on all scales on large measure sets. It is relatively easy to check if the measures of the sequence are \emph{periodic} (that is, supported on a periodic orbit). Indeed, on a periodic orbit Birkhoff averages converge uniformly, which makes this verification relatively simple. In our study, checking ergodicity of limit measures is much more intricate and requires a new approach. Here, we will deal with ergodic measures on horseshoes and have to rely on large deviation arguments from probability theory to identify appropriate sets with control of  Birkhoff averages. 

Our approach has roughly two, somewhat independent, parts: an abstract model for the repeat-and-tail scheme and its implementation. As abstract model we consider a cascade%
\footnote{As in this paper we consider plenty of sequence spaces, we prefer this terminology.}
 of abstract suspension spaces of Bernoulli shifts where we perform the large deviation control. This model is chosen such that it extends the corresponding measure preserving systems on a cascade of horseshoes in the product space $\Sigma_N\times\bS^1$. Each horseshoe has a coded system inside the base shift space $\Sigma_N$ which is obtained by a ``repeat-and-tail process''. Each coded system is uniquely left decipherable, which is the key ingredient to the fact that entropy is not lost when considering this factor. 

We perform this analysis in our context of skew products. However, the general idea of an entropy preserving cascade of suspensions of coded systems is fairly general and can be applied to partially hyperbolic diffeomorphisms following the scheme sketched in \cite[Section 8.3]{DiaGelRam:17} and implemented with all details in \cite{DiaGelSan:20}, see also \cite{YanZha:20}. But this goes beyond the goal of this paper.

\subsection{Organization}

In Section \ref{sec:results}, we state the remaining main results. In particular, we provide Theorem \ref{teo:2} which is the key result towards Theorem \ref{teo:1}. Section \ref{sec:axioms} describes our axiomatic setting and defines the class $\mathrm{SP}^1_{\rm shyp}(\Sigma_N\times\bS^1)$. We also detail the consequences for elliptic cocycles in Section \ref{sec:cocyc}.

The abstract model is developed in Sections \ref{sec:notation}--\ref{sec:cascade-1}. Section \ref{sec:notation} collects some basic properties of coded systems. In Section \ref{se:suspmode}, we define the suspension of a Bernoulli shift and recall some fundamental properties. In particular, we state a key result on large deviations. In Section \ref{sec:cascade-1}, we consider a cascade of those abstract suspension spaces assuming some growth condition of the associated roof functions.

The implementation of this abstract model is done in Sections \ref{sec:horses}--\ref{sec:internal}. In Section \ref{sec:horses}, we return to consider our skew product setting and study horseshoes which are defined by attractors of contracting iterated function systems (CIFS) induced by the family $\{f_i\}_{i=1}^N$ on some interval $J\subset\bS^1$. Here the idea of skeleton plays an important role. 
Section \ref{ssec:horsreptai} introduces a repeat-and-tail scheme. It induces a cascade of CIFSs and hence a cascade of horseshoes whose properties are studied in Section \ref{sec:ncashor}. In particular, in Section, \ref{sec:internal} we describe the inherited internal self-similar structures across this cascade.

The proof of Theorem \ref{teo:2} is split into Sections \ref{sec:proooof} and \ref{sec:proooofb}. In the Appendix we slightly reformulate and prove a result from \cite{GorIlyKleNal:05} guaranteeing ergodicity of a weak$\ast$ limit measure.

\section{Statement of results}\label{sec:results}

Before stating our remaining main results in Section \ref{subsec:res}, let us first give the complete description of our setting. In Section \ref{sec:cocyc} we discuss matrix cocycles in detail.

\subsection{Axiomatic setup}\label{sec:axioms}

For $n\in\bN$ let $\Sigma_N^n\eqdef\{1,\ldots,N\}^n$ and define $\Sigma_N^\ast \eqdef \bigcup_{n=1}^\infty\Sigma_N^n$. 
Given a finite sequence $\xi=(\xi_0,\ldots ,\xi_{n-1})\in\Sigma_N^\ast$, we denote by $\lvert\xi\rvert=n$ its length. Given $x\in\bS^1$, using the notation in \eqref{eq:reference}, consider its \emph{forward} and \emph{backward orbits} defined by
\[
	\cO^+(x)
	\eqdef \bigcup_{n\ge0}\,\,\bigcup_{\xi\in\Sigma_N}f_{\xi}^n(x)
	\spac{and}
	\cO^-(x)
	\eqdef \bigcup_{m\ge1}\,\,\bigcup_{\xi\in\Sigma_N}f_{\xi}^{-m}(x),
\]
respectively. Let $\cO^\pm(H)\eqdef \bigcup_{x\in H}\cO^\pm(x)$ for any subset $H\subset\bS^1$.
Given an interval $H$, we denote by $\lvert H\rvert$ its length. We assume that $\bS^1$ has length one.

We require the following properties to be satisfied. 

\medskip\noindent
\textbf{T (Transitivity)} There is $x\in\bS^1$ such that $\cO^+(x)$ and $\cO^-(x)$ are both dense in $\bS^1$.
%

\medskip\noindent
\textbf{CEC+($J^+$) (Controlled Expanding forward Covering  relative to $J^+$).} 
The set $J^+\subset\bS^1$ is a nontrivial closed interval such that there exist positive constants $K_1,\ldots,K_5$ so that for every interval $H\subset\bS^1$ intersecting $J^+$ with $\lvert H\rvert<K_1$ it holds
\begin{itemize}
\item  (controlled covering)  there exists a finite sequence $(\eta_0,\ldots,\eta_{\ell-1})$ for some positive integer $\ell\le  K_2\,\lvert\log\,\lvert H\rvert\rvert +K_3$ such that
\[
	\left(f_{\eta_{\ell-1}}\circ\cdots\circ f_{\eta_0}\right)(H)\supset B(J^+,K_4),
\]	
where $B(J^+,\delta)$ is the $\delta$-neighborhood of the set $J^+$,
\item  (controlled expansion) for every $x\in H$ we have
\[
	\log \,\lvert  \left(f_{\eta_{\ell-1}}\circ\cdots\circ f_{\eta_0}\right)'(x)\rvert
	\ge \ell K_5.
\]	
\end{itemize}

\medskip\noindent
\textbf{CEC$-(J^-$) (Controlled Expanding backward Covering relative to $J^-$).} The step skew product  $F^{-1}$ satisfies the Axiom CEC$+(J^-)$.

\medskip\noindent
\textbf{Acc$+$($J^+$) (forward Accessibility  relative to $J^+$).} $\cO^+(\interior J^+)=\bS^1$.

\medskip\noindent
\textbf{Acc$-$($J^-$) (backward Accessibility relative to $J^-$).} $\cO^-(\interior J^-)=\bS^1$.

\begin{definition}[The set $\mathrm{SP}^1_{\rm shyp}(\Sigma_N\times\bS^1)$]
A skew product $F$ as in \eqref{eq:sp} belongs to $\mathrm{SP}^1_{\rm shyp}(\Sigma_N\times\bS^1)$ if it satisfies Axioms T (transitivity), CEC$\pm(J^\pm)$, and Acc$\pm(J^\pm)$ for some closed intervals $J^-,J^+\subset\bS^1$, which are called \emph{backward} and \emph{forward blending intervals}, respectively. 
\end{definition}

We state some consequences of our axioms. 

\begin{remark}[Common blending interval and quantifiers]\label{rem:commonblending}
Let $F\in \mathrm{SP}^1_{\rm shyp}(\Sigma_N\times\bS^1)$. 
By \cite[Lemma 2.3]{DiaGelRam:17}, there are positive constants $K_1,\ldots,K_5$, and $K_6$ such that for every $\delta\in(0,K_6/2)$ and $x\in\bS^1$ the interval $J=[x-2\delta,x+2\delta]$ satisfies Axioms CEC$\pm(J)$ and Acc$\pm(J)$ with these constants. We call such $J$ a \emph{blending interval}. 
\end{remark}

The next observation is an immediate consequence of the compactness of $\bS^1$.

\begin{claim}[{\cite[Remark 2.1 and Lemma 2.2]{DiaGelRam:17}}]\label{nclalem:connect}
	Assume Axioms T, CEC$\pm(J)$ and Acc$\pm(J)$ are satisfied for some closed interval $J$. Then for every closed subinterval $I$ of $J$ there exists $m_{\rm c}=m_{\rm c}(I)\in\bN$ such that for every $x\in\bS^1$ there are finite sequences $(\theta_1,\ldots,\theta_r)$ and $(\beta_1,\ldots,\beta_s)$ with $r,s\le m_{\rm c}$ such that
\[
	(f_{\beta_s}\circ\cdots f_{\beta_1})(x)\in I
	\quad\text{ and }\quad
	(f_{\theta_r}^{-1}\circ\cdots\circ f_{\theta_1}^{-1})(x)\in I.
\]	
\end{claim}

\begin{definition}[The constant $L_1$]\label{defrem:commonblending}
Given $F\in \mathrm{SP}^1_{\rm shyp}(\Sigma_N\times\bS^1)$ and a blending interval $J=[x-2\delta,x+2\delta]$ with associated constants $K_1,\ldots, K_5$, define
\[
	L_1
	= L_1(F,J)
	\eqdef K_2\big(2+\lvert\log(4\delta)\rvert +K_3\big)+m_{\rm c}([x-\delta,x+\delta]).
\]
\end{definition}

\subsection{Key results}\label{subsec:res}

We are now ready to state the key result towards the proof of Theorem \ref{teo:1}. First note that by \cite[Theorem A and Lemma 5.2]{DiaGelRam:19} we have $\cL(0)\ne\emptyset$ and there holds the inequality
\begin{equation}\label{eq:resvarpri0}
	\sup\{h(F,\mu)\colon \mu\in\cM_{\rm erg}(F),\chi(\mu)=0\}
	\le h_{\rm top}(F,\cL(0)). 
\end{equation}
By \cite[Theorem A]{DiaGelRam:19}, there are numbers $\alpha_{\rm min}<0<\alpha_{\rm max}$ such that $\cL(\alpha)\ne\emptyset$ if and only if $\alpha\in [\alpha_{\rm min},\alpha_{\rm max}]$. Moreover, the map $\alpha\mapsto h_{\rm top}(F,\cL(\alpha))$ is continuous on the interval $[\alpha_{\rm min},\alpha_{\rm max}]$. Finally, for every $\varepsilon>0$ there exists some ergodic measure $\mu$ with negative Lyapunov exponent satisfying $\alpha=\chi(\mu)\in(-\varepsilon,0)$ and $h(F,\mu)>h_{\rm top}(F,\cL(0))-\varepsilon$; analogously for $\alpha\in(0,\varepsilon)$. With these results at hand, Theorem \ref{teo:1} is now an immediate consequence of the following.

\begin{mtheo}[Transfer of entropy to nonhyperbolic measures]\label{teo:2}
	For every $F\in \mathrm{SP}^1_{\rm shyp}(\Sigma_N\times\bS^1)$, $N\ge2$, there is some constant $L_1=L_1(F)>0$ such that for every $F$-invariant ergodic measure $\mu$ with negative Lyapunov exponent $\alpha\eqdef\chi(\mu)<0$ and positive entropy $h(F,\mu)$ and every $\varepsilon_H\in(0,h(F,\mu))$ there is a sequence of ergodic measures $(\mu_n)_n$ with negative Lyapunov exponents which converges weak$\ast$ to an ergodic measure $\mu_\infty$ satisfying
\[
	\chi(\mu_\infty)=0
	\spac{and}
	h(F,\mu_\infty)
	\ge e^{-L_1\lvert\alpha\rvert}(h(F,\mu)-\varepsilon_H).
\]	
\end{mtheo}

The proof of Theorem \ref{teo:2} is given in Sections \ref{sec:proooof} and \ref{sec:proooofb}. A byproduct of our construction is the following fact, which we prove at the end of Section \ref{sec:ncashor}.

\begin{mpro}\label{pro:main}
	Let $F\in \mathrm{SP}^1_{\rm shyp}(\Sigma_N\times\bS^1)$, $N\ge2$, and  $\mu$ be an $F$-invariant ergodic measure with negative Lyapunov exponent  and positive entropy $h(F,\mu)$. Then for every $\varepsilon_H\in(0,h(F,\mu))$ there are sequences of compact $F$-invariant sets $\Gamma_n\subset\Sigma_N\times\bS^1$ and numbers $\alpha_n\le\beta_n<0$ with
\[
	\lim_{n\to\infty}\alpha_n=\lim_{n\to\infty}\beta_n
	=0
\]
having the following properties: for every $n\in\bN$
\begin{itemize}
\item the set $\Gamma_n$ has uniform fiber contraction in the sense that $\alpha_n\le\chi(\tilde\mu)\le\beta_n<0$ for every $\tilde\mu\in\cM_{\rm erg}(F|_{\Gamma_n})$,
\item natural projection of $\Gamma_n$ to $\Sigma_N$ is a coded subshift and for every $\xi$ in this projection the fiber $\big(\{\xi\}\times\bS^1\big)\cap\Gamma_n$ is a finite set,
\item it holds
\[
	\limsup_{n\to\infty}h_{\rm top}(F,\Gamma_n)
	\ge e^{-L_1\lvert\alpha\rvert}(h(F,\mu)-\varepsilon_H).
\]	
\end{itemize}
\end{mpro}

\subsection{Consequences for elliptic cocycles}\label{sec:cocyc}
The space of cocycles  $\mathrm{SL}(2,\bR)^N$  roughly splits into the disjoint union of the sets of hyperbolic and elliptic cocycles: these sets are open and their union is dense in $\mathrm{SL}(2,\bR)^N$, see \cite[Proposition 6]{Yoc:04}. The set of hyperbolic cocycles, including the description of its boundary, is quite well understood, see \cite{AviBocYoc:10}. However, much less is known about the elliptic cocycles $\mathfrak{E}_N$. In \cite[Section 11]{DiaGelRam:19} it is introduced an open and dense subset $\mathfrak{E}_{N,\rm shyp}$ of $\mathfrak{E}_N$, the set of  \emph{elliptic cocycles having some hyperbolicity}. For our purposes, the key property of the set  $\mathfrak{E}_{N,\rm shyp}$ is that it consists of cocycles $\mathbf{A}$ whose associated skew products $F_{\mathbf{A}}$ with fiber maps defined as in \eqref{neq:defsteskecoc} are contained in  $\mathrm{SP}^1_{\rm shyp}(\Sigma_N\times\bP^1)$.
 
Instead of giving the precise definition of $\mathfrak{E}_{N,\rm shyp}$, let us describe its essential properties.  First, recall that an element $A\in \mathrm{SL}(2,\bR)$ is \emph{hyperbolic} if the absolute value of its trace is larger than $2$, which means that the matrix $A$ has one eigenvalue with absolute value bigger than one and one smaller than one. In particular, the union of the disjoint open sets 
\[
	B_A^-\eqdef\{v\in\bP^1\colon \lvert f_A'(v)\rvert <1\}
	\spac{and}
	B_A^+\eqdef\{v\in\bP^1\colon \lvert f_A'(v)\rvert>1\}
\]	 
is dense in $\bP^1$. The  set $\mathfrak{E}_{N,\rm shyp}$ 
consists of cocycles $\mathbf{A}\in\mathfrak{E}_{N}$ such that:
\begin{itemize}
\item Some hyperbolicity:
The semi-group generated by $\mathbf A$ contains a hyperbolic element.
\item Transitions in finite time: There is $M\ge1$ such that for every $v\in\bP^1$ there are sequences $\theta^+,\beta^+\in\Sigma_N^+$ such that $f_{\theta_{s-1}}\circ\cdots\circ f_{\theta_0}(v)\in B_A^+$ and $f_{\beta_{r-1}}\circ \cdots\circ f_{\beta_0}(v)\in B_A^-$ for some $s,r\le M$.
\end{itemize}
These properties are just the translation of the properties of maps in $\mathrm{SP}^1_{\rm shyp}(\Sigma_N\times\bP^1)$ to skew products arising from cocycles.   They also immediately refer precisely to the  context considered in \cite{GorIlyKleNal:05}. The set  $\mathfrak{E}_{N,\rm shyp}$ is open and dense in $\mathfrak{E}_N$, see \cite[Proposition 11.23]{DiaGelRam:19}. 

Let us now recall some results relating the (upper) Lyapunov exponent of a cocycle $\bA$ with the (fiber) Lyapunov exponent of its associated skew product $F_\bA$ on $\Sigma_N\times\bP^1$.
Let us consider the (forward) Lyapunov exponent
\[
	\chi^+(\xi^+,v)
	\eqdef \lim_{n\to\infty}\frac1n\log\,\lvert(f_{\xi^+}^n)'(v)\rvert,
\]
using the notation analogous to \eqref{eq:reference}, provided this limit exists. For the next results recall the definitions of $\lambda_1(\bA,\xi^+)$ in \eqref{eq:upper} and $\cL_{\mathbf A}^+(0)$ in \eqref{eq:orquic}. 

\begin{lemma}[{\cite[Theorem 11.1 and Claim 11.21 for $\alpha=0$]{DiaGelRam:19}}]\label{lem:blala}
	For every $\bA\in\mathrm{SL}(2,\bR)^N$ 
\[
	h_{\rm top}(\sigma^+,\cL^+_\bA(0))
	= h_{\rm top}(F_\bA,\cL(0)) .
\]	
\end{lemma} 

Given $\xi^+\in\Sigma_N^+$ and $\ell\in\bN$, denote by $v_+(\xi^+,\ell)\in\bP^1$ a vector  at which $\lvert (f_{\xi^+}^\ell)'\rvert$ attains its maximum; note that this vector is unique unless $f_{\xi^+}^\ell$ is an isometry. 

\begin{lemma}[{\cite[Proposition 11.5]{DiaGelRam:19}}]\label{lem:querende}
	Assume $\xi^+\in\Sigma_N^+$ satisfies $\lambda_1(\bA,\xi^+)=\alpha$.
\begin{itemize}
\item[(1)] If $\alpha=0$, then $\chi^+(\xi^+,v)=0$ for all $v\in\bP^1$.
\item[(2)] If $\alpha>0$, then the limit $v_0(\xi^+)=\lim_{\ell\to\infty}v_+(\xi^+,\ell)$ exists and it holds
\[
	\chi^+(\xi^+,v)
	=\begin{cases}
		2\alpha&\text{ for }v= v_0(\xi^+),\\
		-2\alpha&\text{ otherwise}.
	\end{cases}
\]
\end{itemize}
\end{lemma}

\begin{proof}[Proof of Theorem \ref{teo:3}]
	By Lemma \ref{lem:blala}, it suffices to study the metric entropy of  measures  $\nu^+\in\cM_{\rm erg}(\sigma^+)$ satisfying $\lambda_1(\mathbf A,\nu^+)=0$.
Consider the projections  $\pi^+\colon\Sigma_N\to\Sigma_N^+$,  $\pi^+(\xi^-.\xi^+)=\xi^+$, and  $\pi_1\colon\Sigma_N\times\bP^1\to\Sigma_N$, $\pi_1(\xi,x)=\xi$. 

\begin{claim}
	Given $\mu\in\cM_{\rm erg}(F_\bA)$, let $\nu^+\eqdef(\pi^+\circ\pi_1)_\ast\mu$.
	If $\chi(\mu)=0$ then $\lambda_1(\bA,\nu^+)=0$.
\end{claim}

\begin{proof}
By ergodicity, $\chi(\xi,v)=0$ for $\mu$-almost every  $(\xi,v)$. 
Denote by $\mu^+$ the ergodic measure obtained as the push-forward of $\mu$ by the map $(\xi,v)\mapsto(\xi^+,v)$.
Hence for $\mu^+$-almost every $(\xi^+,v)$ it holds $\chi^+(\xi^+,v)=0$. It follows from Lemma \ref{lem:querende} (1) that $\lambda_1(\bA,\xi^+)=0$ for $\nu^+$-almost every $\xi^+$. Note that $\nu^+$ is ergodic. Hence, by the subadditive ergodic theorem, the claim follows. 
\end{proof}

\begin{claim}
	For every $\nu^+\in\cM_{\rm erg}(\sigma^+)$ with $\lambda_1(\bA,\nu^+)=0$ there exists $\mu\in\cM_{\rm erg}(F_\bA)$ satisfying $\chi(\mu)=0$.
\end{claim}

\begin{proof}
Given $\nu^+\in\cM_{\rm erg}(\sigma^+)$ there exists $\mu\in\cM_{\rm erg}(F_\bA)$ such that $\nu^+=(\pi^+\circ\pi_1)_\ast\mu$. It follows from  Lemma \ref{lem:querende} (1)  that $\chi^+(\xi^+,v)=0$ for $\nu^+$-almost every $\xi^+$ and any $v$. Hence, $\chi(\xi,v)=0$ for $\mu$-almost every $(\xi,v)$, which implies $\chi(\mu)=0$. 
\end{proof}

It follows%
\footnote{
It holds
\[
	\sup_{\mu\colon\mu\circ(\pi^+\circ\pi_1)^{-1}=\nu^+}h(F_\bA,\mu)
	=h(\sigma^+,\nu^+)+\int_{\Sigma_N^+}h_{\rm top}(F_\bA,(\pi^+\circ\pi_1)^{-1}(\xi^+))\,d\nu^+(\xi^+).
\]
It is straightforward to check that $h_{\rm top}(F_\bA,(\pi^+\circ\pi_1)^{-1}(\xi^+))=0$ for every $\xi^+$.}  from \cite{LedWal:77} that $h(\sigma^+,\nu^+)=h(F_\bA,\mu)$. 
Hence
\[\begin{split}
	\sup\{h(F_\bA,\mu)&\colon \mu\in\cM_{\rm erg}(F_\bA)
		,\chi(\mu)=0\}\\
	&=	
	\sup\{h(\sigma^+,\nu^+)\colon \nu^+\in\cM_{\rm erg}(\sigma^+)
		,\lambda_1(\bA,\nu^+)=0\}.
\end{split}\] 
Hence, the equality follows from Theorem \ref{teo:1}. It remains to show that the topological entropy of this level set is less than $\log N$ and positive. Indeed, this was shown in \cite[Theorem B]{DiaGelRam:19}. This proves the theorem.
\end{proof}

\section{Coded systems}\label{sec:notation}
\subsection{Preliminaries and decipherability}

Throughout this section, let $\cA$ be a finite collection of \emph{symbols}, also called an \emph{alphabet}. 
The \emph{full shift} over $\cA$, denoted by $\cA^\bZ$, is the collection of all bi-infinite sequences of symbols from $\cA$. We write an element of this space as $\underline a=(a_k)_{k\in\bZ}=(\ldots ,a_{-1}|a_0,a_1,\ldots)$, where $a_k\in\cA$ for all $k\in\bZ$. Equipped with the metric $d_1(\underline a,\underline b)\eqdef e^{-\inf\{\lvert k\rvert\colon a_k\ne b_k\}}$ this space is compact. Given $\underline a\in\cA^\bZ$ and $n\in\bN$, we denote by 
\[
	[a_0,\ldots,a_{n-1}]
	\eqdef \{\underline b\in\cA^\bZ\colon b_k= a_k\text{ for all }k=0,\ldots,n-1\}
\] 
its \emph{$n$-cylinder} or simply \emph{cylinder} associated to $(a_0,\ldots,a_{n-1})$. 

A \emph{word} over $\cA$ is a finite sequence of symbols in $\cA$. The \emph{length} of a word $a$ is the number of symbols it contains and is denoted by $\lvert a\rvert_\cA$. A \emph{$k$-word} is a word of length $k$ and the set of all $k$-words over $\cA$ is denoted by $\cA^k$. The \emph{empty word} $\epsilon$ is the unique word with no symbols and of length zero. The set of all  words (including the empty word) over the alphabet $\cA$ is denoted by $\cA^\ast$.
A \emph{concatenation} of a pair of words $a=(a_1,\ldots ,a_k)$ and  $b=(b_1,\ldots, b_\ell)$ is the word $(a,b)\eqdef (a_1,\ldots,a_k,b_1,\ldots ,b_\ell)$. The concatenation of any number of words is analogous. A \emph{prefix} of a word $b\in\cA^\ast$ is a word $a\in\cA^\ast$ such that $b=(a,c)$ for some $c\in\cA^\ast$.  A \emph{suffix} of a word $b\in\cA^\ast$ is a word $a\in\cA^\ast$ such that $b=(c,a)$ for some $c\in\cA^\ast$. 

The \emph{shift} $\sigma_\cA\colon \cA^\bZ\to\cA^\bZ$, defined by $(\sigma_\cA(\underline a))_k=a_{k+1}$ for every $k\in\bZ$, is a continuous map. A \emph{subshift} is a closed subset of $\cA^\bZ$ which is $\sigma_\cA$-invariant. Replacing $\bZ$ by $\bN_0$ and by $-\bN$ one obtains one-sided shift spaces.

A \emph{coded system} (\emph{CS}) over the alphabet $\cA$ is a compact subshift $S\subset\cA^\bZ$ such that there is a
  collection $\cB\subset\cA^\ast$ of words over $\cA$ such that  $S$ is the shift invariant closure of all bi-infinite concatenations of words in $\cB$. 
  Here we always assume that $\cB$ is finite%
\footnote{The general definition of a \emph{coded system} allows $\cB$ to be infinite. However, this will not be needed here.}.
  Any such collection $\cB$ is called a \emph{code} for $S$. Let us be a bit more precise in our notation. Each code $\cB\subset\cA^\ast$ by itself can be considered as an alphabet giving rise to the space $\cB^\bZ$. Consider the map $\hat\iota_\cB\colon\cB\to\cA^\ast$ which sends each element $w\in\cB$ to the corresponding word in $\cA^\ast$ and let
\begin{equation}\label{eq:iotaW1}
	\iota_\cB\colon\cB^\ast\to\cA^\ast,\quad
	\iota_\cB(w_1,\ldots,w_k)
	\eqdef \big(\hat\iota_\cB(w_1),\ldots,\hat\iota_\cB(w_k)\big).
\end{equation}
Consider the extension of this map to the space of bi-infinite concatenations
\begin{equation}\label{eq:iotaW2}\begin{split}
	&\underline\iota_\cB\colon\cB^\bZ\to\cA^\bZ,\\
	&\underline\iota_\cB(\ldots,w_{-1}|w_0,w_1,\ldots)
	\eqdef (\ldots,\iota_\cB(w_{-1})|\iota_\cB(w_0),\iota_\cB(w_1),\ldots),
\end{split}\end{equation}
which identifies the bi-infinite concatenation $(\ldots,w_{-1}|w_0,w_1,\ldots)\in\cB^\bZ$ of elements in $\cB$ (that is, words over $\cA$) with the corresponding  sequence in $\cA^\bZ$. We denote by
\begin{equation}\label{eq:Pnumber}
	\PCS{\cB}
	\eqdef \underline\iota_\cB\big(\cB^\bZ\big)
	\subset \cA^\bZ	
\end{equation}
the \emph{pre-coded system} defined by $\cB$. As we assume $\cB$ to be finite,  
\begin{equation}\label{eq:number}
	\CS{\cB}
	\eqdef \bigcup_{k\in\bZ}\sigma_\cA^k\big(\PCS{\cB}\big)
	= \big\{(\sigma_\cA^k\circ\underline\iota_\cB)(\underline w)
		\colon \underline w\in\cB^\bZ,k\in\bZ\big\}
\end{equation}
is a coded system (that is, this set is, in particular, compact and shift-invariant).
Note that every coded system is transitive. We refer to \cite[Chapter 13.5]{LinMar:95} for more details and references on coded systems.

Let $\cB$ be a finite collection  of (nonempty) words over the alphabet $\cA$. We say that $\cB$ is \emph{disjoint} if no element in $\cB$ is a prefix of another element in $\cB$. This term is justified by the fact that for any disjoint collection of (nonempty) words their associated cylinders are pairwise disjoint. We say that $\cB$ is \emph{uniquely left decipherable} if whenever $\iota_\cB(w_1,\ldots, w_m)$ is a prefix of $\iota_\cB(u_1,\ldots ,u_n)$ where $w_i,u_j\in\cB$, then $m\le n$ and $w_i=u_i$ for every $i=1,\ldots,m$.%
\footnote{Note that this is slightly weaker than being \emph{uniquely decipherable} as  in \cite[Definition 8.1.21]{LinMar:95}.}

\begin{lemma} \label{lem:COND}
	Every disjoint finite collection of words  is uniquely left decipherable. 
\end{lemma}

\begin{proof}
Let $\cB$ be a disjoint collection of words over  the alphabet $\cA$.
	We proceed by induction over the length of a prefix of a word over the alphabet $\cB$. As $\cB$ by hypothesis is disjoint, the statement is true for $k=1$. Assuming the statement is true for $k\in \bN$, suppose that $\iota_\cB(w_{i_1}, w_{i_2},\ldots, w_{i_k},w_{i_{k+1}})$ is a prefix of $\iota_\cB(u_{j_1},u_{j_2},\ldots, u_{j_m})$ for some $m\in\bN$. Again invoking our hypothesis that $\cB$ is disjoint, it follows $w_{i_1}=u_{j_1}$. Hence $\iota_\cB(w_{i_2},\ldots ,w_{i_k},w_{i_{k+1}})$ is a prefix of $\iota_\cB(u_{j_2},\ldots, u_{j_m})$. Then the induction hypothesis implies $w_{i_n}=u_{j_n}$ for every $n=2,\ldots,k+1$, proving the lemma.
\end{proof}

In what follows we assume that $\cB=\{w_1,\ldots,w_M\}$. Given $\underline a\in\PCS{\cB}$ a bi-infinite concatenation of words in $\cB$,  a \emph{decoding} of  $\underline a$ in $\cB$ is a sequence $(\ldots,i_{-1}|i_0,i_1,\ldots)\in\{1,\ldots,M\}^\bZ$ such that
\[
	\underline a
	= \iota_\cB (\ldots ,w_{i_{-1}}|w_{i_0},w_{i_1},\ldots),
	\spac{that is}
	\underline a
	\in\iota_\cB^{-1}(\underline w).
\]
We call then the one-sided sequences $(\ldots,i_{-1})\in\{1,\ldots,M\}^{-\bN}$ and $(i_0,i_1,\ldots)\in\{1,\ldots,M\}^{\bN_0}$ a \emph{backward} and  \emph{forward decoding} of $\underline a$, respectively. 

Let us denote by 
\[
	\underline\iota_\cB^+\colon \cB^{\bN_0}\to \cA^{\bN_0}
	\spac{and}
	\underline\iota_\cB^-\colon\cB^{-\bN}\to\cA^{-\bN}
\]	 
the maps analogously defined to the one in \eqref{eq:iotaW2}, then a code $\cB$ is uniquely left decipherable if, and only if, $\underline\iota_\cB^+$ is invertible.
Lemma \ref{lem:COND} hence implies immediately the following. 

\begin{corollary}\label{cor:forward}
	Let $\cB$ be a disjoint finite collection of words over $\cA$. Then every element in $\PCS{\cB}$ has a unique  forward decoding in $\cB$ and hence $\underline\iota_\cB^+$ is invertible.
\end{corollary}

\begin{remark}{\rm
In general, a (even disjoint) collection of words defines sequences which cannot be uniquely decoded. 
For example, for the alphabet $\cA=\{a,b\}$  the collection $\cB=\{(a,b,a)$, $(a,b,b),(b,a),(b,b,a)\}$ is disjoint, but 
\[
	\big((a,b,b)^{-\bN},a,b,a|\ldots\big)
	= \big((b,b,a)^{-\bN},b,a|\ldots\big)
	\in\PCS{\cB}\subset\cA^\bZ
\]
can be written in two different ways as concatenation of words in $\cB$. Nevertheless, the following fundamental fact holds. For completeness, we provide its proof.
}\end{remark}

\begin{lemma}\label{lem:soandso}
	Let $\cB$ be a disjoint finite collection of words over $\cA$. Then every element in $\PCS{\cB}$ has at most $R$ decodings in $\cB$, where $R$ is the largest length of a word in $\cB$.
\end{lemma}

\begin{proof}
	By Corollary \ref{cor:forward} it suffices to study backward decodings of a sequence.
	By contradiction, suppose that there exists $\underline a\in\PCS{\cB}$ which has $R+1$ backward decodings. We consider a family $\mathcal I\subset\{1,\ldots,M\}^{-\bN}$ of decodings of $\underline a$ having $R+1$ elements. 
Given two such backward decodings $i^-=(\ldots,i_{-1}),j^-=(\ldots,j_{-1})\in\mathcal I$, for $k\in\bN$ let 
\[
	I_k
	\eqdef \lvert \iota_\cB(w_{i_{-k}},\ldots, w_{i_{-1}})\rvert_\cA
	\spac{and}
	J_k
	\eqdef \lvert \iota_\cB(w_{j_{-k}},\ldots, w_{j_{-1}})\rvert_\cA.	
\]	
If there is $k\in\bN$ such that there exists an index $\ell\in\bN$ with $I_k=J_\ell$, then corresponding parts of the decodings coincide as concatenated words over the alphabet $\cA$, that is,
\[
	\iota_\cB(w_{i_{-k}},\ldots ,w_{i_{-1}})
	= \iota_\cB(w_{j_{-\ell}},\ldots ,w_{j_{-1}}).
\] 
As we assume that $\cB$ is disjoint, by Lemma \ref{lem:COND} it is uniquely left decipherable. Hence, it follows $k=\ell$ and $i_{-n}=j_{-n}$ for every $n=k,\ldots,1$. If there are infinitely many such indices $k\in\bN$, then it follows $i^-=j^-$, contradicting the fact that we consider two distinct backward decodings. Hence, there exists one largest such index $k_0=k_0(i^-,j^-)$. 
Let 
\[
	N
	\eqdef \max_{i^-,j^-\in\mathcal I}k_0(i^-,j^-)
\]
and note that $N$ is finite as $\mathcal I$ has $R+1$ elements.
Hence, by uniquely left decipherability,  this implies that 
\begin{equation}\label{eq:stepproof}
	\big\{\iota_\cB(w_{i_{-N}},\ldots,w_{i_{-1}})
	\colon i^-\in\mathcal I,w_{-\ell}\in\cB,\ell=N,\ldots,1\big\}
\end{equation}
is a collection of distinct words.
Because every word in $\cB$ has a length bounded by $R$, for every $i^-\in\mathcal I$ there exists some index $n\ge N$ such that
\[
	-NR-R
	\le -\sum_{\ell=1}^{n}\lvert \iota_\cB(w_{i_{-\ell}})\rvert_\cA
	\le -NR-1.
\]
In other terms, the marker of the starting position of the $n$th word in this decoding is between  position $-NR-R$ and position $-NR-1$ of the ``spelling'' of $\underline a$ in the alphabet $\cA$. As by assumption we have $R+1$ such decodings, by the pigeonhole principle (at least) two of them must end at the same position. Say, there are $i^-,j^-\in\mathcal I$ and indices $k$ and $\ell$, respectively,  such that
\[
	-\lvert \iota_\cB(w_{i_{-k}})\rvert_\cA-\ldots-\lvert \iota_\cB(w_{i_{-1}})\rvert_\cA
	=
	-\lvert \iota_\cB(w_{j_{-\ell}})\rvert_\cA-\ldots-\lvert \iota_\cB(w_{j_{-1}})\rvert_\cA.
\]
As in the argument before, uniquely left decipherability implies
\[
	\iota_\cB(w_{i_{-k}},\ldots,w_{i_{-1}})
	=\iota_\cB(w_{j_{-\ell}},\ldots,w_{j_{-1}})
\]
and hence $k=\ell$ and $i_n=j_n$ for every $n=k,\ldots,1$. In particular, the latter implies $i_{-n}=j_{-n}$ for every $n=N,\ldots,1$. This contradicts the fact that all words in \eqref{eq:stepproof} are distinct. This proves the lemma.
\end{proof}

The following facts are immediate consequences of the definition of disjointness.

\begin{corollary}\label{cor:dis2}
	Let $\cB$ be a disjoint collection of words over the alphabet $\cA$ and $m\in\bN$. Then $\cB^m$ is a disjoint collection of words over $\cA$.
\end{corollary}

\begin{corollary}\label{corlem:dis3}
	Let $\cB'=\{w_1',\ldots,w_M'\}$ and $\cB=\{w_1,\ldots,w_M\}$ be two collections of words over the alphabet $\cA$ such that $w_i$ is a prefix of $w_i'$ for every $i=1,\ldots,M$. If $\cB$ is disjoint then $\cB'$ is disjoint.
\end{corollary}

\subsection{Coded subsystems of the sequence space $\Sigma_N$}\label{ssec:codedSigma}

The base space $\Sigma_N=\cA^\bZ$ of the skew product \eqref{eq:sp} with $\cA=\{1,\ldots,N\}$ is a special case of the above. 
 
\begin{notation}\label{eq:notation}{\rm
In Sections \ref{sec:horses}--\ref{sec:proooof}, our base alphabet will always be $\cA=\{1,\ldots,N\}$. In Sections \ref{se:suspmode}--\ref{sec:cascade-1}, the alphabet $\cA$ is unspecified. We are also going to use two families of other alphabets, $(\cA_n)_n$ and $(\cB_n)_n$, defined by some finite families of words from $\cA^\ast$.
All of those alphabets are going to be disjoint. For better readability, we will identify the finite words in alphabets $\cA_n$ and $\cB_n$ with the corresponding finite words in $\cA$. Note that, because of disjointness we have decipherability of any \emph{finite} word, hence this convention is not going to lead to any ambiguity. For infinite or bi-infinite words in those alphabets we will use the precise notation, to always keep track whether we are talking about an element of $\cB_n^\bZ$ or an element of $\cA^\bZ$:
Note that $\PCS{\cB}$ is a subset of $\cA^\ast$. Given  $w_1,\ldots,w_m\in\cB$, we denote by 
\[
	\lvert(w_1,\ldots,w_m)\rvert
	\eqdef\lvert (\iota_\cB(w_1),\ldots,\iota_\cB(w_m))\rvert_\cA
\]	 
the length of the concatenated word (spelled in the alphabet $\{1,\ldots,N\}$). We let
\begin{equation}\label{eq:specific}
	[w]^+
	\eqdef [\iota_\cB(w)]^+
\end{equation}
denote the cylinder in $\Sigma_N^+$. Given $u\in\Sigma_N^\ast$ we let
\[
	(w_1,\ldots,w_m,u)
	\eqdef (\iota_\cB(w_1,\ldots,w_m),u)
	\in\Sigma_N^\ast
\]
denote the corresponding concatenated word (spelled in the alphabet $\{1,\ldots,N\}$). Finally, denote
\[
	\cB^m
	\eqdef \{(w_1,\ldots,w_m)\colon w_k\in\cB,k=1,\ldots,m\}
	\subset\cA^\ast.
\]
Analogously to notation \eqref{eq:reference}, for $k=1,\ldots,\lvert(w_1,\ldots,w_m)\rvert\eqdef L$, we denote by 
\[
	f_{[w_1,\ldots,w_m]}^k
	\eqdef f_{\xi_{k-1}}\circ\cdots\circ f_{\xi_0},\,\,\text{where}\,\,
	(\xi_0,\xi_1,\ldots,\xi_{L-1})\eqdef(\iota_\cB(w_1,\ldots,w_m))
		\in\Sigma_N^\ast,
\]
the map obtained by concatenating the maps from the family $\{f_i\}_i$ which are indexed by the first  $k$  elements  of the concatenated words $w_1,\ldots,w_m$ (spelled in $\{1,\ldots,N\}$). 	
}\end{notation}

\section{Suspensions of shift spaces}\label{se:suspmode}

We consider measure preserving systems obtained from suspensions of Bernoulli shifts, collect some standard facts (Section \ref{ssec:susmod}), and discuss large deviation results (Section \ref{ssec:LD}). 

Throughout this section, we fix a finite initial alphabet $\cA$. Later (starting in Section \ref{sec:horses}) we will assume that this alphabet is $\{1,\ldots,N\}$ and will also invoke the axioms in Section \ref{sec:axioms}, however these ingredients are irrelevant in this section.

\subsection{Suspension model in the full shift over $\cA$}\label{ssec:susmod}

Given a function $R\colon \cA\to\bN$, we extend it to a step function on the sequence space $\cA^\bZ$ by
\[
	\underline R \colon \cA^\bZ\to\bN,\quad
	\underline a=(\ldots ,a_{-1}|a_0,a_1,\ldots)\in\cA^\bZ
	\mapsto \underline R(\underline a)=R(a_0).
\]	
We define the \emph{discrete-time suspension space}
\[
	\fS_{\cA,R}
	\eqdef (\cA^\bZ\times\bN_0)_{\sim},
\]
defined as the quotient space of $\cA^\bZ\times\bN_0$ modulo the equivalence relation $\sim$ that identifies $(\underline a,s)$ with $(\sigma_\cA(\underline a),s-\underline R(\underline a))$ for every $s\ge \underline R(\underline a)$ and $\underline a\in\cA^\bZ$. For convenience, we represent each class by its element $(\underline a,s)$ with $s\in\{0,\ldots,\underline R(\underline a)-1\}$, called its \emph{canonical representation}. 
 
The \emph{suspension of $\sigma_\cA$ by $\underline R$} is the map 
\[
	\Phi_{\cA,R}\colon \fS_{\cA,R}\to \fS_{\cA,R} 
	,\quad
	\Phi_{\cA,R}(\underline a,s)
	\eqdef 
	\begin{cases}
	(\underline a,s+1) &\text{ if }0\le s+1<\underline R(\underline a),\\
	(\sigma_\cA(\underline a),0) &\text{ if }s=\underline R(\underline a)-1.
	\end{cases}
\]
We also consider the \emph{ground floor}
\[
	\fG_{\cA,R}
	\eqdef\cA^\bZ\times\{0\}
	\subset \fS_{\cA,R}.
\]

\begin{remark}[(Piecewise constant) roof function]
	We can view the map $\underline R$ as a \emph{roof function} over $\cA^\bZ$. By our choice, this function is piecewise constant on each cylinder determined by a symbol of $\cA$, that is, for every $a_0\in\cA$ it holds
\[
	\underline R(\underline a)
	= R(a_0)
	\spac{for every}
	\underline a\in[a_0].
\]
\end{remark}

Let $M\eqdef \card\cA$, $\fb_\cA$ be the $(\frac{1}{M},\ldots,\frac{1}{M})$-Bernoulli measure on $\cA^\bZ$, and $\mathfrak m$ be the counting measure on $\bZ$. Define  the measure $\lambda_\cA$ on $\fS_{\cA,R}$ by  
\begin{equation}\label{eq:defmeasure}
	\lambda_{\cA,R}
	\eqdef \frac{1}{(\fb_\cA\times \mathfrak m)(\fS_{\cA,R})}
		(\fb_\cA\times \mathfrak m)|_{\fS_{\cA,R}}
	= \frac{(\fb_\cA\times \mathfrak m)|_{\fS_{\cA,R}}}
		{\int \underline R\,d\fb_\cA}	.
\end{equation}
Given a continuous function $\psi\colon\fS_{\cA,R}\to\bR$, define $\Delta\psi\colon\cA^\bZ\to\bR$ by
\begin{equation}\label{eq:defDeltapsi}
	\Delta\psi(\underline a)
	\eqdef \sum_{k=0}^{\underline R(\underline a)-1}\psi(\underline a,k).
\end{equation}

\begin{lemma}[Abramov's formula]\label{lem:Abramov}
The measure $\lambda_{\cA,R}$ is a $\Phi_{\cA,R}$-invariant and ergodic Borel probability measure satisfying
\[
	h_{\rm top}(\Phi_{\cA,R},\fS_{\cA,R})
	\ge h(\Phi_{\cA,R},\lambda_{\cA,R})
	= \frac{\log \card \cA}{\int \underline R\,d\fb_\cA}
	= \frac{\log M}{\frac1M\sum_{a\in\cA} R(a)}.
\]
 Moreover, for any continuous function $\psi\colon\fS_{\cA,R}\to\bR$ it holds
\[
	\int\psi\,d\lambda_{\cA,R}
	=\frac{\int\Delta\psi\,d\fb_\cA}{\int \underline R\,d\fb_\cA}.
\]
\end{lemma}

\subsection{Controlled large deviations}\label{ssec:LD}

We now study the fluctuation of  Birkhoff sums of repeated returns to the ground floor of the suspension space. For that, given a potential $\psi\colon\fS_{\cA,R}\to\bR$ and $\Delta\psi\colon\cA^\bZ\to\bR$ defined as above, let
\begin{equation}\label{eq:defAve}
	\var_\cA(\Delta\psi)
	\eqdef \max_{a\in\cA}
		\max_{\underline b,\underline c\in[ a]}
		\,\Big\{\big(\Delta\psi(\underline b)-\Delta\psi(\underline c)\big)\Big\} .
\end{equation}

\begin{proposition}\label{pro:LD}
	Let $\psi\colon\fS_{\cA,R}\to\bR$ be a continuous potential. 
	For every $\varepsilon>0$ there exists $N_0=N_0(\psi,\varepsilon)\in\bN$ such that if $m\ge N_0$ then there exists a set $A\in\cA^\bZ$ such that
$
	\fb_\cA(A)>1-\varepsilon
$	
and for every $\underline a\in A$, $i=0,\ldots,m-1$, and $k\in\{1,\ldots,m\}$ it holds
\[\begin{split}
	\Big\lvert\sum_{j=i}^{i+k-1}\left(\underline R(\sigma_\cA^j(\underline a))
		-\int \underline R\,d\fb_\cA\right)\Big\rvert
	&<m\varepsilon,\\
	\Big\lvert\sum_{j=i}^{i+k-1}\left(\Delta\psi(\sigma_\cA^j(\underline a))
		-\int \Delta\psi\,d\fb_\cA\right)\Big\rvert
	&<m(2\var_\cA(\Delta\psi)+\varepsilon).	
\end{split}\]
\end{proposition} 

\begin{proof}
We use the following probability result based on the Bernstein inequality. 

\begin{lemma}\label{lemcla:Bernstein}
Let $X$ be a bounded random variable with expected value $\bE(X)$. Then for every $\varepsilon>0$, there exists $N_0\in\bN$ such that for every $m\ge N_0$ the following holds. Let $X_1,\ldots, X_{2m}$ be independent and identically distributed copies of $X$. 
Then for every $k\in\{1,\ldots,m\}$ it holds
\[
	\bP\left(
		\big\lvert \sum_{j=i}^{i+k-1} \big(X_j - \bE (X)\big)\big\rvert < m\varepsilon
		\quad\text{ for every } i\in \{1,\ldots,m\} \right)
		 > 1- \varepsilon.
\]
\end{lemma}

\begin{proof}
Take $C>0$ such that $\lvert X-\bE(X)\rvert\le C$. By the Bernstein inequality, for every $m\in\bN$, $i\in\{0,\ldots,m-1\}$, and $k\in\{1,\ldots,m\}$ it holds
\[
	\bP\Big( \big\lvert\sum_{j=i}^{i+k-1} (X_j-\bE(X))\big\rvert < m\varepsilon \Big)
	> 1- 2e^{-3 m\varepsilon/2C}.
\] 
Hence, it follows
\[
	\bP\Big(\big\lvert\sum_{j=i}^{i+k-1} (X_j-\bE(X))\big\rvert < m\varepsilon
		\text{ for every }i\in\{1,\ldots,m\} \Big)
	>1- 2me^{-3 m\varepsilon/2C}.
\] 
The assertion follows taking $m$ large enough.
\end{proof}

To continue with the proof of the proposition, given $\cA=\{a_1,\ldots,a_M\}$,  consider the partition $\{[a_k]\colon k=1,\ldots, M\}$ of $\cA^\bZ$ into cylinders and let
\[
	B(a_k)
	\eqdef \max_{[a_k]}\Delta\psi.
\]
Let
\[
	\fB
	\eqdef \frac{1}{M}(B(a_1)+\cdots +B(a_{M}))
	,\quad
	\frakR
	\eqdef \frac{1}{M}(R(a_1)+\cdots+R(a_{M})).
\]

Consider the Bernoulli measure $\fb_\cA$ on $\cA^\bZ$ and the two random variables on this space
\[
	X_j(\underline a)=B({a_j})
	\quad\text{ and }\quad
	Y_j(\underline a)=R({a_j}),
\]	
which are bounded from above by $\max_kB(a_k)$ and $\max_kR(a_k)$, respectively. Moreover, they are independent and identically distributed and have expected values $\fB$ and $\frakR$, respectively. Given $\varepsilon>0$, applying Lemma \ref{lemcla:Bernstein} to each of those variables, there exists $N_0\in\bN$ such that for every $m\ge N_0$ there is a set $A$ satisfying $\fb_\cA(A)>1-\varepsilon$ such that for every $\underline a\in A$, $i=0,\ldots,m-1$, and  $k\in\{1,\ldots,m\}$ it holds
\begin{equation}\label{eq:usethis}
	 \Big\lvert
		\sum_{j=i}^{i+k-1}B(a_j)-k\fB
	\Big\rvert
	\le m\varepsilon,
	\qquad
	\Big\lvert
		\sum_{j=i}^{i+k-1}R(a_j)-k\frakR
	\Big\rvert
	\le m\varepsilon.
\end{equation}

The second inequality in \eqref{eq:usethis} and the fact that $\underline R$ is piecewise constant proves the first claim of the proposition as
$\frakR=\int \underline R\,d\fb_\cA$.

To prove the second claim, recall that by the very definition of $\var_\cA$ in \eqref{eq:defAve} it holds
\begin{equation}\label{eq:usethisb}
	\Big\lvert
		\Delta\psi(\sigma_\cA^j(\underline a))
		- B(a_j)
	\Big\rvert
	\le\var_\cA(\Delta\psi)
\end{equation}
for every $j\in\bZ$. For every $i=0,\ldots,m-1$, by summing over $j=i,\ldots,i+k-1$, it follows
\[\begin{split}
	\Big\lvert
		\sum_{j=i}^{i+k-1}\Delta\psi(\sigma_\cA^j(\underline a))
		-k\fB\Big\rvert
	&\le \Big\lvert
		\sum_{j=i}^{i+k-1}\big(\Delta\psi(\sigma_\cA^j(\underline a))
		- B(a_j)\big)\Big\rvert
		+\Big\lvert 
		 \sum_{j=i}^{i+k-1}B(a_j)
		 -k\fB
	\Big\rvert\\
	{\tiny{\text{use \eqref{eq:usethis},  \eqref{eq:usethisb}, and $k\le m$}}}\quad	
	&\le k\var_\cA(\Delta\psi)+m\varepsilon
	\le m\var_\cA(\Delta\psi)+m\varepsilon.
\end{split}\]
Also note that  
\[
	\lvert \fB - \int \Delta\psi\,d\fb_\cA\rvert
	\le \var_\cA(\Delta\psi),
\]
which together with $k\le m$ and the above  implies 
\[\begin{split}
	\Big\lvert
		\sum_{j=i}^{i+k-1}\Delta\psi(\sigma_\cA^j(\underline a))
		-k\int\Delta\psi\,d\fb_\cA \Big\rvert
	&\le
\Big\lvert
		\sum_{j=i}^{i+k-1}\Delta\psi(\sigma_\cA^j(\underline a))
		-k\fB\Big\rvert
		+m\var_\cA(\Delta\psi)\\
	&\le m \big(\var_\cA(\Delta\psi)+\varepsilon\big)
		+m\var_\cA(\Delta\psi)\\
	&= m(2\var_\cA(\Delta\psi)+\varepsilon)	.
\end{split}\]
This implies the second claim of the proposition.
\end{proof}

\section{Cascade of alphabets: Repetition and tailing}\label{sec:cascade-1}

Fix a finite alphabet $\cA$ and any increasing sequence of positive integers $(m_n)_{n\ge1}$. In Section \ref{sec:inddef}, we present an inductive construction of a cascade of alphabets $\cA_n$, each obtained by concatenating $m_n$ words of the former. Each such alphabet comes with a sequence space and corresponding suspension space as in Section \ref{se:suspmode}. The inductive definition gives some ``self-similar'' structure in the sense that each space ``contains copies'' of the precedents. Section \ref{sec:susmodint} studies this structure and develops some terminology. As our measure preserving suspensions are defined by means of Bernoulli measures, all their ``copy measures'' coincide, see Section \ref{ssec:bermea}.   
In Section \ref{sec:rooffunctions} we collect some relations between the roof functions across the cascade.
\subsection{Inductive definition of alphabets}\label{sec:inddef}

Let $\cA_0\eqdef\cA$, $M_0\eqdef\card\cA_0=M$.
For $n\ge1$, assume that we have given the collection $\cA_{n-1}$ of $M_{n-1}$ finite words over $\cA$. Let
\[
	\cA_n
	\eqdef (\cA_{n-1})^{m_n}
	\subset\cA^\ast.
\]
 Note that 
\[
	\card\cA_n
	= M_n
	\eqdef  (M_{n-1})^{m_n}.
\]
This concludes the inductive definition of the cascade of alphabets $\cA_n$. 

Note that $\cA_n$ is a finite collection of words over the alphabet $\cA$: the collection of all $(m_1\cdot m_2\cdots m_n)$-words over $\cA$. Indeed, each word is obtained by a concatenation of $m_n$-words in the alphabet $\cA_{n-1}$. On one hand, by convention, each concatenation of words over $\cA$ again is a word over $\cA$. On the other hand, the collection $\cA_n$ by itself can serve as an alphabet, each of its elements being a symbol in $\cA_n$. 
In each step, we have the corresponding canonical bijection between the symbols in $\cA_n$ and the $m_n$-words over the alphabet $\cA_{n-1}$. This defines a \emph{substitution map} from $\cA_n$ to $\cA_{n-1}$:
\[
	\Sub_{n,n-1}
	\colon\cA_n\to(\cA_{n-1})^{m_n}
\]
We extend this map to the bijection between the corresponding spaces of sequences
\begin{equation}\label{eq:tailaddn-1n}\begin{split}
	&\underline \Sub_{n,n-1}
	\colon (\cA_n)^\bZ\to\big((\cA_{n-1})^{m_n}\big)^\bZ
	=(\cA_{n-1})^\bZ,
	\quad\\
	&\underline\Sub_{n,n-1}(\ldots | a_0^{(n)},a_1^{(n)}, \ldots)
	\eqdef (\ldots | \Sub_{n,n-1}(a^{(n)}_0) ,\Sub_{n,n-1}(a^{(n)}_1),\ldots).
\end{split}\end{equation}

Denote by $\sigma_n=\sigma_{\cA_n}\colon(\cA_n)^\bZ\to(\cA_n)^\bZ$ the left shift over $\cA_n$. Let $\fb_n=\fb_{\cA_n}$ be the $(\frac{1}{M_n},\ldots,\frac{1}{M_n})$-Bernoulli measure on $(\cA_n)^\bN$.

\begin{remark}\label{rem:isomorphisms}
Note that $\underline\Sub_{n,n-1}$ defined in \eqref{eq:tailaddn-1n} is a bijection. Moreover 
$\sigma_n\colon(\cA_n)^\bZ\to(\cA_n)^\bZ$ is topologically conjugate with $\sigma_{n-1}^{m_n}\colon(\cA_{n-1})^\bZ\to(\cA_{n-1})^\bZ$ by $\underline\Sub_{n,n-1}$. Observe that 
\begin{equation}\label{eq:isedbelow}
	(\underline\Sub_{n,n-1})_\ast\fb_n=\fb_{n-1}.
\end{equation}	 
Thus, 
\[
	\underline\Sub_{n,n-1}\colon
	\big((\cA_n)^\bZ,\sigma_{n},\fb_n\big)\to
	\big((\cA_{n-1})^\bZ,\sigma_{{n-1}}^{m_n},\fb_{n-1}\big)
\]  
is a metric isomorphism. Hence, applying the previous argument inductively, we obtain that for every $n\in\bN$ and $\ell\in\{n-1,\ldots,0\}$ it holds that
\begin{equation}\label{def:Tnell}
	\underline\Sub_{n,\ell}
	\eqdef \underline\Sub_{\ell+1,\ell}\circ\cdots\circ \underline\Sub_{n,n-1}\colon
	\big((\cA_n)^\bZ,\sigma_{n},\fb_n\big)\to
	\big((\cA_{\ell})^\bZ,\sigma_{{\ell}}^{m_{\ell+1}\cdots m_n},
		\fb_{\ell}\big)
\end{equation}
is a metric isomorphism. In particular, taking $\ell=0$, the map
\begin{equation}\label{eq:isomorfus}
	\underline\Sub_{n,0}\colon
	\big((\cA_n)^\bZ,\sigma_{n},\fb_n\big)\to
	\big(\cA^\bZ,\sigma_\cA^{m_1\cdots m_n},\fb\big),
\end{equation}
where $\sigma_\cA$ is the shift in the original alphabet $\cA$, is a metric isomorphism.
\end{remark}

\subsection{Suspension spaces and their internal structure}\label{sec:susmodint}

We now invoke the suspension model in Section \ref{ssec:susmod} and apply it to each alphabet $\cA_n$, $n\in\bN$.
 Let $R_n\colon\cA_n\to\bN$ be  some  function and consider the corresponding function $\underline R_n\colon(\cA_n)^\bZ\to\bN$. Consider the suspension space $\fS_n=\fS_{\cA_n,R_n}$ and the suspension of $\sigma_n=\sigma_{\cA_n}$ by $R_n$ and denote it by $\Phi_n=\Phi_{\cA_n,R_n}$. Denote by $\sim_n$ the corresponding equivalence relation in the definition of the suspension space. Recalling \eqref{eq:defmeasure}, consider the $\Phi_n$-ergodic Borel probability measure
\begin{equation}\label{eq:deflambdan}
	\lambda_n
	= \lambda_{\cA_n,R_n}
	\eqdef \frac{1}{(\fb_n\times \mathfrak m)(\fS_n)}
		(\fb_n\times \mathfrak m)|_{\fS_n}.
\end{equation}
We will always use these short notations, unless there is risk of confusion. 
 
The suspension spaces have a ``self-similar'' internal structure that we proceed to study.  
 
\subsubsection{Roofs and tailing functions}

We impose the following assumption about the roof functions across the cascade. 	Most of this section only requires the lower bound in Assumption \ref{ass:roof}. We will invoke the upper bound only in Proposition \ref{procor:notormenta} to estimate the ``expected roof heights'' and ``expected length of tails''.

\begin{assumption}[Roof functions]\label{ass:roof}
There exists $K>0$ such that for every $b\in\cA_n$ with $\Sub_{n,n-1}(b) =(a_0,\ldots,a_{m_n-1})\in (\cA_{n-1})^{m_n}$ it holds
\[
	\sum_{k=0}^{m_n-1} R_{n-1}(a_k)
	< R_n(b)
	\le \Big(1+K2^{-(n-1)}\Big)\sum_{k=0}^{m_n-1} R_{n-1}(a_k).
\]
\end{assumption}

Assumption \ref{ass:roof} allows to define the ``tailing length'' function 
\begin{equation}\label{eq:suppe}
	t_n\colon \cA_n\to\bN,\quad
	t_n( b)
	\eqdef R_n( b)
		- \sum_{k=0}^{m_n-1}R_{n-1}\big((\Sub_{n,n-1}( b))_k\big).
\end{equation}
As before, we extend it to $\cA_n^\bZ$ by $\underline t_n(\underline b)\eqdef t_n(b_0)$, $\underline b=(\ldots,b_{-1}|b_0,b_1,\ldots)\in(\cA_n)^\bZ$.

\subsubsection{Ground floors}\label{sec:Susspagroflo}

Consider the \emph{$n$th level ground floor} 
\begin{equation}\label{def:groundlfoor}
	\fG_n
	= \fG_{\cA_n,R_n}
	\eqdef (\cA_n)^\bZ\times\{0\}
	\subset \fS_n.
\end{equation}
Given $(\underline a,s)\in \fS_n$ in its canonical form, define by
\[
	\fp_n(\underline a,s)
	\eqdef \underline a
\]
the natural projection from the suspension space ``to its ground floor''. 
By definition, 
\begin{equation}\label{def:groundlfoorreturn}
	\Phi_n^{\underline R_n\circ\fp_n(\underline b,0)}(\underline b,0)
	= \Phi_n^{\underline R_n(\underline b)}(\underline b,0)
	= (\sigma_{n}(\underline b),0)
	\in \fG_n
	\spac{for every}
	(\underline b,0)\in\fG_n.
\end{equation}
Note that \eqref{def:groundlfoorreturn} is a return map on the ground floor.

\subsubsection{Intermediate floors of first order}

We now extend the concept of \emph{ground floor} $\fG_n$ to so-called \emph{intermediate floors}. For that we first divide the suspension space  into its \emph{principal part} and its \emph{tail},
\begin{equation}\label{eq:regtai}
	\fS_n
	= \fR_n\,\smallcupdot \,\,\fT_n,
\end{equation}
as follows. A point $(\underline b,s)\in \fS_n$ (in its canonical representation) is in $\fR_n$ if and only if
\[
	0
	\le s
	< \sum_{k=0}^{m_n-1}R_{n-1}\big((\underline\Sub_{n,n-1}(\underline b))_k\big)
	= \underline R_n(\underline b)-\underline t_n(\underline b).
\]
Otherwise it belongs to $\fT_n$. We define the map 
\[
	P_{n,n-1}\colon \fR_n\to \fS_{n-1}
	,\quad	P_{n,n-1}(\underline b,s)
	\eqdef (\underline \Sub_{n,n-1}(\underline b),s).
\]
Note that for $s>R_{n-1}(\underline\Sub_{n,n-1}(\underline b))$, the latter is not in its canonical representation. To obtain this representation, one has to take into account (possibly several times) the identification 
\[
	(\underline\Sub_{n,n-1}(\underline b),s)
	\sim_{n-1} \big(\sigma_{{n-1}}(\underline\Sub_{n,n-1}(\underline b)),s
		-\underline R_{n-1}(\underline\Sub_{n,n-1}(\underline b))\big).
\]
By construction, the following holds.

\begin{lemma}
 The suspension map $\Phi_{n-1}$ on $\fS_{n-1}$ is a topological factor of the map $\Phi_n$ restricted to $\fR_n$ by the factor map $P_{n,n-1}$. 
\end{lemma}

We now extend the term \emph{ground floor}. The quotient map $P_{n,n-1}$ is $m_n$-to-$1$. Indeed, this follows because $P_{n,n-1}$  precisely restricts to the the principal part $\fR_n$ of the suspension space.  Moreover, there is a natural order of the preimages given by the number of times we have to take into account the identification $\sim_{n-1}$ in order to obtain the canonical representation. Denote by $\fG_n^{(i)}$ the $i$th preimage of the $(n-1)$st level ground floor $\fG_{n-1}$ under $P_{n,n-1}$, for $i=1,\ldots,m_n-1$, and call them \emph{intermediate floors}. More precisely, let $\fG_n^{(0)}\eqdef\fG_n$ and for $i=1,\ldots,m_n-1$ let
\begin{equation}\label{eq:intfloor-a}
	\fG_n^{(i)}
	\eqdef 
	\Big\{(\underline b,s)\in \fS_n
		\colon s = \sum_{\ell=0}^{i-1}
		R_{n-1}\big((\underline \Sub_{n,n-1}(\underline b))_\ell\big)
	\Big\}
	\subset\fS_n.
\end{equation}
They are ``lifted copies'' of the ground floor $\fG_{n-1}$ in the suspension space $\fS_n$, using the fact that any symbol in the alphabet $\cA_n$ is obtained as a concatenation of words in the lower-level alphabet $\cA_{n-1}$. 

Note that the intermediate floors separate the suspension space $\fS_n$ into the \emph{strips} 
\[
	\fL_n^{(i)}
	\eqdef \Big\{(\underline b,s)\in \fS_n\colon
		 \sum_{\ell=0}^{i-1}R_{n-1}\big((\underline\Sub_{n,n-1}(\underline b))_\ell\big)
		 \le s 
		 <  \sum_{\ell=0}^{i}R_{n-1}\big((\underline\Sub_{n,n-1}(\underline b))_\ell\big)
	\Big\}
	\subset\fS_n,
\]
for $i=0,\ldots,m_n-1$ (where for $i=0$ the first sum is understood to be $0$). 

\begin{remark}
A point $(\underline b,s)\in\fS_n$ (in its canonical representation) belongs to the strip $\fL_n^{(i)}$ with address ${}^{(i)}$ if the first (symbolic) coordinate of the canonical representation of $P_{n,n-1}(\underline b,s)$ is $\sigma_{{n-1}}^i(\underline\Sub_{n,n-1}(\underline b))$. 
\end{remark}

By construction, the following holds.

\begin{lemma}\label{lemeq:bijima}
Each strip $\fL_n^{(i)}$ is mapped by $P_{n,n-1}$  onto $\fS_{n-1}$ in a bijective way,
\[
	P_{n,n-1}\big(\fL_n^{(i)}\big)
	= \fS_{n-1}.
\]
Moreover, the (disjoint) union of all strips is the principal part, that is,
\[
	\fR_n
	= \bigcup_{i=0}^{m_n-1}\fL_n^{(i)}.
\]
\end{lemma}

\subsubsection{Intermediate floors of higher order: Inductive definition}\label{ssec:intfloors}
Above we defined intermediate floors $\fG_n^{(\cdot)}$ by means of the roof function defined on the alphabet $\cA_{n-1}$. Recalling that any word in $\cA_{n-1}$ is in turn spelled in the symbols of the alphabet $\cA_{n-2}$, inside any strip $\fL_n^{(i)}$ that is bounded by the intermediate floors $\fG_n^{(i)}$ and $\fG_n^{(i+1)}$ (for some index $i=0,\ldots,m_n-1$) we will introduce further, deeper-level, intermediate floors and strips, and we will continue from level $n-2$ down to any level $\ell\in\{1,\ldots,n-1\}$. 

For reasons which will be apparent in what follows, let us write
\[
	\fR_n^{(n-1)}
	\eqdef \fR_n
	\spac{and}
	\fT_n^{(n-1)}
	\eqdef \fT_n
\]
emphasizing the level $n-1$ that was taken into account in the definition. Analogously, let
\begin{equation}\label{eq:notationintflor}
	\fG_n^{(n-1,(i))}
	\eqdef \fG_n^{(i)}
	\spac{and}
	\fL_n^{(n-1,(i))}
	\eqdef \fL_n^{(i)}.
\end{equation}

Before giving the full, inductive, definition, let us first proceed with one further step. As in \eqref{eq:regtai}, the $(n-1)$st level suspension space splits into its principal part and its tail, $\fS_{n-1}=\fR_{n-1}\,\smallcupdot\,\,\fT_{n-1}=\fR_{n-1}^{(n-2)}\,\smallcupdot\,\,\fT_{n-1}^{(n-2)}$. Let 
\[
	\fR_n^{(n-2)}
	\eqdef P_{n,n-1}^{-1}(\fR_{n-1}^{(n-2)})
	\subset \fR_n^{(n-1)}
	\spac{and}
	\fT_{n}^{(n-2)}
	\eqdef P_{n,n-1}^{-1}(\fT_{n-1}^{(n-2)})
	\cup \fT_n^{(n-1)}.
\]
Hence, the $n$th level suspension space splits as
\[
	\fS_n
	= \fR_n^{(n-2)}\,\smallcupdot\,\,\fT_n^{(n-2)}.
\]
We subdivide the principal part $\fR_n^{(n-2)}$ into strips 
\[\begin{split}
	\fL_n^{(n-2,(j,i))}
	&\eqdef \big(P_{n,n-1}|_{\fL_n^{(n-1,(i))}}\big)^{-1}(\fL_{n-1}^{(n-2,(j))})\\
	\tiny{\text{ using Lemma \ref{lemeq:bijima} }}\quad
	&= 	\big(P_{n,n-1}|_{\fL_n^{(n-1,(i))}}\big)^{-1}
			\circ
		\big(P_{n-1,n-2}|_{\fL_{n-1}^{(n-2,(j))}}\big)^{-1}	(\fS_{n-2}),
\end{split}\]
where $i=0,\ldots,m_n-1$ and $j=0,\ldots,m_{n-1}-1$. Each such strip is separated by the corresponding intermediate floors, defined by
\[
	\fG_n^{(n-2,(j,i))}
	\eqdef \big(P_{n,n-1}|_{\fG_n^{(n-1,(i))}}\big)^{-1}(\fG_{n-1}^{(n-2,(j))}).
\]
The pair of indices $\va=(j,i)$ above labels what we call the \emph{$(n-2,n)$-address} of the strip and the intermediate floor.

Below we will consider further levels of our construction. Let us describe our general terminology.

\begin{notation}[Addresses]{\rm
Each intermediate floor and strip in $\fS_n$ will be indexed by ${}_n$ and ${}^{(\ell,\va)}$, where $\ell\in\{0,\ldots,n-1\}$ and $\va=(\mathrm a_\ell,\ldots,\mathrm a_{n-1})$ is a tuple with $\mathrm a_k\in\{0,\ldots,m_{k+1}-1\}$ for $k=\ell,\ldots,n-1$. Here the lower index ${}_{n}$ indicates to which suspension space the defined set belongs and the upper index ${}^{(\ell,\va)}$ indicates which previous levels are taken into account. The length of the tuple $\va$ indicates the difference of levels, it also implicitly determines $\ell$ which we keept in the notation for better readability. This notation will be also used for other objects of our construction. Compare Figure \ref{fig.1}.
}\end{notation}

We now provide the full inductive definition. Given $n\in\bN$ and $\ell=n-2,\ldots,0$, assuming that the principal part $\fR_{n-1}^{(\ell)}$ and the tail part $\fT_{n-1}^{(\ell)}$ were already defined, let 
\[
	\fR_n^{(\ell)}
	\eqdef P_{n,n-1}^{-1}\big(\fR_{n-1}^{(\ell)}\big),
	\quad
	\fT_n^{(\ell)} 
	\eqdef \fT_n^{(n-1)} \,\smallcupdot\,\, P_{n,n-1}^{-1}\big( \fT_{n-1}^{(\ell)}\big).
\]

\begin{remark}\label{rem:tails}
	Note that $\fT_n^{(\ell)}$ gathers all tails added at levels $\ell+1,\ldots,n$, that is, in each fiber of the suspension space there are
\[\begin{split}
	&\text{one tail added at level }n \text{ of length }\ft_n(\cdot),\\
	&m_n\text{ tails added at level }n-1\text{ of length }\ft_{n-1}(\cdot),
	\ldots,\\
	&m_n\cdots m_{\ell+1}\text{ tails added at level }\ell\text{ of length }\ft_\ell(\cdot).
\end{split}\]	
\end{remark}

We subdivide the principal part $\fR_n^{(\ell)}$ into strips. 
For that, we call $\va=(\mathrm a_\ell,\ldots,\mathrm a_{n-1})$, where $ \mathrm a_k\in\{0,\ldots,m_k-1\}$ for every $k=\ell,\ldots,n$, an \emph{$(\ell,n)$-address}. We define the map that ``lifts'' the $\ell$th level suspension space into the  $n$th level suspension space,
\begin{equation}\label{eq:Lnellalpha}
	L_n^{(\ell,\va)}
	\colon\fS_\ell\to\fS_n,
	\quad
	L_n^{(\ell,\va)}
	\eqdef \big(P_{n,n-1}|_{\fL_n^{(n-1,(\mathrm a_{n-1}))}}\big)^{-1}
			\circ\cdots\circ
		\big(P_{\ell+1,\ell}|_{\fL_{\ell+1}^{(\ell,(\mathrm a_\ell))}}\big)^{-1},
\end{equation}
by concatenating the corresponding inverse branches. Define  by
\begin{equation}\label{eq:defintflo}
	\fL_n^{(\ell,\va)}
	\eqdef L_n^{(\ell,\va)}(\fS_\ell)
	\spac{and}
		\fG_n^{(\ell,\va)}
	\eqdef L_n^{(\ell,\va)}(\fG_\ell)
\end{equation}
the  \emph{strip} and the \emph{intermediate floor} with $(\ell,n)$-address $\va$, respectively. By construction, the following holds.

\begin{lemma}
Every strip $\fL_n^{(\ell,\va)}$ is the bijective image of $\fS_\ell$ under $L_n^{(\ell,\va)}$. Every intermediate floor $\fG_n^{(\ell,\va)}$ is the bijective image of the ground floor $\fG_\ell$ under $L_n^{(\ell,\va)}$.
\end{lemma}

Finally note that
\begin{equation}\label{eq:fodase}
	\fS_n
	= \fR_n^{(\ell)}\,\smallcupdot\,\,\fT_n^{(\ell)}
	\spac{and}
	\fR_n^{(\ell)}
	= \cupdot_\va\fL_n^{(\ell,\va)},
	\end{equation}
where in the latter the union is taken over all $(\ell,n)$-addresses $\va$.	

\begin{figure}[h] 
 \begin{overpic}[scale=.43]{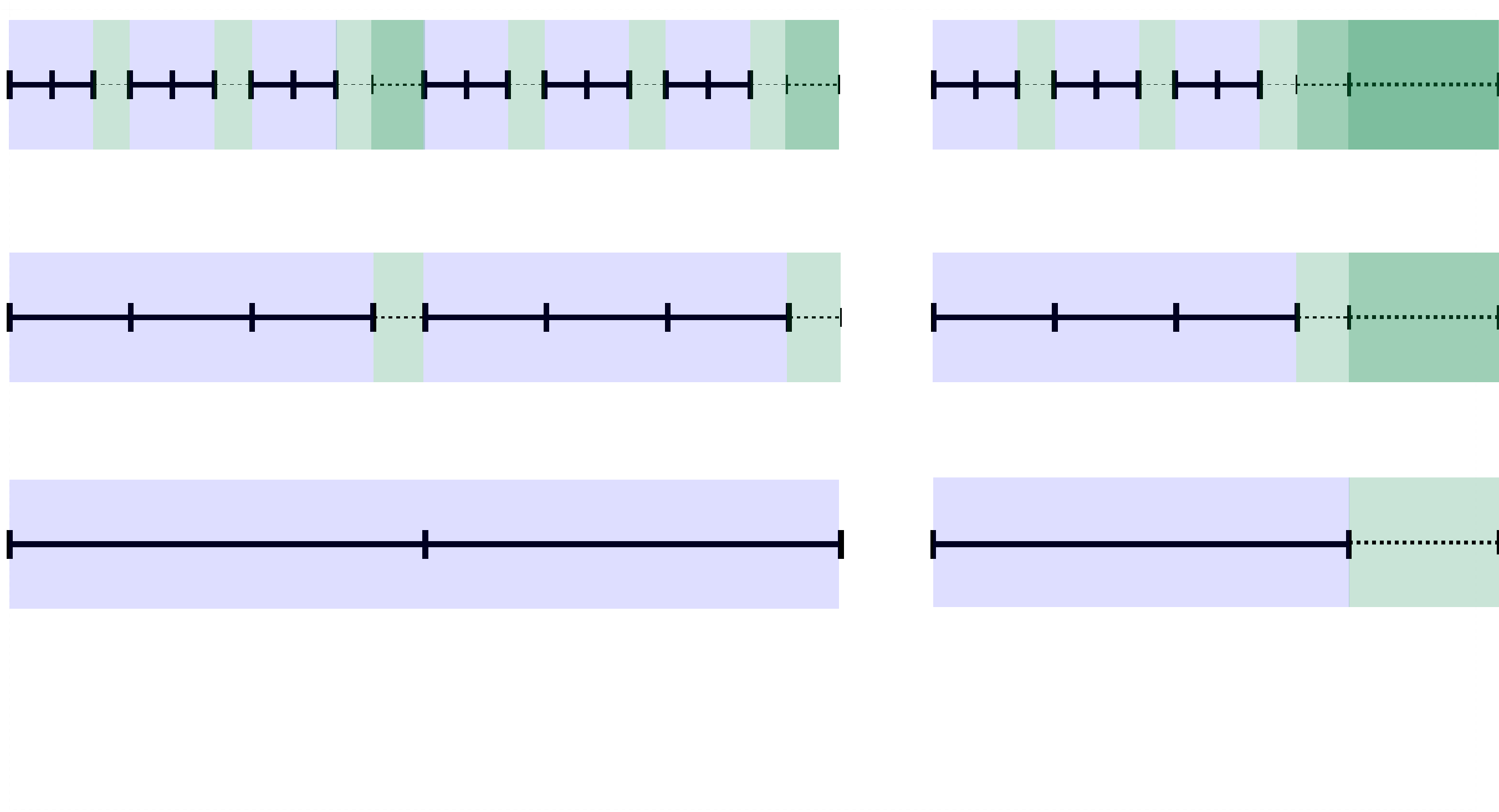}
 	\put(-1,0){ {\rotatebox{90}{\small$\fG_n$}}}
 	\put(-1,26){ {\rotatebox{90}{\tiny$(0,0)$}}}
 	\put(7,26){ {\rotatebox{90}{\tiny$(1,0)$}}}
 	\put(15,26){ {\rotatebox{90}{\tiny$(2,0)$}}}
 	\put(27,26){ {\rotatebox{90}{\tiny$(0,1)$}}}
 	\put(35.3,26){ {\rotatebox{90}{\tiny$(1,1)$}}}
 	\put(43.5,26){ {\rotatebox{90}{\tiny$(2,1)$}}}
 	\put(54,22){ {\rotatebox{45}{\tiny$(0,m_n-1)$}}}
 	\put(62,22){ {\rotatebox{45}{\tiny$(1,m_n-1)$}}}
 	\put(70,22){ {\rotatebox{45}{\tiny$(2,m_n-1)$}}}
 	\put(-1,13){ {\rotatebox{90}{\tiny$(0)$}}}
 	\put(27,13){ {\rotatebox{90}{\tiny$(1)$}}}
 	\put(56,9){ {\rotatebox{45}{\tiny$(m_n-1)$}}}	
	\put(-4,12){{\rotatebox{90}{\tiny$\va=(\mathrm a_{n-1})$}}}
	\put(-7,25){{\rotatebox{90}{\tiny$(\mathrm a_{n-2},\mathrm a_{n-1})$}}}
	\put(-4,37){{\rotatebox{90}{\tiny$(\mathrm a_{n-3},\mathrm a_{n-2},\mathrm a_{n-1})$}}}
 	\put(-1,40){ {\rotatebox{90}{\tiny$(0,0,0)$}}}	
 	\put(1.7,40){ {\rotatebox{90}{\tiny$(1,0,0)$}}}	
	\put(7.3,40){ {\rotatebox{90}{\tiny$(0,1,0)$}}}	
 	\put(10,40){ {\rotatebox{90}{\tiny$(1,1,0)$}}}	
 	\put(52,37){ {\rotatebox{45}{\tiny$(0,0,m_n-1)$}}}	
 	\put(55,37){ {\rotatebox{45}{\tiny$(1,0,m_n-1)$}}}	
	\put(17,40){{\small$\ldots$}}
	\put(57.5,48){{\tiny$\ldots$}}
	\put(57.5,32.6){{\tiny$\ldots$}}
	\put(57.5,18){{\tiny$\ldots$}}
	\put(68,40){{\tiny$\ldots$}}
 \end{overpic}
  \caption{Ground floor $\fG_n$ and addresses of intermediate floors at levels $\ell=n-3,n-2$, and $n-1$ (from top to bottom), respectively. All are subsets of the suspension space $\fS_n$. The shaded regions indicate  the principle part $\fR_n^{(\ell)}$ (blue) and the tail part $\fT_n^{(\ell)}$ (green) for $\ell=n-3,n-2$, and $n-1$ (from top to bottom), respectively.}
 \label{fig.1}
\end{figure}
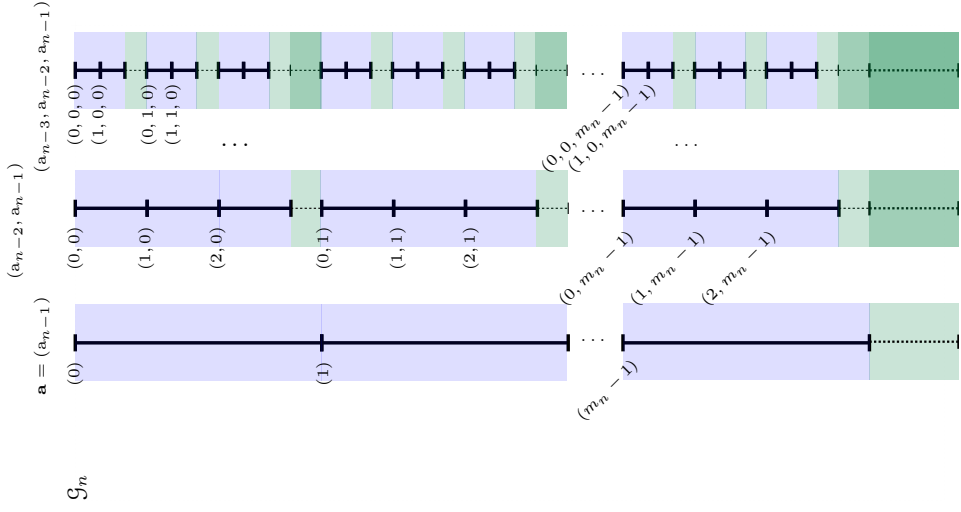

\begin{remark}[Factors between the principal part and lower-level suspension spaces]
	For every $n\in\bN$ and $\ell\in\{n-1,\ldots,0\}$ the map
\[
	P_{n,\ell}
	\eqdef P_{\ell+1,\ell}\circ\cdots\circ P_{n,n-1}\colon
	\fP_n^{(\ell)}\to \fS_\ell,
\]
defines a ``factor map'', though $\fP_n^{(\ell)}$ is not a $\Phi_n$-invariant set.
Given a canonically represented point $(\underline a,s) $,
\[
	(\underline a,s) 
	\eqdef P_{n,n-1}(\underline b,t)	\in\fS_{n-1},
	\spac{where}
	(\underline b,t)\in{\cL}_n^{(n-1,(\mathrm a_{n-1}))}
	\subset\fS_n,
\] 
then
\[
	\underline a
	= (\sigma_{{n-1}}^{\mathrm a_{n-1}}\circ \underline\Sub_{n,n-1})(\underline b)
\]
and by \eqref{eq:isomorfus} it follows that 
\[\begin{split}
	\underline\Sub_{n-1,0}(\underline a) 
	&= \big(\underline\Sub_{n-1,0}\circ \sigma_{{n-1}}^{\mathrm a_{n-1}}\big)(\underline\Sub_{n,n-1}(\underline b))\\
	&= \big((\sigma_\cA^{m_1\cdots m_{n-1}} )^{\mathrm a_{n-1} }\circ \underline\Sub_{n-1,0}\big)
		(\underline\Sub_{n,n-1}(\underline b))
\end{split}\]	 
where $\sigma_\cA=\sigma_0$ is the shift in the original alphabet $\cA$.
In other words, using the notation above, 
\[
	\underline\Sub_{n-1,0}\circ\fp_{n-1}\circ P_{n,n-1}|_{\fL_n^{(n-1,\mathrm a_{n-1})}}
	= \sigma_\cA^{\mathrm a_{n-1}\cdot m_1\cdots m_{n-1}} \circ \underline\Sub_{n,0}\circ\fp_n.
\]
Analogously, for $\ell\in\{0,\ldots,n-1\}$ and $(\ell,n)$-address $\va$ it holds
\begin{equation}\label{eq:conjuu}
	\underline\Sub_{\ell,0}\circ \fp_\ell \circ P_{n,\ell}|_{\fL_n^{(\ell,\va)}}
	= \sigma_\cA^{\sum_{i=\ell}^{n-1} \mathrm a_i \cdot m_1\cdots m_i} 
		\circ \underline\Sub_{n,0} \circ \fp_n.
\end{equation}
\end{remark}

\subsubsection{Localization of intermediate floors in the suspension space}\label{sec:locintflor}

Given $n\in\bN$, for every $(\underline a,0)\in \fG_n$,  $\ell\in\{0,\ldots, n-1\}$, and  $(\ell,n)$-address $\va=(\mathrm a_\ell, \ldots,\mathrm  a_{n-1})$ there is a unique $s_n^{(\ell,\va)}(\underline a)\in\bN_0$ such that $(\underline a, s_n^{(\ell,\va)}(\underline a) ) \in \fG_n^{(\ell,\va)}$ (in its canonical form). Lemma \ref{lem:hund}, that we pospone to the end of the section, precisely describes this number. 
To state and prove this lemma, we start by introducing the necessary notation that will be used thereafter.

Given numbers $n\in\bN$, $\ell\in\{1,\ldots, n-1\}$, and some $(\ell,n)$-address $\va=(\mathrm a_\ell,\ldots,\mathrm a_{n-1})$, for $j\in\{0,\ldots,m_\ell-1\}$ denote by 
\[
	j\va\eqdef(j,\mathrm a_\ell,\ldots,\mathrm a_{n-1})
\]	 
the corresponding $(\ell-1,n)$-address. Given $\ell\in\{0,\ldots, n-1\}$, define the collection of all $(\ell,n)$-addresses
\[
	W_n^{(\ell)}
	\eqdef \{\va=(\mathrm a_\ell,\ldots,\mathrm a_{n-1})
			\colon \mathrm a_k \in\{0,\ldots, m_{k+1}-1\}, k=\ell,\ldots,n-1\}
\] 
 and  let
\[
	W_n
	\eqdef \bigcup_{\ell=0}^{n-1} W_n^{(\ell)}.
\]

Consider on $W_n$ the equivalence relation obtained by identifying $\va\in W_n^{(\ell)}$ with $0\va\in W_n^{(\ell-1)}$ and denote by $\tilde W_n$ the corresponding collection of equivalence classes.
For simplicity, by a slight abuse of notation, we denote by $\va$ the equivalence class it represents. An element in $\tilde W_n$ can have several \emph{representations}. 
Given $\va=(\mathrm a_\ell,\ldots,\mathrm a_{n-1})\in \tilde W_n$, $\va\ne\overline0$, its \emph{simplified representation}%
\footnote{We use the term \emph{simplified} to avoid confusion with the term \emph{canonical} defined above.}
 is the unique tuple $(\mathrm a_k,\ldots,\mathrm a_{n-1})$ where $k\ge\ell$ and $\mathrm a_k\ne0$ and where $k\in\{\ell,\ldots,n-1\}$ is minimal with this property, and we let in this case
\[
	w(\va)
	\eqdef k.
\] 
The simplified representation of a tuple $\overline 0$ consisting only of $0$s is $0$, and in this case we let $w(\overline 0)=0$.
Given $\va=(\mathrm a_\ell,\ldots,\mathrm a_{n-1})\in \tilde W_n$, let 
\[
	\lVert\va\rVert
	\eqdef \sum_{k=w(\va)}^{n-1} \mathrm a_k.
\]	
Note that this value does not depend on the representation of $\va\in \tilde W_n$.

Given $\va=(\mathrm a_\ell,\ldots,\mathrm a_{n-1})\in \tilde W_n$, to ``move between intermediate floors'', assuming $\mathrm a_\ell<m_\ell-2$, let us introduce the notation
\[
	\va+1_\ell
	\eqdef (\mathrm a_\ell,\ldots,\mathrm a_{n-1}) +1_\ell
	\eqdef (\mathrm a_\ell+1,\ldots,\mathrm a_{n-1}).
\]
Analogously, if $\ell\in\{1,\ldots,n-1\}$, let
\[
	\va+1_{\ell-1}
	\eqdef (\mathrm a_\ell,\ldots,\mathrm a_{n-1}) +1_{\ell-1}
	\eqdef (1,\mathrm a_\ell,\ldots,\mathrm a_{n-1}).	
\]
For every $\va\in \tilde W_n$ there exists a unique sequence $\va^{(0)}=0,\va^{(1)},\ldots, \va^{(\lVert\va\rVert)}=\va$ of elements of $\tilde W_n$ such that 
\[
	\va^{(k+1)}
	= \va^{(k)} + 1_{w(\va^{(k+1)})}.
\]	
Indeed, if $\va=(\mathrm a_\ell,\ldots,\mathrm a_{n-1})$ is in its simplified representation, then:
\begin{equation}\label{eq:serefer}\begin{split}
	&\va^{(0)}=
	(0)\\
	&\va^{(1)}=(1),\va^{(2)}=(2),\ldots,\va^{(\mathrm a_{n-1})}=(\mathrm a_{n-1}),\\
	&\va^{(\mathrm a_{n-1}+1)}=(1,\mathrm a_{n-1}),\va^{(\mathrm a_{n-1}+2)}=(2,\mathrm a_{n-1}),\ldots,\\
	&\va^{(\mathrm a_{n-1}+\mathrm a_{n-2})}=(\mathrm a_{n-2},\mathrm a_{n-1}),\\
	&\ldots,\\
	&\va^{(\sum_{k=\ell+1}^{n-1}\mathrm a_k+1)}=(1,\mathrm a_{\ell+1},\ldots,\mathrm a_{n-1}),\ldots,\va^{(\lVert\va\rVert)}=(\mathrm a_\ell,\mathrm a_{\ell+1},\ldots,
	\mathrm a_{n-1}).
\end{split}\end{equation}

Given $\underline a\in(\cA_n)^\bZ$, recall that by \eqref{def:Tnell} the sequence $\underline b=\underline\Sub_{n,\ell}(\underline a)\in(\cA_\ell)^\bZ$ is obtained by ``reading $\underline a$ in its spelling in the alphabet $\cA_\ell$''. Recall also that if $\underline b=(\ldots,b_{-1}|b_0,b_1,\ldots)$ then $b_k$ denotes the $k$th element in this bi-infinite sequence (in the  alphabet $(\cA_\ell)^\bZ$). The above discussion proves the lemma that we finally state. Compare also Figure \ref{fig.1b}.

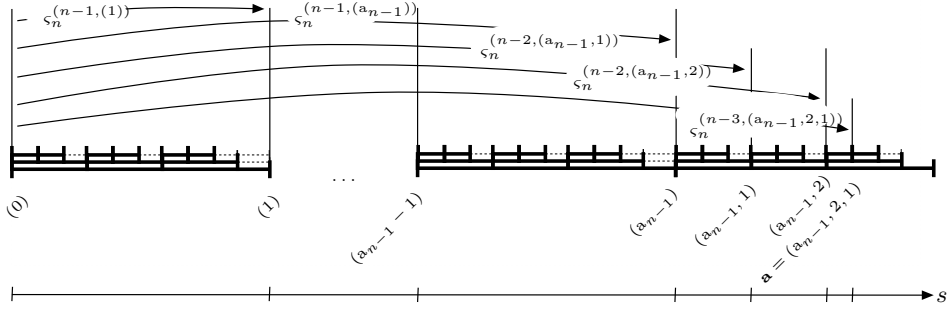
\begin{figure}[h] 
 \begin{overpic}[scale=.45]{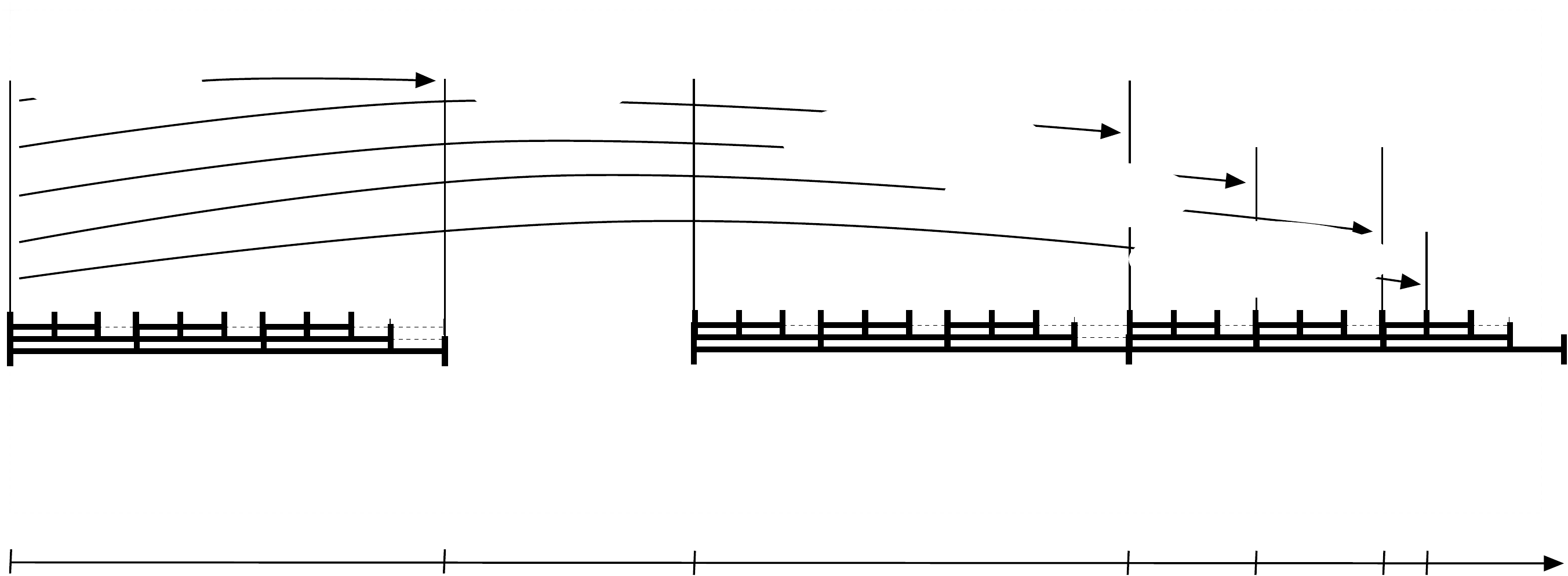}
 	\put(-1,9){ {\rotatebox{45}{\tiny$(0)$}}}
 	\put(26,9){ {\rotatebox{45}{\tiny$(1)$}}}
 	\put(36,4){ {\rotatebox{45}{\tiny$(\mathrm a_{n-1}-1)$}}}	
 	\put(66,7){ {\rotatebox{45}{\tiny$(\mathrm a_{n-1})$}}}	
 	\put(73,6){ {\rotatebox{45}{\tiny$(\mathrm a_{n-1},1)$}}}	
 	\put(81,7){ {\rotatebox{45}{\tiny$(\mathrm a_{n-1},2)$}}}	
 	\put(80,2){ {\rotatebox{45}{\tiny$\va=(\mathrm a_{n-1},2,1)$}}}	
	\put(35,13){{\tiny$\ldots$}}
	\put(51,26.5){{\rotatebox{0}{\tiny$\varsigma_n^{(n-2,(\mathrm a_{n-1},1))}$}}}
	\put(61,23.2){{\rotatebox{0}{\tiny$\varsigma_n^{(n-2,(\mathrm a_{n-1},2))}$}}}
	\put(73.5,18){{\rotatebox{0}{\tiny$\varsigma_n^{(n-3,(\mathrm a_{n-1},2,1))}$}}}
	\put(4,30){{\rotatebox{0}{\tiny$\varsigma_n^{(n-1,(1))}$}}}
	\put(31,30){{\rotatebox{0}{\tiny$\varsigma_n^{(n-1,(\mathrm a_{n-1}))}$}}}
	\put(100,0){{\small $s$}}
 \end{overpic}
  \caption{Example: Moving in $\fS_n$ to the intermediate floor with address $\va=(\mathrm a_{n-1},2,1)$. The maps depict the shifts in the bi-infinite sequences to the corresponding symbol in the corresponding alphabet. The horizontal coordinate marks the position $s$ of a point $(\underline a,s)\in\fS_n$, where $s=s_n^{(\ell,\va)}(\underline a)$ is given by the corresponding $(\ell,n)$-address $\va$.}
 \label{fig.1b}
\end{figure}

\begin{lemma}\label{lem:hund}
For every $n\in\bN$, $\ell\in\{0,\ldots,n-1\}$, and $(\ell,n)$-address $\va=(\mathrm{a}_\ell,\ldots,\mathrm a_{n-1})$ the following is true.
Let
\[
	k=k(n,\ell,\va)=\sum_{i=\ell}^{n-1} 
		\mathrm a_i \cdot m_{\ell+1}\cdots m_{i+1}.
\]		
and consider the map 
\begin{equation}\label{eq:posintfloooor}
	\varsigma_n^{(\ell,\va)}\colon(\cA_n)^\bZ\to\cA_\ell,\quad
	\varsigma_n^{(\ell,\va)}(\underline a)
	\eqdef (\underline\Sub_{n,\ell}(\underline a))_k
	=\big( (\sigma_\ell^k\circ \Sub_{n,\ell})(\underline a)\big)_0.
\end{equation}
Then the number  
\[
	s
	= s_n^{(\ell,\va)}(\underline a)
	\eqdef \sum_{i=1}^{\lVert\va\rVert} 
		R_{\ell_i}(\varsigma_n^{(\ell_i,\va^{(i-1)})}(\underline a)),
	\spac{where}
	\ell_i = w(\va^{(i)}),
\]
is the unique number such that $(\underline a,s)\in\fG_n^{(\ell,\va)}$ (in its canonical form).
\end{lemma}

Note that, in the above lemma, $\ell_k$ is simply the corresponding level at which we change from one address to its successor.

\subsection{Lifted Bernoulli measures on intermediate floors}\label{ssec:bermea}

For $n\in\bN$, consider the Bernoulli measure $\fb_n$ on $(\cA_n)^\bZ$. Recall that $\fb_0=\fb$. For every $(\ell,n)$-address $\va$ let
\begin{equation}\label{eq:defbnellva}
	\fb_n^{(\ell,\va)} 
	\eqdef \fb_n \circ  \fp_n|_{\fG_n^{(\ell,\va)}}.
\end{equation}
Note that if $\va$ is an $(\ell,n)$-address for some $\ell>0$, then $0\va$ is an $(\ell-1,n)$-address. Hence
\[
	\fb_n^{(\ell,\va)}
	= \fb_n^{(\ell-1,0\va)}.
\]

\begin{lemma}\label{lem:defBerinmflo}
	The measure $\fb_n^{(\ell,\va)}$ is a Borel probability measure on $\fG_n^{(\ell,\va)}$ satisfying
\[	
	\fb_n^{(\ell,\va)} 
	= \fb_\ell \circ \fp_\ell \circ P_{n,\ell}|_{\fG_n^{(\ell,\va)}}.
\]
\end{lemma}

\begin{proof}
Note that, by definition, it holds
\[
	\fp_n(\fG_n^{(\ell,\va)}) 
	= (\cA_n)^\bZ
\]
and that $\fp_n$ bijectively maps $\fG_n^{(\ell,\va)}$ onto $(\cA_n)^\bZ$. Recall that $\underline\Sub_{n,0}$ is a bijection and that, by definition, it holds 
\begin{equation}\label{eq:doris}
	\fb_n 
	= 
	\fb_0 \circ \underline\Sub_{n,0}.
\end{equation}
As $(\underline\Sub_{n,0} \circ \fp_n|_{\fG_n^{(\ell,\va)}})$ is bijective, to see that both measures coincide, it is enough to check 
\[
	(\underline\Sub_{n,0} \circ \fp_n|_{\fG_n^{(\ell,\va)}})_\ast \fb_n^{(\ell,\va)} 
	=(\underline\Sub_{n,0} \circ \fp_n|_{\fG_n^{(\ell,\va)}})_\ast 
		(\fb_\ell \circ \fp_\ell \circ P_{n,\ell}|_{\fG_n^{(\ell,\va)}}).
\]
On one hand \eqref{eq:doris} immediately implies
\[\begin{split}	
	(\underline\Sub_{n,0} \circ \fp_n|_{\fG_n^{(\ell,\va)}})_\ast \fb_n^{(\ell,\va)} 
	=  (\underline\Sub_{n,0} \circ \fp_n|_{\fG_n^{(\ell,\va)}})_\ast 
		(\fb_n \circ \fp_n|_{\fG_n^{(\ell,\va)}})
	= \fb_0.
\end{split}\]
On the other hand, by definition and using \eqref{eq:conjuu}, it holds
\[\begin{split}	
	&(\underline\Sub_{n,0} \circ \fp_n|_{\fG_n^{(\ell,\va)}})_\ast 
		(\fb_\ell \circ \fp_\ell \circ P_{n,\ell}|_{\fG_n^{(\ell,\va)}})\\
	&= (\sigma_{\cA}^{-\sum_{i=\ell}^{n-1} \mathrm a_i \cdot m_1\cdots m_i} 
		\circ \underline\Sub_{\ell,0} \circ \fp_\ell \circ P_{n,\ell}|_{\fG_n^{(\ell,\va)}})_\ast 
		(\fb_\ell \circ \fp_\ell \circ P_{n,\ell}|_{\fG_n^{(\ell,\va)}})	\\
	{\tiny{\text{using \eqref{eq:doris}}}}\quad	
	&= (\sigma_{\cA}^{-\sum_{i=\ell}^{n-1} \mathrm a_i \cdot m_1\cdots m_i} 
		\circ \underline\Sub_{\ell,0} )_\ast \fb_\ell 	
	= \big(\sigma_{\cA}^{-\sum_{i=\ell}^{n-1} \mathrm a_i \cdot m_1\cdots m_i} \big)_\ast
	\fb_0
	= \fb_0,
\end{split}\]
where the latter follows from the fact that $\fb_0$ is Bernoulli and hence $\sigma_\cA$-invariant.
\end{proof}

\subsection{Roof functions: Upper bounds and estimates}\label{sec:rooffunctions}

Recall the constant $K$ in Assumption \ref{ass:roof}. Let 
\[
	L_2
	\eqdef  2\max\Big\{
		K e^K
			\frac{\max R_0}{\min R_0},
		e^K\frac{\max R_0}{\min R_0}\Big\}.
\]

\begin{proposition}[Estimates on roof functions]\label{procor:notormenta}
	Let $\frakR_n\eqdef\int R_n\,d\fb_n$. Under Assumption \ref{ass:roof} it holds  
\begin{itemize}
\item[(1)] $m_n \max R_{n-1} < \max R_n $,\vspace{0.1cm}
\item[(2)] $\displaystyle	\frac{\max R_n}{\frakR_n}
	\le \frac{\max R_n}{\min R_n}
	\le \prod_{k=\ell}^n(1+\frac{1}{2^{k-1}}K)
		\frac{\max R_{\ell-1}}{\min R_{\ell-1}}
	<L_2$, \vspace{0.1cm}
\item[(3)] $\displaystyle	1<\frac{\frakR_n}{m_n\frakR_{n-1}}<1+L_2\frac{1}{2^n}$,
\item[(4)] $\displaystyle	\max\lvert\ft_{n}\rvert
	\le L_2\frac{1}{2^n}\frakR_n$, where $\ft_n$ is as in \eqref{eq:suppe}.
\end{itemize}
\end{proposition}

\begin{proof}
By Assumption \ref{ass:roof}, it holds
\[
	\sum_{j=0}^{m_n-1}\underline R_{n-1}\circ\sigma_{{n-1}}^j
	< \underline R_n\circ \underline\Sub_{n,n-1}^{-1}.
\]
Observe that this implies property (1). 
Integrating the above, using that $\fb_{n-1}$ is $\sigma_{n-1}$-invariant and   \eqref{eq:isedbelow}, we obtain the following estimate which we use below
\begin{equation}\label{eq:chuvab}
	m_n\frakR_{n-1}
	<\int \underline R_n\circ \underline\Sub_{n,n-1}^{-1}\,d\fb_{n-1}
	=\int \underline R_n\,d\fb_n
	=\frakR_n.
\end{equation}

To prove (2), observe that the second inequality in Assumption \ref{ass:roof} implies that
\begin{equation}\label{eq:chuvaanew}
	\underline R_n\circ \underline\Sub_{n,n-1}^{-1}
	\le (1+K\frac{1}{2^{n-1}})
		\sum_{j=0}^{m_n-1}\underline R_{n-1}\circ\sigma_{{n-1}}^j.
\end{equation}
Hence, together with $\min R_n>m_n\min R_{n-1}$, it follows
\begin{equation}\label{eq:fellow}\begin{split}
	\frac{\max R_n}{\min R_n}
	&\le (1+K\frac{1}{2^{n-1}})
		\frac{\max R_{n-1}}{\min R_{n-1}}\\
	&\le \prod_{k=\ell}^n(1+K\frac{1}{2^{k-1}})
		\frac{\max R_{\ell-1}}{\min R_{\ell-1}}
	< e^K \frac{\max R_0}{\min R_0},
\end{split}\end{equation}
proving (2).

Recalling again that $\fb_{n-1}$ is $\sigma_{{n-1}}$-invariant, we note that by  \eqref{eq:chuvaanew} 
\[\begin{split}
	\frakR_n
	&= \int \underline R_n\,d\fb_n
	=\int \underline R_n\circ \underline\Sub_{n,n-1}^{-1}\,d\fb_{n-1}\\
	&\le \Big(1+K\frac{1}{2^{n-1}}\Big) \int
		 \sum_{j=0}^{m_n-1}\underline R_{n-1}\circ \sigma_{{n-1}}^j\,d\fb_{n-1}
	= \Big(1+L_2\frac{1}{2^n}\Big)m_n\frakR_{n-1}	,
\end{split}\]
where we also used $2K\le L_2$.
This together with \eqref{eq:chuvab} implies  (3).

To get (4), by the second inequality in Assumption \ref{ass:roof}  and item (2) it holds
\[
	\lvert \ft_n\rvert
	\le K\frac{1}{2^{n-1}} \cdot m_n\max R_{n-1}
	< K\frac{1}{2^{n-1}}\cdot\max R_n
	\le K\frac{1}{2^{n-1}}  \cdot\frac{\max R_n}{\min R_n}\,\frakR_n.
\]
The estimate in (4)  now follows from \eqref{eq:fellow}.
\end{proof}

%
%

\section{Contracting IFSs and horseshoes}\label{sec:horses}

In what follows, we consider $C^1$ diffeomorphisms $f_1,\ldots,f_N\colon\bS^1\to\bS^1$ and its associated skew product $F$ as in \eqref{eq:sp}. In Section \ref{ssec:horsIFS}, given an appropriate finite collection of words $\cB\subset\Sigma_N^\ast$, following \cite{Hut:81}, we consider the attractor of an associated contracting IFS.  In Section \ref{ssec:horsexpent}, we introduce a contracting IFS with further quantifiers. In Section \ref{sec:pieces}, we explain how such collection is derived from an $F$-ergodic hyperbolic measure with negative Lyapunov exponent. Here we invoke the idea of skeletons associated to an ergodic measure, relying on the axioms stated in Section \ref{sec:axioms}. The main result of this section is Theorem \ref{teo:existenceCIFS}. 
In Section \ref{ssec:horsdist} we derive an auxiliary distortion result. 
Finally, in Section \ref{sec:horrrrse} we explain how these attractors lead to horseshoes invariant under the skew product $F$.

 In this section, all words are over the alphabet $\{1,\ldots,N\}$. We will invoke the concepts and objects in Sections \ref{sec:notation} and \ref{se:suspmode} in this particular case. Given a finite collection of words $\cB\subset\Sigma_N^\ast$, let 
 \[
 	\CS{\cB}
	= \bigcup_{k\in\bZ}\sigma^k(\PCS{\cB})
	=\bigcup_k(\sigma^k\circ\iota_\cB)(\cB^\bZ)
\]	 
as in \eqref{eq:number}, where $\sigma$ is the usual shift in $\Sigma_N$   and $\underline\iota_\cB\colon\cB^\bZ\to\Sigma_N$ as in \eqref{eq:iotaW2}. We will drop the corresponding index ${}_\cB$ unless there is risk of confusion. Recall our simplifying Notation \ref{eq:notation}. In particular,  for $w\in\cB$ we denote by $\lvert w\rvert$ the length of this word ``spelled in $\{1,\ldots,N\}$''.
 
 \begin{notation}{\rm
Denote by $\cB^{-\bN}$ and $\cB^{\bN_0}$ the corresponding one-sided shift spaces, similarly  $\Sigma_N^-=\{1,\ldots,N\}^{-\bN}$ and $\Sigma_N^+=\{1,\ldots,N\}^{\bN_0}$. 
Given $\xi=(\ldots,\xi_{-1}|\xi_0,\xi_1,\ldots)\in\Sigma_N$ we write $\xi=\xi^-|\xi^+$, where $\xi^+\in\Sigma_N^+$ and $\xi^-\in\Sigma_N^-$.
Denote by $\pi^\pm\colon\Sigma_N\to\Sigma_N^\pm$ the projections 
\[
	\pi^-(\xi^-|\xi^+)
	\eqdef\xi^-,
	\quad
	\pi^+(\xi^-|\xi^+)
	\eqdef\xi^+.
\] 
For $n\in\bN_0$, let $[\xi_0,\ldots,\xi_n]^+=\pi^+([\xi_0,\ldots,\xi_n])$. 
We consider the distance 
\[
d^+_1(\xi^+,\eta^+)\eqdef e^{-n(\xi^+,\eta^+)}, \,\,\text{where}\,\,
n(\xi^+,\eta^+)\eqdef \inf\{ \ell\colon \xi^+_\ell\ne\eta^+_\ell\},
\]
on $\Sigma_N^+$. Analogously, we define $d^-_1$ on $\Sigma_N^-$. 
}\end{notation}
 
\subsection{The attractor of a contracting IFS}\label{ssec:horsIFS}

\begin{definition}[CIFS]\label{def:CIFS0}
	A finite collection of words $\cB\subset\Sigma_N^\ast$ defines a \emph{contracting iterated function system} (\emph{CIFS}) on a closed interval  $J\subset\bS^1$ if for every $w\in\cB$ it holds
\begin{itemize}
\item[(a)]  $f_{[w]}(J)\subset J$,
\item[(b)] $\lvert f_{[w]}'(x)\rvert<1$ for every $x\in J$.
\end{itemize}
\end{definition}

In Definition \ref{def:CIFS} below, we will further specify a CIFS including some quantifiers of the rate of contraction. We first establish the existence of the attractor for a CIFS.

\begin{proposition}[Attractor of a CIFS]\label{pro:defineshorseshoe}
	Let $\cB\subset\Sigma_N^\ast$ be a finite collection of words over the alphabet $\{1,\ldots,N\}$ defining a CIFS on a closed interval $J\subset\bS^1$. The map
\[
	x\colon\cB^\bZ\to\bS^1,\quad x(\underline w)
	\eqdef \lim_{n\to\infty} \big(f_{[w_{-1}]} \circ \cdots\circ f_{[w_{-n}]}\big) (x_0),
\]
is well defined for every $\underline w=(\ldots,w_{-1}|w_0,w_1,\ldots)\in\cB^\bZ$ and independent of the point  $x_0\in J$.	
Moreover, the map 
\[
	\Pi_\cB\colon\cB^\bZ\to \Sigma_N\times\bS^1,\quad
	\Pi_\cB(\underline w)\eqdef (\underline\iota_\cB(\underline w),x(\underline w))
\]
is continuous and satisfies
\[
	\big(\Pi_\cB\circ\sigma_\cB\big)
	(\ldots,w_{-1}|w_0,w_1,\ldots)
	=\big(F^{\lvert w_0\rvert}\circ\Pi_\cB\big)
	(\ldots,w_{-1}|w_0,w_1,\ldots).
\]
Consider the  \emph{attractor} associated to $\cB$,
\[
	\Lambda(\cB)\eqdef\Pi_\cB(\cB^\bZ).
\]
The map 
\[
	\big(\underline w,x)
	\mapsto F^{\lvert w_0\rvert}(\underline\iota_\cB(\underline w),x)
\]
is a return map on $\Lambda(\cB)$.

Moreover, if $\cB$ is disjoint then $\Pi_\cB$ is uniformly finite-to-one so that 
\[
	\card \Pi_\cB^{-1}(\{X\})
	\le \max_{w\in\cB}\,\lvert w\rvert,
	\spac{for every}X\in \Lambda(\cB).
\]	
\end{proposition}

\begin{proof}
Let $\cB=\{w_1,\ldots,w_M\}$ be some enumeration. Recall that the (bi-)infinite concatenation of words in $\cB$ gives a (bi-)infinite sequence in $\Sigma_N$. We will first consider one-sided sequences in $\cB^{-\bN}\subset\Sigma_N^-$. For every $k=1,\ldots, M$ define
\[
	\hat\sigma_k^{-1}\colon\cB^{-\bN}\to\cB^{-\bN},\quad
	\hat\sigma_k^{-1}(\ldots,w_{i_{-2}},w_{i_{-1}})
	\eqdef(\ldots,w_{i_{-2}},w_{i_{-1}},w_k)
\]	
and consider the map
\[
	\hat f_k\eqdef \hat\sigma^{-1}_k\times f_{[w_k]}
	\colon D\to \cB^{-\bN}\times \bS^1,
	\quad\text{ where }\quad
	D\eqdef \cB^{-\bN}\times J.
\]
By construction, it holds $\hat f_k(D)\subset D$ and $\hat f_k$ is uniformly contracting for every $k$. Therefore, $\{\hat f_1,\ldots,\hat f_M\}$ is a finite family of contractions on $D$. By  \cite{Hut:81}, we can consider its associated attractor $\Att^-(\cB)\subset D$. Every point $\big(\underline w^{-},x\big)\in \Att^-(\cB)$, with $\underline w^-=(\ldots,w_{i_{-2}},w_{i_{-1}})\in\cB^{-\bN}$, is uniquely defined by its first coordinate. Indeed, the map
\[
	\underline w^-\in\cB^{-\bN}\mapsto (\underline w^-,\hat x(\underline w^-)),\quad
	\hat x(\underline w^-)
	\eqdef \lim_{n\to\infty} \big(f_{[w_{i_{-1}}]} \circ \cdots\circ f_{[w_{i_{-n}}]}\big) (x_0),
	\quad x_0\in J,
\]
is continuous and onto $\Att^-(\cB)$ (and in particular, it does not depend on $x_0$).
Let
\[
	\Att(\cB)\eqdef \Att^-(\cB)\times\cB^{\bN_0}.
\]
For convenience, we write a point $(\underline w^-,\hat x(\underline w^-),\underline w^+)\in \Att(\cB)$ as $(\underline w^-,\hat x(\underline w^-),\underline w^+)=(\underline w^-|\underline w^+,\hat x(\underline w^-))=(\underline w,\hat x(\underline w^-))$. Letting $x(\underline w)\eqdef \hat x(\underline w^-)$, this ends the definition of $\Pi_\cB$. 

By construction, for every $\underline w= (\ldots,w_{i_{-1}}|w_{i_0},w_{i_1},\ldots)$ it holds
\[\begin{split}
	(\Pi_\cB\circ\sigma_\cB)(\underline w)
	&= \big((\underline\iota_\cB\circ\sigma_\cB)(\underline w),
		x(\sigma_\cB(\underline w))\big)\\
	&= \big((\underline\iota_\cB\circ\sigma_\cB)(\underline w),f_{[w_{i_0}]}(x(\underline w))\big)
	= F^{\lvert w_{i_0}\rvert}(\underline\iota_\cB(\underline w),x(\underline w))\\
	&= \big(F^{\lvert w_{i_0}\rvert}\circ\Pi_\cB\big)(\underline w).
\end{split}\]

Finally, let us check the cardinality of the set of preimages $\Pi_\cB^{-1}(\{X\})$ for any point $X=(\iota_\cB(\underline w),x(\underline w))$. Note that if $\cB$ disjoint then together with Lemma \ref{lem:soandso} every element in $\PCS{\cB}=\underline\iota_\cB(\cB^\bZ)$ has at most $\max_{ w\in\cB}\,\lvert w\rvert$ decodings in $\cB$.
\end{proof}

\subsection{Contracting IFS with quantifiers}\label{ssec:horsexpent}

The following extends Definition \ref{def:CIFS0}, adding some contraction quantifiers.

\begin{definition}[CIFS with quantifiers]\label{def:CIFS}
	A finite collection of words $\cB	\subset\Sigma_N^\ast$ defines a \emph{contracting iterated function system} (\emph{CIFS}) on an interval  $J\subset\bS^1$ relative to $K\ge1$, $\alpha_0<0$, $\alpha<0$, and  $\varepsilon\in(0,\lvert\alpha\rvert)$ if
\begin{itemize}
\item[(a)] for every $ w\in\cB$ it holds $f_{[ w]}(J)\subset J$,
\item[(b)] for every $m\in\bN$, $ w_1,\ldots, w_m\in\cB$, and $k=1,\ldots,\lvert( w_1,\ldots, w_m)\rvert$
 it holds
\[
	\lvert (f_{[ w_1,\ldots, w_m]}^k)'(y)\rvert
	\le Ke^{k\alpha_0},
\]
\item[(c)]
the \emph{spectrum of finite-time Lyapunov exponents} satisfies
\[
	\Big\{\frac{1}{\lvert w\rvert}\log\,\lvert (f_{[ w]})'(x)\rvert
	\colon x\in J, w\in\cB\Big\}
	\subset (\alpha-\varepsilon,\alpha+\varepsilon).
\] 
\end{itemize}
\end{definition}

\subsection{Existence of contracting IFS with quantifiers}\label{sec:pieces}

Given any $F$-ergodic measure with negative Lyapunov exponent, the following theorem provides a collection of words which defines a CIFS with quantifiers. It builds on the existence of ``skeletons'', that is, orbit pieces which ``ergodically mimic'' the measure, see Claim \ref{clalem:skeletonstar} and \cite[Section 4]{DiaGelRam:17} for further discussion. 

\begin{theorem}[Existence of a CIFS with quantifiers]\label{teo:existenceCIFS}
	Let $F\in \mathrm{SP}^1_{\rm shyp}(\Sigma_N\times\bS^1)$ and $\mu$ an $F$-ergodic hyperbolic measure with Lyapunov exponent $\alpha=\chi(\mu)<0$ and entropy $h=h(F,\mu)>0$. Then for every $\varepsilon_E\in(0,\lvert\alpha\rvert/4)$ and $\varepsilon_H\in(0,h)$ there exist a closed interval $J\subset\bS^1$ and a finite disjoint collection of words $\cB\subset\Sigma_N^\ast$ defining a CIFS on $J$ relative to some constant $K>1$ and $\alpha+\varepsilon_E$, $\alpha$, and $\varepsilon_E$ such that 
\[
	\min_{ w\in\cB}\lvert w\rvert(h-\varepsilon_H)
	\le \log\card\cB
	\le \max_{ w\in\cB}\lvert w\rvert(h+\varepsilon_H).
\]
\end{theorem}

\begin{proof}
Let $F\in \mathrm{SP}^1_{\rm shyp}(\Sigma_N\times\bS^1)$ with associated constants $K_1,\ldots, K_5,K_6$  as in Remark  \ref{rem:commonblending}.
Fix $\varepsilon_E\in(0,\lvert\alpha\rvert/4)$ and $\varepsilon_H\in(0,h)$.
We will use the following result.

\begin{claim}[Existence of skeletons, {\cite[Proposition 4.11]{DiaGelRam:17}}]\label{clalem:skeletonstar}
	There exist $K_0,L_0\ge1$, and $n_0\in\bN$ such that for every $n\ge n_0$ there exists a finite set $\fX=\fX(n)=\{(\xi^i,x_i)\}\subset\Sigma_N\times\bS^1$, where $\xi^i=(\ldots,\xi^i_{-1}|\xi^i_0,\xi^i_1,\ldots)$, satisfying:
\begin{itemize}
\item[(i)] the set $\fX$ has cardinality
\[
	L_0^{-1}e^{n(h-\varepsilon_H/2)}
	\le\card\fX
	\le L_0e^{n(h+\varepsilon_H/2)},
\] 
\item[(ii)] the words $(\xi^i_0,\ldots,\xi^i_{n-1})$ are all different, and
\item[(iii)] for every $k=1,\ldots,n$ it holds
\[
	K_0^{-1}e^{k(\alpha-\varepsilon_E/4)}
	\le \lvert(f_{\xi^i}^k)'(x_i)\rvert
	= \lvert(f_{[\xi^i_0,\ldots,\xi^i_{k-1}]})'(x_i)\rvert
	\le K_0e^{k(\alpha+\varepsilon_E/4)}.
\]
\end{itemize}
\end{claim}

\smallskip\noindent{\textbf{Control of distortion.}}
Let $K_0,L_0$, and $n_0$ be as in Claim \ref{clalem:skeletonstar}. We need some auxiliary distortion results. Let
\begin{equation}\label{def:cunifconst}
		\lVert F\rVert
		\eqdef \max\Big\{\lvert f_i'(x)\rvert,\lvert(f^{-1}_i)'(x)\rvert
		\colon i=1,\ldots,N,x\in\bS^1\Big\}.
\end{equation}
Let
\begin{equation}\label{eq:defMod}	
	\Mod_F(\varepsilon)
	\eqdef \Big\{\big\lvert\log{\lvert f_i'(y)\rvert}-\log{\lvert f_i'(x)\rvert}\big\rvert
			\colon i=1,\ldots,N,x,y\in\bS^1,\lvert y-x\rvert\le\varepsilon\Big\}. 
\end{equation}
Clearly, $\Mod_F(\varepsilon)\to0$ as $\varepsilon\to0$. Let $r\in(0,1)$ so that 
\begin{equation}\label{eq:rK1}
	\Mod_F(K_0r)\le \varepsilon_E/4
	\spac{and}
	2r <K_1.
\end{equation}

\begin{lemma}\label{lem:distortionC1b}
Let $(\xi,x)\in \Sigma_N\times\bS^1$ and $n\in\bN$ such that for every $k=0,\ldots,n$ it holds
\[
	K_0^{-1}e^{k(\alpha-\varepsilon_E/4)}
	\le \lvert (f_\xi^k)'(x)\rvert
	\le K_0e^{k(\alpha+\varepsilon_E/4)}.
\]
Then, with $r$ satisfying \eqref{eq:rK1}, for every $y\in B(x,r)$ and $k=0,\ldots,n$ it holds
\[
	K_0^{-1}e^{k(\alpha-\varepsilon_E/2)}
	\le \lvert (f_\xi^k)'(y)\rvert
	\le K_0e^{k(\alpha+\varepsilon_E/2)}.
\]
\end{lemma}

\begin{proof}
The claim is true for $k=0$. By induction, suppose that the claim is true for $k$. Then $y\in B(x,r)$ satisfies 
\[
	\lvert f_\xi^k(y)-f_\xi^k(x)\rvert \le K_0 e^{k(\alpha+\varepsilon_E/2)} r<K_0r.
\]	 
The above choice of $r$ implies 
\[
	\lvert\log\lvert f_{\xi_{k}}'(f_\xi^k(y))\rvert - \log \lvert f_{\xi_{k}}'(f_\xi^k(x))\rvert\rvert
	<\frac{\varepsilon_E}{4},
\]	 
and hence the claim for $k+1$.
\end{proof}

\smallskip\noindent{\textbf{Fixing a covering by blending intervals.}}
Recalling Remark \ref{rem:commonblending}, fix 
\begin{equation}\label{eq:passaro}
	\delta\in\big(0,\min\{\frac r2,K_6/2\}\big)
\end{equation}	 
and take a cover of $\bS^1$ by finitely many intervals of the form $J_j=[y_j-2\delta,y_j+2\delta]$. Let
\[
	m_{\rm c}
	\eqdef\max_j m_{\rm c}(I_j),
	\spac{where}
	I_j\eqdef [y_j-\delta,y_j+\delta])
\]	 
and $m_{\rm c}(I_j)$ is as in Claim \ref{nclalem:connect}, and let
\[
 	L_1
	\eqdef \max_{j}L_1(F,J_j),
\]
where $L_1(F,J_j)$ is as in Definition \ref{defrem:commonblending}.

\smallskip\noindent{\textbf{Choice of further constants.}}
Let 
\[
	K
	\eqdef 
	K_0 \D^{m_{\rm c}}
	e^{-m_{\rm c}(\alpha+\varepsilon_E/2)}.
\]
Consider now $n_1\in\bN$  large enough such that
\begin{equation}\label{eq:choice1}\begin{split}
	&2rK_0\lVert F\rVert^{m_{\rm c}}e^{n_1(\alpha+\varepsilon_E/2)}
	<\delta,\\
	&\frac{1}{n_1}\log K
	<\frac{\varepsilon_E}{4},\quad
	\frac{1}{n_1}\log L_0
	<\frac{\varepsilon_H}{2}.
\end{split}\end{equation}

\smallskip\noindent{\textbf{Choice of the IFS.}}
Fix any integer $n\ge \max\{n_0,n_1\}$ let $\fX =\fX(n)=\{(\xi^i,x_i)\}_i$ be the set provided by Claim \ref{clalem:skeletonstar} so that for every $i$ and $k=1,\ldots,n$
\begin{equation}
	K_0^{-1}e^{k(\alpha-\varepsilon_E/4)}
	\le \lvert(f_{\xi^i}^k)'(x_i)\rvert
	\le K_0e^{k(\alpha+\varepsilon_E/4)}.
\end{equation}
By Lemma \ref{lem:distortionC1b}, for every $i$, $y\in B(x_i,r)$, and $k=1,\ldots,n$ it holds
\begin{equation}\label{eq:cheval}
	K_0^{-1}e^{k(\alpha-\varepsilon_E/2)}
	\le \lvert(f_{\xi^i}^k)'(y)\rvert
	\le K_0e^{k(\alpha+\varepsilon_E/2)}.
\end{equation}
Hence, in particular,
\begin{equation}\label{eq:cheval-b}
	\lvert f_{\xi^i}^n(B(x_i,r))\rvert
	\le 2r K_0e^{n(\alpha+\varepsilon_E/2)}.
\end{equation}

\smallskip\noindent{\textbf{Choice of a common blending interval.}}
Choose now an index $j$ for which
$
	N_j
	\eqdef \card\big(J_j\cap \{x_i\}_i\big)
$
is maximal and let $J=J_j$, $I=I_j$, and $N=N_j$. Observe that by Claim \ref{clalem:skeletonstar} (i) and the choice of $N$
\begin{equation}\label{eq:usingclaim78}
	L_0e^{n(h+\varepsilon_H/2)}
	\ge \card\fX
	\ge N
	\ge \frac{1}{2\delta}\cdot \card\fX
	\ge \frac{1}{2\delta}\cdot \frac{1}{L_0}e^{n(h-\varepsilon_H/2)}.
\end{equation}
We can, renumbering this set of points, assume that $x_1,\ldots,x_N\in J$.
By Claim \ref{nclalem:connect}, there are words $(\beta^i_1,\ldots,\beta^i_{s_i})$, $s_i\le m_{\rm c}$, such that $f_{[\xi^i_1,\ldots,\xi^i_n,\beta^i_1,\ldots,\beta^i_{s_i}]}(x_i)\in I$. 
Let now
\[
	\cB
	\eqdef \{ w_i\}_{i=1}^N,
	\quad\text{ where }\quad
	 w_i
	\eqdef (
		\xi^i_1,\ldots,\xi^i_n,\beta^i_1,\ldots,\beta^i_{s_i}).
\]

\begin{lemma}
	The collection of words    $\cB$ is disjoint.
\end{lemma}

\begin{proof}
	This is an immediate consequence of Claim \ref{clalem:skeletonstar} (ii).
\end{proof}

\noindent{\textbf{Checking properties of a CIFS with quantifiers.}}

\begin{lemma}
	The collection of words $\cB$ satisfies properties (a), (b), and (c) of a CIFS on $J$ relative to $K$, $\alpha+\varepsilon_E$, $\alpha$, and $\varepsilon_E$.
\end{lemma}

The following two claims prove the above lemma.

\begin{claim}
	Property (a) holds.
\end{claim}

\begin{proof}
Observe that \eqref{eq:cheval-b}, \eqref{eq:choice1}, and $n\ge n_1$ together imply
\[
	\lvert f_{[\xi^i_1,\ldots,\xi^i_n,\beta^i_1,\ldots,\beta^i_{s_i}]}(B(x_i,r))\rvert
	\le  \D^{m_{\rm c}}\cdot 2rK_0 e^{n(\alpha+\varepsilon_E/2)}
	<\delta.		
\]
Hence, together with \eqref{eq:passaro}, it follows
\[
	f_{[ w_i]}(J)
	= f_{[\xi^i_1,\ldots,\xi^i_n,\beta^i_1,\ldots,\beta^i_{s_i}]}(J)
	\subset f_{[\xi^i_1,\ldots,\xi^i_n,\beta^i_1,\ldots,\beta^i_{s_i}]}(B(x_i,r))
	\subset J,
\]
which gives property (a).
\end{proof}

Notice that for every $i$ it holds
\begin{equation}\label{eq:n1m}
	\lvert w_i\rvert
	= n+s_i
	> n_1.
\end{equation}

Using \eqref{eq:cheval} together with the estimates $s_i\le m_{\rm c}$ and \eqref{eq:choice1}, for every $y\in J$ and  $k=1,\ldots,\lvert w_i\rvert$ it holds
\begin{equation}\label{eq:seu1}\begin{split}
	\lvert (f_{[ w_i]}^k)'(y)\rvert
	&\le K_0 \D^{m_{\rm c}}e^{-m_{\rm c}(\alpha+\varepsilon_E/2)}
		\cdot e^{k(\alpha+\varepsilon_E/2)}
	\le Ke^{k(\alpha+\varepsilon_E/2)},
\end{split}\end{equation}
which is a first step towards proving (b) and also (c). 

\begin{claim}
	Properties (b) and (c) hold.	
\end{claim}

\begin{proof}
We first prove a slightly stronger version of property (c). Observe that  \eqref{eq:seu1} and \eqref{eq:choice1} together with \eqref{eq:n1m} imply
\begin{equation}\label{eq:theabove}
	\frac{1}{\lvert w_i\rvert}\log	\,\lvert (f_{[ w_i]})'(y)\rvert
	\le \frac{1}{n+s_i}
		\log K + \alpha + \frac12\varepsilon_E
	< \alpha + \frac34\varepsilon_E	,
\end{equation}
together with the analogous lower bound.

By \eqref{eq:seu1} property (b) holds for $m=1$. For $m\ge1$, let $ w_1,\ldots, w_m, w_{m+1}\in\cB$. For every $k\in\{\lvert( w_1,\ldots, w_m)\rvert+1,\ldots,\lvert ( w_1,\ldots, w_{m+1})\rvert\}$, by \eqref{eq:theabove} and \eqref{eq:seu1} it follows
\[\begin{split}
	\lvert (f_{[ w_1,\ldots , w_m, w_{m+1}]}^k)'(y)\rvert 
	&\le e^{\lvert w_1\rvert(\alpha+3\varepsilon_E/4)}\cdots
		e^{\lvert w_m\rvert(\alpha+3\varepsilon_E/4)}
		Ke^{(k-\lvert( w_1,\ldots, w_m)\rvert)(\alpha+\varepsilon_E/2)}\\
	&< Ke^{k(\alpha+\varepsilon_E)}.	
\end{split}\]
This proves property (b).
\end{proof}

\smallskip\noindent{\textbf{Cardinality of $\cB$.}}
Recall that $\max_{ w\in\cB}\lvert w\rvert>n$. By \eqref{eq:usingclaim78} and also using \eqref{eq:choice1} we obtain 
\[\begin{split}
	\log\card\cB
	= \log N
	&\le \log L_0 + n(h+\frac{\varepsilon_H}{2})\\
	&< \max_{ w\in\cB}\lvert w\rvert
		\Big(h+\frac{\varepsilon_H}{2}+\frac{1}{n}\log L_0\Big)
	< \max_{ w\in\cB}\lvert w\rvert(h+\varepsilon_H),
\end{split}\] 
the estimate from below is analogous, adapting the choice of $n$.
 This completes the proof of the theorem.
\end{proof}

\subsection{Distortion}\label{ssec:horsdist}

In what follows, given a function $\phi\colon \Sigma_N\times \bS^1 \to \bR$, for each $n\in\bN$ we denote by 
\[
	S_n\phi
	\eqdef \phi+\phi(F)+\ldots+\phi\circ F^{n-1}.
\]
the corresponding Birkhoff sum of  $\phi$ (relative to $F$). Denote $\lVert\phi\rVert\eqdef\sup\lvert\phi\rvert$.

\begin{proposition}\label{pro:distortion}
Let $\cB\subset\Sigma_N^\ast$ be a finite collection of words defining a CIFS on a compact interval $J\subset\bS^1$ relative to $K$, $\alpha_0$, $\alpha$, and $\varepsilon$. 
Then for every $\phi\colon \Sigma_N\times \bS^1 \to \bR$ continuous and $\tau>0$ there exists $N_1=N_1(\phi,\tau)\in\bN$ such that for every $m\ge N_1$ and finite sequence of concatenated words $( w_1,\ldots, w_m)\in\cB^m$ it holds
\[
	\max_{X,Y\in  \Sigma_N^- \times [ w_1,\ldots, w_m]^+
		\times J} 
		\big\lvert S_n\phi(X) - S_n\phi(Y)\big\rvert 
	< \tau n,
	\spac{where}
	n=\sum_{j=1}^m\lvert w_j\rvert.
\]
\end{proposition}

\begin{proof}
The function $\phi$ is uniformly continuous and hence there is $\delta>0$ so that at any pair of points in distance at most $\delta$ the values of $\phi$ differ at most by $\tau/2$. Fix $\ell,N_1\in\bN$ so that
\begin{equation}\label{eq:contra}
	\lvert J\rvert\cdot Ke^{\ell\alpha_0}
	\le \delta
,\quad	e^{-\ell}
	\le \delta
,\spac{and}	N_1
	> \max\big\{2\ell,\frac2\tau 
		\cdot 4\ell\lVert\phi\rVert\big\}.
\end{equation}

Fix $m\ge N_1$ and let $( w_1,\ldots, w_m)\in\cB^m$ and  $n=\sum_{j=1}^m \lvert w_j\rvert\ge N_1$.
Consider $H= \Sigma_N^- \times [ w_1 ,\ldots,  w_m]^+\times J$ and observe that $H$ is a cartesian product of $\Sigma_N^-$, a cylinder of level $n$, and the interval $J$. Hence, recalling that $F$ is a step skew product, for every $j=0,\ldots, n$, the image $F^j(H)$ is a cartesian product of  a cylinder of level $j$,  a cylinder of level $n-j$, and an interval. 

For the course of this proof, denote by $\pi_k$, $k=1,2,3$, the projection to the $k$th component of the product space $\Sigma_N^-\times\Sigma_N^+\times\bS^1$.  
As by property (b) of a CIFS, every map $f_{[ w_j]}$ is a contraction, together with \eqref{eq:contra} it follows 
\[
	\lvert\pi_3(F^j(H))\rvert
	\le\delta
	\quad\text{ for all }j= \ell,\ldots,n.
\]	
Recall the metrics $d^\pm$ on $\Sigma_N^\pm$ defined in Section \ref{ssec:codedSigma}. Note that for every $j=0,\ldots,n$ 
\[
	\diam_{d^-}(\pi_1(F^j(H)))
	\le e^{-j}
	\quad\text{ and }\quad
	\diam_{d^+}(\pi_2(F^j(H)))
	\le e^{-n+j}.
\]
Together with \eqref{eq:contra} it then follows that for every $j=\ell,\ldots,n-\ell$ it holds
\[
	\diam_{d^-}(\pi_1(F^j(H)))
	\le\delta,\quad
	\diam_{d^+}(\pi_2(F^j(H)))
	\le\delta,\quad
	\lvert\pi_3(F^j(H))\rvert
	\le\delta.
\]
Thus, for every $X,Y\in H$ we obtain
\[\begin{split}
	\lvert S_n\phi(X)-S_n\phi(Y)\rvert
	&\le \ell2\lVert\phi\rVert + (n-2\ell)\frac\tau2 
		+ \ell2\lVert\phi\rVert
	< n\frac\tau2 +4\ell\lVert\phi\rVert	\\
	{\tiny\text{using \eqref{eq:contra} }}\quad
	&< n\frac\tau2 
		+ \frac\tau2 N_1
	\le n\frac\tau2+\frac\tau2 n
	= \tau n.
\end{split}\]
This finishes the proof.
\end{proof}

\subsection{Horseshoes associated to CIFSs}\label{sec:horrrrse}

For every CIFS $\cB\subset\Sigma_N^\ast$, Proposition \ref{pro:defineshorseshoe} asserts the existence of its associated attractor $\Lambda(\cB)\subset\Sigma_N\times\bS^1$. Moreover, if $\cB$ is a CIFS on an interval $J$ then 
\begin{equation}\label{eq:msoqui1}	
	\Lambda(\cB)
	=\Pi_\cB(\cB^\bZ)
	\subset \PCS{\cB}\times J
	\subset\Sigma_N\times J.
\end{equation}

\begin{proposition}[Horseshoe induced by a CIFS]\label{procor:specCIFS}
	Let $\cB=\{ w_1,\ldots, w_M\}\subset\Sigma_N^\ast$ be a finite disjoint collection of words defining a CIFS on an interval $J$ relative to $K,\alpha_0,\alpha$, and $\varepsilon$ and $\Lambda(\cB)$ its associated attractor. Let
\[
	\Gamma(\cB)
	\eqdef \bigcup_{k=0}^{R-1}F^k(\Lambda(\cB)),
	\spac{where}
	R\eqdef\max_{ w\in\cB}\,\lvert w\rvert.
\]
Then $\Gamma(\cB)$ is a compact $F$-invariant set such that every ergodic Borel probability measure $\mu'\in\cM(F|_{\Gamma(\cB)})$ satisfies
\[
	\chi(\mu')
	\in (\alpha-\varepsilon,\alpha+\varepsilon).
\]
\end{proposition}

\begin{proof}
By Proposition \ref{pro:defineshorseshoe}, $\Lambda(\cB)$ is the image of a compact set under a continuous map, and hence compact. The semi-conjugation  in Proposition \ref{pro:defineshorseshoe} implies that $\Gamma(\cB)$ is $F$-invariant. 

Hence, the property of the range of Lyapunov exponents is an immediate consequence of property (c) of a CIFS and the fact that the orbit of every point generic for an $F$-ergodic measure is described by an infinite concatenation of fiber maps $f_{[ w]}$ with $ w\in\cB$.
\end{proof}

\begin{remark}[Horseshoes]
The set $\Gamma(\cB)$ can be seen as a $F$-invariant multi-variable-time horseshoe as in \cite[Section 5]{DiaGelRam:17}. For simplicity, we will refer to such sets simply as \emph{horseshoes}.
\end{remark}

\section{Repetition and tailing scheme}\label{ssec:horsreptai}

In this section, we introduce the repeat-and-tail scheme which will provide us a cascade of collections of words $\cB$ over the alphabet $\{1,\ldots,N\}$.  By writing $\lvert w\rvert$ for some $ w\in\cB$ we always mean its length as spelled in $\{1,\ldots,N\}$.

\begin{definition}\label{def:reptaiSFT}
Let $\cB\subset\Sigma_N^\ast$ be a collection of (nonempty) words. Given $m\in\bN$, consider a \emph{tailing map} $\ft=\ft_{\cB,m}\colon\cB^m\to\Sigma_N^\ast$ and define 
\[
	(\cB^m)_\ft
	\eqdef \big\{\big( w_1,\ldots, w_m,
		\ft( w_1,\ldots, w_m)\big)
		\colon  w_k\in\cB
		\text{ for every }k=1,\ldots,m\big\}
	\subset\Sigma_N^\ast.
\]
We say that $(\cB^m)_\ft$ \emph{$m$-times repeats and $\ft$-tails} $\cB$. We define the \emph{tail-adding map} 
\[
	T_{(\cB^m)_\ft}\colon\cB^m\to{(\cB^m)_\ft}
,\quad
	T_{(\cB^m)_\ft}( w_1,\ldots, w_m)
	\eqdef \big( w_1,\ldots, w_m,		
		\ft( w_1,\ldots, w_m)\big).
\]
\end{definition}

Recall that in the above definition we use our simplifying Notation \ref{eq:notation}. 
We point out that the words in $\cB$ may have different length. The same applies to words in $(\cB^m)_\ft$. The following is an immediate consequence of Corollary \ref{corlem:dis3} and Lemma \ref{lem:COND}.

\begin{corollary}\label{cor:CONreptai}
	Let $\cB\subset\Sigma_N^\ast$ be a finite collection of words which is disjoint. Let $m\in\bN$ and consider a tailing map $\ft=\ft_{\cB,m}\colon\cB^m\to\Sigma_N^\ast$. Then  $T_{(\cB^m)_\ft}$ is bijective. Moreover, $(\cB^m)_\ft$ is disjoint and hence uniquely left decipherable. 
\end{corollary}

The next theorem is a key ingredient. It provides a choice of CIFS's (and hence of the associated attractors and the horseshoes they generate) whose Lyapunov exponent drops by a controlled amount. The estimate on the length of the tails also allows to control the drop of  entropy of the horseshoes. 

\begin{theorem}[Choice of a tailing map]\label{thepro:tailing}
	Consider $F\in \mathrm{SP}^1_{\rm shyp}(\Sigma_N\times\bS^1)$, $N\ge2$. 
	Let $J\subset\bS^1$ be a blending interval. 
	Let $\cB$ be a finite disjoint collection of words defining a CIFS on   $J$ relative to  $K\ge1,\alpha_0=\alpha+\varepsilon<0$, $\alpha<0$, and $\varepsilon$, for some $\varepsilon \in(0,\lvert\alpha\rvert/2)$.
	There is $N_2=N_2(\cB)\in\bN$ such that for every $m\ge N_2$ there exists a tailing map $\ft=\ft_{\cB,m}\colon\cB^m\to\Sigma_N^\ast$ such that the $m$-times repeated and $\ft$-tailed collection of words $(\cB^m)_\ft$ defines a CIFS on $J$ relative to  $K,\alpha_0'$,  $\alpha'$, and $\varepsilon'$, where
\[
	\alpha_0'=\frac12(\alpha+\varepsilon),\quad
	\alpha'=\frac12\alpha,\quad
	\varepsilon'=\frac\varepsilon2.
\] 
Moreover, the tailing map satisfies for every $ w_1,\ldots, w_m\in\cB$
\begin{equation}\label{eq:lengthoftailings}
	\lvert\ft( w_1,\ldots, w_m)\rvert
	\le L_1\sum_{j=1}^m\lvert w_j\rvert\lvert\alpha\rvert,
\end{equation}
where $L_1=L_1(F,J)>0$ is as in Definition \ref{defrem:commonblending}.
\end{theorem}

\begin{proof} Similarly to the proof of Theorem \ref{teo:existenceCIFS}, 
we consider the  constants $K_1,\ldots$, $K_5$ associated to the blending interval $[x-2\delta,x+2\delta]=J$. Let $I=[x-\delta,x+\delta]$ and $m_{\rm c}=m_{\rm c}(I)$ as in Claim \ref{nclalem:connect}. Recall that
\begin{equation}\label{eq:L1recall}
	L_1=
	L_1(F,J)=
	K_2(2+\lvert\log(4\delta)\rvert+K_3)+m_{\rm c}.
\end{equation}

\smallskip\noindent{\textbf{Choice of quantifiers.}}
Choose $r>0$ such that  
\begin{equation}\label{eq:choicer}
	r<\min\{K_1,K_4,\D^{-m_{\rm c}}\delta\},
	\spac{and}
	\Mod_F(r)<\frac\varepsilon8,
\end{equation}
where $\D$ is defined in \eqref{def:cunifconst} and $\Mod_F$ in \eqref{eq:defMod}.

\begin{claim}\label{clai:distoortion}
	For $N_2\in\bN$ sufficiently large, every $m\ge N_2$, and $ w_1,\ldots, w_m\in\cB$  it holds
\[		
	\log \sup_{x,y\in J}
		\frac{\lvert (f_{[ w_1,\ldots, w_m]})'(x)\rvert}
			{\lvert (f_{[ w_1,\ldots, w_m]})'(y)\rvert}
			\le \sum_{j=1}^m\lvert w_j\rvert\frac\varepsilon8.
\]			
\end{claim}

\begin{proof}
	Since $\cB$ is a CIFS on $J$, applying repeatedly its maps to $J$ shrinks this interval exponentially fast, which implies that the modulus of continuity of $f_{[ w_j]}$ on the corresponding image interval also decreases. Then
\[
	\max_{x,y\in J}
		\log\frac{\lvert (f_{[ w_1,\ldots, w_m]})'(x)\rvert}
			{\lvert (f_{[ w_1,\ldots, w_m]})'(y)\rvert}
	\le \sum_{j=1}^{m-1}
	\max_{x,y\in f_{[ w_1,\ldots, w_j]}(J)}
		\log\frac{\lvert (f_{[ w_{j+1}]})'(x)\rvert}
			{\lvert (f_{[ w_{j+1}]})'(y)\rvert}		,	
\]	
where the latter is a finite sum of terms converging to zero. This implies the claim.
\end{proof}

We also assume the following properties to be satisfied for $N_2$:
\begin{equation}\label{eq:formula-1nnnmao}\begin{split}
	&\lvert J\rvert 
	\max\Big\{
		e^{N_2 \min_{ w\in\cB}\lvert w\rvert(\alpha+\varepsilon)},
		K e^{N_2\frac12\min_{ w\in\cB}\lvert w\rvert\alpha}
		\Big\}
	< r,\\
	&1\le N_2\lvert\alpha\rvert,\\
	&\frac{1}{N_2}\log\D^{1+m_{\rm c}}	
	< \frac\varepsilon8	.	
\end{split}\end{equation}
In the following, let
\[
	m\ge N_2.
\]

\smallskip\noindent{\textbf{Length of iterates of $J$.}}
Fix some enumeration $\cB=\{ w_1,\ldots, w_M\}$. Given $(i_1,\ldots ,i_m)\in \{1,\ldots,M\}^m$, by property (a) of the CIFS on $J$ defined by $\cB$, it holds
\[
	H_{i_1,\ldots,i_m}
	\eqdef  f_{[ w_{i_1},\ldots, w_{i_m}]}(J)
	\subset J.
\]
Property (c) of the CIFS first implies that for every $y\in J$ it holds
\[
	e^{\sum_{j=1}^m\lvert w_{i_j}\rvert(\alpha-\varepsilon)}
	\le \lvert (f_{[ w_{i_1},\ldots, w_{i_m}]})'(y)\rvert
	\le e^{\sum_{j=1}^m\lvert w_{i_j}\rvert(\alpha+\varepsilon)}
\]
and hence
\begin{equation}\label{eq:forrr}\begin{split}
	\lvert J\rvert e^{\sum_{j=1}^m\lvert w_{i_j}\rvert(\alpha-\varepsilon)}
	\le \lvert H_{i_1,\ldots,i_m}\rvert
	&= \lvert f_{[ w_{i_1},\ldots, w_{i_m}]}(J)\rvert
	\le \lvert J\rvert  e^{\sum_{j=1}^m\lvert w_{i_j}\rvert(\alpha+\varepsilon)}\\
		{\tiny\text{by  \eqref{eq:formula-1nnnmao} and $m\ge N_2$}}\quad
	&\le \lvert J\rvert  e^{m\min_{ w\in\cB}\lvert w\rvert(\alpha+\varepsilon)}
	< r.
\end{split}\end{equation}

\smallskip\noindent{\textbf{Definition of the tailing map.}}
Given $(i_1,\ldots ,i_m)\in \{1,\ldots,M\}^m$, it holds $H_{i_1,\ldots,i_m}\subset J$. By Axiom CEC$+(J)$, there exists a finite expanding and covering sequence $(\eta_1,\ldots,\eta_{L})$ such that
\[
	f_{[\eta_1,\ldots,\eta_{L}]}(H_{i_1,\ldots,i_m})
	\supset B( J,K_4),
\]
where $L\in\bN$ satisfies
\begin{equation}\label{eq:bach}\begin{split}
	L
	&\le K_2\lvert\log\lvert H_{i_1,\ldots,i_m}\rvert\rvert+K_3\\
	{\tiny\text{using \eqref{eq:forrr} }}\quad
	&\le K_2\sum_{j=1}^m\lvert w_{i_j}\rvert\lvert\alpha-\varepsilon\rvert	
		+K_2\lvert\log\lvert J\rvert\rvert +K_3\\
	{\tiny\text{using $\lvert\alpha-\varepsilon\rvert=\lvert\alpha\rvert+\varepsilon<2\lvert\alpha\rvert$}}\quad	
	&\le \sum_{j=1}^m\lvert w_{i_j}\rvert \lvert\alpha\rvert
		\left(2K_2	
		+\frac{1}{m\lvert\alpha\rvert}(K_2\lvert\log\lvert J\rvert\rvert +K_3)\right).
\end{split}\end{equation}

In the following, instead of ``going all the way'' to cover the blending interval $J$, we will only consider a certain truncated sequence $(\eta_1,\ldots,\eta_{\ell})$ for some $\ell\le L$. Indeed, choose $\ell\in\{1,\ldots,L\}$ to be the smallest number satisfying
\begin{equation}\label{eq:maaximal}
	\lvert f_{[\eta_1,\ldots,\eta_{\ell}]}(H_{i_1,\ldots,i_m})\rvert
	\ge \lvert J\rvert e^{\frac12\sum_{j=1}^m\lvert w_{i_j}\rvert\alpha}.
\end{equation}
Note that the estimate in \eqref{eq:forrr} together with $\alpha+\varepsilon<\frac12\alpha<0$ implies $\ell\ge1$.

By Claim \ref{nclalem:connect}, there exists a finite sequence $(\beta_1,\ldots,\beta_s)$, $s\le m_{\rm c}$, such that 
\begin{equation}\label{eq:Hintersect}
	(f_{[\beta_1,\ldots,\beta_s]}\circ f_{[\eta_1,\ldots,\eta_\ell]})( H_{i_1,\ldots,i_m})
	\cap I
	\ne\emptyset.
\end{equation}	 
Define now the tailing map 
\[
	\ft\colon\cB^m\to\Sigma_N^\ast,\quad
	\ft( w_{i_1},\ldots, w_{i_m})
	\eqdef (\eta_1,\ldots,\eta_\ell,\beta_1,\ldots,\beta_s).
\]
Using \eqref{eq:formula-1nnnmao} we have $m\lvert\alpha\rvert\ge N_2\lvert\alpha\rvert\ge1$.
Observing that $\ell\le L$, together with \eqref{eq:bach} it hence follows
\[\begin{split}
	\ell + s
	&\le \sum_{j=1}^m\lvert w_{i_j}\rvert\lvert\alpha\rvert
		\left(2K_2 
		+\frac{1}{m\lvert\alpha\rvert}(K_2\lvert\log\lvert J\rvert\rvert+K_3) 
		\right)
	+m_{\rm c}	
	\le L_1\sum_{j=1}^m\lvert w_{i_j}\rvert\lvert\alpha\rvert,
\end{split}\]
where $L_1$ as in \eqref{eq:L1recall}.
This implies property \eqref{eq:lengthoftailings}.

What remains to prove is that
\[
	(\cB^m)_\ft
	\eqdef \big\{\big( w_{i_1},\ldots, w_{i_m}	,
			\ft( w_{i_1}\ldots w_{i_m})\big)\colon
		( w_{i_1},\ldots, w_{i_m})\in\cB^m\big\}.
\]
 defines a CIFS on $J$ with the claimed quantifiers.

\smallskip\noindent{\textbf{Checking properties of a CIFS with quantifiers.}}

\begin{lemma}
	The collection of words $(\cB^m)_\ft$ satisfies properties (a), (b), and (c) of a CIFS on $J$ relative to $K$, $\frac12(\alpha+\varepsilon)$, $\frac12\alpha$, and $\frac12\varepsilon$.
\end{lemma}

We split the proof of this lemma into claims.

\begin{claim}
	Property (a) holds.
\end{claim}

\begin{proof}
By the choice of $\ell$ in \eqref{eq:maaximal}, for every $k=1,\ldots,\ell-1$ it holds
\begin{equation}\label{eq:choicer2}
	\lvert f_{[\eta_1,\ldots,\eta_k]}( H_{i_1,\ldots,i_m})\rvert
	< \lvert J\rvert e^{\frac12\sum_{j=1}^m\lvert w_{i_j}\rvert\alpha}
	<r,
\end{equation}
where for the latter inequality we used \eqref{eq:formula-1nnnmao} and $m\ge N_2$.
This together with $s\le m_{\rm c}$ and \eqref{eq:choicer} implies
\[
	\lvert (f_{[\beta_1,\ldots,\beta_s]}\circ
		f_{[\eta_1,\ldots,\eta_\ell]})( H_{i_1,\ldots,i_m})\rvert
	\le r \D^{m_{\rm c}}
	< \delta.
\]
Hence, by the intersection property in \eqref{eq:Hintersect} and the fact that $I=[x-\delta,x+\delta]\subset J=[x-2\delta,x+2\delta]$, it follows
\[
	(f_{[\beta_1,\ldots,\beta_s]}\circ f_{[\eta_1,\ldots,\eta_\ell]})( H_{i_1,\ldots,i_m})
	\subset J.
\]
This implies the claim.
\end{proof}

\begin{claim}
	Property (b) holds.
\end{claim}

\begin{proof}
By our choice of $\ell$ in \eqref{eq:maaximal}, there exists $y\in J$ such that
\begin{equation}\label{eq:disss}\begin{split}
	\lvert (f_{[\eta_1,\ldots,\eta_{\ell}]}\circ f_{[ w_{i_1},\ldots, w_{i_m}]})'(y)\rvert
	&\ge e^{\frac12\sum_{j=1}^m\lvert w_{i_j}\rvert\alpha}.
\end{split}\end{equation}
As $\ell$ is minimal satisfying \eqref{eq:maaximal},  for every $k=1,\ldots,\ell$ there exists $z_k\in J$ so that
\begin{equation}\label{eq:gelad}
	\lvert (f_{[\eta_1,\ldots,\eta_k]}\circ f_{[ w_{i_1},\ldots, w_{i_m}]})'(z_k)\rvert
	< \D e^{\frac12\sum_{j=1}^m\lvert w_{i_j}\rvert\alpha}.
\end{equation}
Using \eqref{eq:choicer2}, by the choice of $r$ in \eqref{eq:choicer}, for every $k=1,\ldots,\ell$
\begin{equation}\label{eq:disss2}
	\log \sup_{x,y\in H_{i_1,\ldots,i_m}}
		\frac{\lvert (f_{[\eta_1,\ldots,\eta_{k}]})'(x)\rvert}
						{\lvert (f_{[\eta_1,\ldots,\eta_{k}]})'(y)\rvert}
			\le k \frac\varepsilon8.
\end{equation}
Hence, together with  \eqref{eq:gelad}, for every $z\in J$ and $k=1,\ldots,\ell$ it holds
\begin{equation}\label{eq:below1}
	\lvert (f_{[\eta_1,\ldots,\eta_k]}\circ f_{[ w_{i_1},\ldots, w_{i_m}]})'(z)\rvert
	  \le \D  e^{\frac12\sum_{j=1}^m\lvert w_{i_j}\rvert\alpha}
		\cdot e^{k\varepsilon/8}.
\end{equation}
On the other hand, by \eqref{eq:disss2} and \eqref{eq:disss} and also distortion Claim \ref{clai:distoortion},  for every $z\in J$ and  $k=\ell$
\begin{equation}\label{eq:below}
	  e^{-\sum_{j=1}^m\lvert w_{i_j}\rvert\varepsilon/8}
	  e^{-\ell\varepsilon/8}
	\le \frac{\lvert (f_{[\eta_1,\ldots,\eta_{\ell}]}\circ 
		f_{[ w_{i_1},\ldots, w_{i_m}]})'(z)\rvert}
		{e^{\frac12\sum_{j=1}^m\lvert w_{i_j}\rvert\alpha}}
	\le \D\cdot e^{\ell\varepsilon/8}.
\end{equation}
Further for every $x\in J$ and $k=1,\ldots,\ell$, using \eqref{eq:below1}, it holds
\begin{equation}\label{Aprime3}\begin{split}
	\frac{1}{\sum_{j=1}^m\lvert w_{i_j}\rvert+k}
		&\log\,\lvert(f_{[\eta_1,\ldots,\eta_k]}\circ 
			f_{[ w_{i_1},\ldots, w_{i_m}]})'(x)\rvert\\
	&\le
	 \frac{1}
		{\sum_{j=1}^m\lvert w_{i_j}\rvert+k}
	\left(\log\, \D
		+\sum_{j=1}^m\lvert w_{i_j}\rvert
		\frac12\alpha+k\frac\varepsilon8\right)\\
		{\tiny{\text{by \eqref{eq:formula-1nnnmao} and $m\ge N_2$}}}\quad	
	&<\frac\varepsilon8+\frac12\alpha + \frac\varepsilon8
	< \frac12\alpha+\frac\varepsilon2.
\end{split}\end{equation}
Moreover, for every $x\in J$ and $k=1,\ldots,s$ it holds
\begin{equation}\label{Aprime2}\begin{split}
	\frac{1}{\sum_{j=1}^m\lvert w_{i_j}\rvert+\ell+k}
		&\log\,\lvert(f_{[\beta_1,\ldots,\beta_k]}\circ 
					f_{[\eta_1,\ldots,\eta_\ell]}\circ 
					f_{[ w_{i_1},\ldots, w_{i_m}]})'(x)\rvert	\\
		{\tiny{\text{using \eqref{eq:below1} for $\ell$ }}}\quad	
	&\le 	\frac{1}{\sum_{j=1}^m\lvert w_{i_j}\rvert+\ell+k}
		\log\D^{1+k}			
	+ \frac12\alpha +\frac\varepsilon8\\
		{\tiny{\text{using \eqref{eq:formula-1nnnmao} and $m\ge N_2$ }}}\quad	
	&<\frac\varepsilon8 +\frac12\alpha +\frac\varepsilon8
	< \frac12\alpha+\frac\varepsilon2 .
\end{split}\end{equation}
Hence,  the hypothesis on $\cB$ defining a CIFS with quantifiers, \eqref{Aprime3}, and \eqref{Aprime2} together imply property (b).
\end{proof}

\begin{claim}
	Property (c) holds.
\end{claim}

\begin{proof}
The upper bound for the spectrum follows from \eqref{Aprime2}. It remains to prove the lower one. Note that \eqref{eq:disss} together with \eqref{eq:below} implies
\[
\begin{split}
	&\frac{1}{\sum_{j=1}^m\lvert w_{i_j}\rvert+\ell+s}
	\log\,\lvert(f_{[\beta_1,\ldots,\beta_s]}\circ 
					f_{[\eta_1,\ldots,\eta_\ell]}\circ 
					f_{[ w_{i_1},\ldots, w_{i_m}]})'(x)\rvert\\
	&\ge \frac{1}{\sum_{j=1}^m\lvert w_{i_j}\rvert+\ell+s}
	\left(\log\,\D^{-m_{\rm c}}		
		+ \frac12\sum_{j=1}^m\lvert w_{i_j}\rvert\alpha
		-\ell\frac\varepsilon8	
		- \sum_{j=1}^m\lvert w_{i_j}\rvert
			\frac\varepsilon8
	\right)	\\
	&> \frac{1}{\sum_{j=1}^m\lvert w_{i_j}\rvert+\ell+s}
	\log\,\D^{-m_{\rm c}}		
		+\frac{\sum_{j=1}^m\lvert w_{i_j}\rvert}
			{\sum_{j=1}^m\lvert w_{i_j}\rvert+\ell+s}
		\left(\frac12\alpha-\frac\varepsilon8\right)
		-\frac\varepsilon8\\
			{\tiny{\text{using \eqref{eq:formula-1nnnmao} }}}\quad	
	&>	-\frac\varepsilon8+\frac12\alpha-\frac\varepsilon4
	> \frac12\alpha-\frac\varepsilon2,
\end{split}\]
proving the claim.
\end{proof}

This finishes the proof of the theorem.
\end{proof}

\section{Cascades of horseshoes}\label{sec:ncashor}

Throughout this section, consider $F\in \mathrm{SP}^1_{\rm shyp}(\Sigma_N\times\bS^1)$, $N\ge2$. 
Let $\mu$ be an ergodic measure with Lyapunov exponent $\alpha=\chi(\mu)<0$ and entropy $h=h(F,\mu)>0$. 
Fix $\varepsilon\in(0,\lvert\alpha\rvert/4)$ and $\varepsilon_H\in(0,h)$. Let $J\subset\bS^1$ be a blending interval and $\cB\subset\Sigma_N^\ast$ be a finite disjoint collection of words as provided by Theorem \ref{teo:existenceCIFS}, defining a CIFS on $J$ relative to some constant $K>1$ and $\alpha_0=\alpha+\varepsilon$, $\alpha$, and $\varepsilon$.  
In particular, it holds
\begin{equation}\label{eq:estentropyy}
	(h(F,\mu)-\varepsilon_H)\min_{ w\in\cB}\lvert  w\rvert
	\le \log\card \cB
	\le (h(F,\mu)+\varepsilon_H)\max_{ w\in\cB}\lvert  w\rvert.
\end{equation}
Let $L_1=L_1(F,J)$ as in Definition \ref{defrem:commonblending}.

In Section \ref{ssec:cascadehorsese}, we construct two cascades of alphabets $(\cA_n)_n$ and $(\cB_n)_n$. Every alphabet $\cB_n$ is formed by words in $\{1,\ldots,N\}$ and obtained from the previous one $\cB_{n-1}$ by the repeat-and-tail scheme with tailing functions $\ft_n$ as in Theorem \ref{thepro:tailing}. Moreover, every $\cB_n$ defines a CIFS with associated attractor $\Lambda_n=\Lambda(\cB_n)$  which in turn generates a horseshoe $\Gamma_n$ (as in Propositions \ref{pro:defineshorseshoe} and \ref{procor:specCIFS}). Each alphabet $\cA_n$ is the abstract companion of $\cB_n$ and gives rise to a suspension space for an appropriate roof function $R_n\colon \cA_n\to\bN$, see \eqref{eq:bug}. 
Note that all these objects depend on $\mu$. 

In Section \ref{subsec:suspen}, we see that those horseshoes are factors of the suspension spaces. Using the latter, in Section \ref{subsec:entropiiie} we obtain estimates of entropy and exponents of the horseshoes. We conclude this section by proving Proposition \ref{pro:main}.

\subsection{Construction of a cascade of horseshoes}\label{ssec:cascadehorsese}

In the following we introduce the two cascades $(\cA_n)_n$ and $(\cB_n)_n$ of alphabets. Our scheme is fairly general and only requires $\cB_0$ and an initially fixed sequence $(m_n)_n$. We always denote by $\lvert\cdot\rvert$ the length of the corresponding word spelled in $\{1,\ldots,N\}$.

\subsubsection{Inductive definition of alphabets}

We proceed inductively.  
Let $\cB_0\eqdef\cB$, $M_0\eqdef\card\cB_0$, and write $\cB_0=\{ w^{(0)}_1,\ldots, w^{(0)}_{M_0}\}$. By hypothesis in the beginning of this section, $\cB_0$ defines a CIFS on $J$ relative to $K,\alpha_0=\alpha+\varepsilon,\alpha,$ and $\varepsilon$. Let
\[
	\cA_0
	\eqdef \cB_0
\]
For $n\ge1$, assume that there is a finite disjoint collection of words
\[
	\cB_{n-1}
	=\{ w^{(n-1)}_1,\ldots, w^{(n-1)}_{M_{n-1}}\}\subset\Sigma_N^\ast,
\]	 
which defines a CIFS on $J$ relative to $K$, $2^{-(n-1)}\alpha_0$, $2^{-(n-1)}\alpha$, and $2^{-(n-1)}\varepsilon$, a collection $\cA_{n-1}$, and $m_{n-1}\in\bN$. Let $N_2=N_2(\cB_{n-1})\in\bN$ as in Theorem \ref{thepro:tailing} and choose $m_n\ge N_2$ with $m_n\ge m_{n-1}$. Consider the tailing map $\ft_n=\ft_{\cB_{n-1},m_n}\colon(\cB_{n-1})^{m_n}\to\Sigma_N^\ast$ as provided in that theorem and, recalling Definition \ref{def:reptaiSFT}, denote by 
\[
	\cB_n
	\eqdef (\cB_{n-1}^{m_n})_{\ft_n}
\]	
 the collection which $m_n$-times repeats and $\ft_n$-tails $\cB_{n-1}$.
Let
\[
	\cA_n\eqdef(\cA_{n-1})^{m_n}.
\] 
 Note that for every $n\in\bN$
\begin{equation}\label{eq:cardAn}
	\card\cA_n
	= M_n
	= m_n\cdot m_{n-1}\cdots m_0.
\end{equation} 
This concludes the inductive description of the alphabets.
 
\subsubsection{Dictionaries}

 Let us point out natural ``dictionaries'' between $(\cB_n)_n$ and $(\cA_n)_n$ and the corresponding cascade of sequence spaces. For every $n\in\bN$, $\cB_n$ and $\cA_n$ have the same cardinality, the former is a collection of words that almost coincide with the words in $\cA_n$ up to a cascade of tails which were introduced at every intermediate level. Recursively, we define a bijection between each such pair of collections:
\begin{itemize}
\item $\Cut_0$ is the identity on $\cA_0=\cB_0=\cB$,
\item for every $ w = ( w_{i_1}^{(n-1)},\ldots, w_{i_{m_n}}^{(n-1)},
			\ft_n( w_{i_1}^{(n-1)},\ldots, w_{i_{m_n}}^{(n-1)}))\in\cB_n$, let
\[
	\Cut_n\big( w\big)
	\eqdef \big(\Cut_{n-1}( w_{i_1}^{(n-1)}),
			\ldots,\Cut_{n-1}( w_{i_{m_n}}^{(n-1)})\big).
\]
\end{itemize}
The map $\Cut_n$ ``cuts out any tail" which was added in the definitions of $\cB_1,\ldots,\cB_n$. 
We let
\begin{equation}\label{secn}\begin{split}
	&\underline \Cut_n\colon(\cB_n)^\bZ\to(\cA_n)^\bZ,\\
	&\underline \Cut_n(\ldots, w_{-1}| w_0, w_1,\ldots)
	\eqdef (\ldots,\Cut_n( w_{-1})|\Cut_n( w_0),\Cut_n( w_1),\ldots).
\end{split}\end{equation}

To prove the next lemma just note that both alphabets have the same cardinality.

\begin{lemma}\label{lem:gnconjugates}
For every $n\in\bN$,  the maps $\sigma_{\cB_n}$ on $(\cB_n)^\bZ$ and $\sigma_{\cA_n}$ on $(\cA_n)^\bZ$ are topologically conjugate by $\underline \Cut_n$.
\end{lemma}

\subsubsection{Definition and control of roof functions}

For every $n\in\bN_0$ define the roof function
\begin{equation}\label{eq:bug}
	R_n\colon\cA_n\to\bN,\quad
	R_n(a)
	\eqdef \left\lvert \Cut_n^{-1}(a)\right\rvert,
\end{equation}
where the length is considered in $\Sigma_N^\ast$ identifying each concatenation of words with the corresponding word spelled in $\{1,\ldots,N\}$. As before, we also consider the associated map $\underline R_n\colon(\cA_n)^\bZ\to\bN$.

The next corollary estimates the lengths of the (inductively defined) tails added in each step. It is an immediate consequence of Theorem \ref{thepro:tailing} and Corollary \ref{cor:CONreptai}.

\begin{corollary}[Control of tail-lengths]\label{cor:cons1}
For every $n\in\bN$, $\cB_n\subset\Sigma_N^\ast$ is a finite disjoint collection of words which defines a CIFS on $J$ relative to $K$, $2^{-n}\alpha_0$, $2^{-n}\alpha$, and $2^{-n}\varepsilon$. Moreover,
\[
	\lvert\ft_n( w^{(n-1)}_{i_1},\ldots, w^{(n-1)}_{i_{m_n}})\rvert
	\le L_1\frac{1}{2^{n-1}}\sum_{j=1}^{m_n}
		\lvert w^{(n-1)}_{i_j}\rvert\lvert\alpha\rvert.
\]
\end{corollary}

The following corollary puts the above bounds on the tailing map  into the context of roof functions in our abstract model suspension spaces. 

\begin{corollary}[Estimates on roof functions]\label{newcor:notormenta}
	The associated  family of roof functions $(R_n)_n$ satisfies Assumption \ref{ass:roof}  with $K=L_1\lvert\alpha\rvert$.	
\end{corollary}
\begin{proof}
The first inequality in  Assumption \ref{ass:roof} holds true by construction. 
Recalling that $\underline\Sub_{n,n-1}$ denoted the substitution map from $(\cA_n)^\bZ$ to $(\cA_{n-1})^\bZ$ as defined in \eqref{eq:tailaddn-1n}, from Corollary \ref{cor:cons1} it follows that
\[
	\underline R_n\circ \underline\Sub_{n,n-1}^{-1}
	\le (1+L_1\frac{1}{2^{n-1}}\lvert\alpha\rvert)
		\sum_{j=0}^{m_n-1}\underline R_{n-1}\circ\sigma_{\cA_{n-1}}^j,
\]
which implies the second inequality taking $K=L_1\lvert\alpha\rvert$.  
\end{proof}

\subsubsection{Inductive definition of horseshoes}\label{ssec:inddefhor}

For every $n\in\bN$, let
\begin{equation}\label{eq:defGammanLambdan}
	\Lambda_n
	\eqdef \Lambda(\cB_n)
	\subset \Sigma_N\times J
	\spac{and}
	\Gamma_n
	\eqdef \Gamma(\cB_n)
	\supset\Lambda_n
\end{equation}
be as in Propositions \ref{pro:defineshorseshoe} and \ref{procor:specCIFS}, respectively. We can view each set $\Lambda_n$ as the ``ground floor''  of the horseshoe $\Gamma_n$ (the reason for this notation will become clear thereafter). Recall that $(\underline w,x)\mapsto F^{\lvert w_0\rvert}(\underline w,x)$ is a return map on $\Lambda_n$. Moreover, the map 
\begin{equation}\label{eq:agrree}
	\Pi_n\eqdef\Pi_{\cB_n}\colon (\cB_n)^\bZ\to\Lambda_n,
\end{equation}	
provided by Proposition \ref{pro:defineshorseshoe} satisfies
\begin{equation}\label{eq:conj2}
	(\Pi_n\circ \sigma_{\cB_n})(\underline w,x)
	= (F^{\lvert w_0\rvert}\circ\Pi_n)(\underline w,x).
\end{equation}
As $\cB_n$ is disjoint, by Proposition \ref{pro:defineshorseshoe}, the map $\Pi_n$ is ($\max R_n$)-to-one. Recall that, in general, $\Pi_n$ is uniformly finite-to-one and hence $\Pi_n^{-1}$ is multivalued. However, as $\cB_n$ is disjoint and hence, by Lemma \ref{lem:COND}, is uniquely left decipherable, the value 
\begin{equation}\label{eq:simplee}
	\underline R_n\circ\underline\Cut_n\circ\Pi_n^{-1}
\end{equation}	
is well-defined.

\subsection{Horseshoes are factors of suspension spaces}\label{subsec:suspen}

We now invoke the construction in Section \ref{ssec:susmod} to obtain the cascade of suspension spaces $\fS_n=\fS_{\cA_n,R_n}$ associated to the cascade of words $\cA_n$ and roof functions $R_n$, $n\in\bN$.
We will also consider the suspension of $\sigma_n=\sigma_{\cA_n}$ by $R_n$ and denote it by $\Phi_n=\Phi_{\cA_n,R_n}$. Recall the definition of the ground floor $\fG_n=(\cA_n)^\bZ\times\{0\}$ in \eqref{def:groundlfoor}. By \eqref{def:groundlfoorreturn} and using Notation \ref{rem:returns}, it holds
\[
	\Phi_n^{\underline R_n\circ\fp_n}|_{\fG_n}
	= \sigma_{n}\times\id.
\]
By construction of the suspension space, the map $\Phi_n^{\underline R_n\circ\fp_n}$ is the first return-map on $\fG_n$.

Recall the homeomorphism $\underline \Cut_n\colon(\cB_n)^\bZ\to(\cA_n)^\bZ$  in \eqref{secn} and the continuous surjective finite-to-one map $\Pi_n\colon(\cB_n)^\bZ\to\Lambda_n$ in \eqref{eq:agrree}. The following (commuting) diagrams put into relation the shift map on our abstract shift space $(\cA_n)^\bZ$, the shift map on the word space $(\cB_n)^\bZ\subset\Sigma_N$, and the induced return map on the part of the horseshoe obtained as the attractor of the CIFS at level $n$. The map $h_n$ is defined in \eqref{eq:defhn} below. 
\[
\hspace{-2cm}
\xymatrixcolsep{4pc}\xymatrix{
	&(\cA_n)^\bZ	\ar[d]^{\underline \Cut_n^{-1}}\ar[r]^{\,\,\sigma_{n}=\sigma_{\cA_n}\,\,}	
	&(\cA_n)^\bZ 	\ar[d]^{\underline \Cut_n^{-1}}\\
	&(\cB_n)^\bZ	\ar[d]^{\Pi_n}\ar[r]^{\sigma_{\cB_n}} 			
	&(\cB_n)^\bZ	\ar[d]	^{\Pi_n}\\ 
	&\Lambda_n\ar[r]^{F^{\underline R_n\circ \underline \Cut_n\circ\Pi_n^{-1}}} 
	&\Lambda_n	
	}
\xymatrixcolsep{3pc}\xymatrix{
	&\ar@/_2.5pc/[dd]_{h_n}
	(\cA_n)^\bZ\times\{0\}
		\ar[d]^{\underline \Cut_n^{-1}\circ\fp_n}	
		\ar[r]^{\,\Phi_n^{\underline R_n\circ\fp_n}\,\,}	
	&(\cA_n)^\bZ\times\{0\} 	\ar[d]^{\underline \Cut_n^{-1}\circ\fp_n}
	\\
	&(\cB_n)^\bZ	\ar[d]	^{\Pi_n} 
	&(\cB_n)^\bZ	\ar[d]	^{\Pi_n}
	\\ 
	&\Lambda_n\ar[r]^{F^{\underline R_n\circ \underline \Cut_n\circ\Pi_n^{-1}}} 
	&\Lambda_n	
	}
\]
Hence, the map
\begin{equation}\label{eq:forget}
	\underline a\in(\cA_n)^\bZ\mapsto
	\big(\Pi_n\circ\underline\Cut_n^{-1}\big)(\underline a)\in\Lambda_n
\end{equation}
is continuous, onto $\Lambda_n$,  at most ($\max R_n$)-to-one, and the above diagram comutes. The key result in this section is to extend the map \eqref{eq:forget} to a factor map between the suspension space and the full horseshoe, considering corresponding invariant measures. For that consider the $\Phi_n$-ergodic Borel probability measure 
\[
	\lambda_n
	\eqdef \lambda_{\cA_n,R_n}
\]	 
on the suspension space $\fS_n$ as defined in \eqref{eq:deflambdan}. 

\begin{proposition}\label{pro:semiconj}
	There is a continuous surjective map $H_n\colon \fS_n=\fS_{\cA_n,R_n}\to\Gamma_n$ which is uniformly finite-to-one such that
\[
	\card H_n^{-1}(\{X\})
	\le \left(\max R_n\right)^2,
	\spac{for every} X\in\Gamma_n,
\]	
satisfying \[H_n\circ \Phi_n=F\circ H_n.\]
Letting 
\begin{equation}\label{eq:defmun}	
	\mu_n
	\eqdef (H_n)_\ast\lambda_n,
\end{equation}
the measure preserving system $(\Gamma_n,F,\mu_n)$ is a factor of the measure preserving system $(\fS_n, \Phi_n,\lambda_n)$ by $H_n$. Moreover, $(\Gamma_n,F,\mu_n)$ is ergodic and it holds
\[
	h(F,\mu_n)
	= h(\Phi_n,\lambda_n)
	= \frac{\log M_n}{\frac{1}{M_n}\sum_{a\in\cA_n} R_n(a)},
	\spac{where}
	M_n=\card\cA_n.
\]
\end{proposition}

The following commuting diagram illustrates the above proposition.

\[\xymatrix{
	&\ar[d]_{h_n}
	(\cA_n)^\bZ\times\{0\}\subset
	&\ar[d]_{H_n}
	\fS_n			
	\ar[r]^{\Phi_n}
	&\fS_n
	\ar[d]^{H_n}
	&\lambda_n
	\ar[d]^{(H_n)_\ast}
	\\
	&\Lambda_n\subset
	&\Gamma_n\ar[r]^{F}
	&\Gamma_n&\mu_n
	}	
\]
To prove the  proposition, we need a preliminary result and start by introducing some notation.

\begin{notation}[Return maps]\label{rem:returns}{\rm
	Throughout this paper we consider several types of return maps. Given a map $S\colon X\to X$, a set $A\subset X$, and a function $R\colon A\to\bN$, we let 
\[
	S^R\colon A\to X,\quad
	S^R(x)
	\eqdef S^{R(x)}(x)
	= \overset{R(x)\text{ times}}{(S\circ\cdots\circ S)}(x).
\]	 
}\end{notation}

\begin{lemma}\label{cla:cunhc}
	Let $S\colon X\to X$ and $T\colon Y\to Y$ be two homeomorphisms on compact metric spaces. Assume that there are sets $A\subset X$ and $B\subset Y$ and continuous function $R_S\colon A\to\bN$ and $R_T\colon B\to\bN$  such that $S^{R_S}\colon A\to A$ is the first return-map on $A$ and $T^{R_T}\colon B\to B$ is a (not necessarily first) return map on $B$. Suppose that there is a continuous surjective map $h\colon A\to B$ satisfying 
\[
	R_T\circ h=R_S
	\spac{and}
	h\circ S^{R_S}=T^{R_T}\circ h.
\]	
Let $A'\eqdef \bigcup_{k\ge0}S^k(A)$ and $B'\eqdef \bigcup_{k\ge0}T^k(B)$.
Then there exists a continuous surjective map $H\colon A'\to B'$ which extends $h$ to $A'$ such that 
\[
	H\circ S|_{A'} 
	= T\circ H.
\] 
Moreover, if there is $K\in\bN$ satisfying
\[
	\card h^{-1}(\{b\})\le K
	\spac{for every }b\in B
\]	 
then $H$ is finite-to-one with 
\[
	\card H^{-1}(\{y\})\le K\sup R_T
	\spac{for every}
	y\in B'.
\]
\end{lemma}

\begin{proof}
Let us first define $H\colon A'\to B'$. 
Given $x'\in A'$, as $S^{R_S}$ is a first return to $A$, there are uniquely determined $x\in A$ and $k\in\{0,\ldots, R_S(x)-1\}$, so that $x'=S^k(x)$.  Let 
\[
	H(x')\eqdef T^k(h(x)).
\]	 
Noting that $y'=S(x')=S^{k+1}(x)$ satisfies $H(y')=T^{k+1}(h(x))$,  it holds
\[
	H\circ S(x')
	= H\circ S^{k+1}(x)
	= T^{k+1}\circ h(x)
	= T(T^k\circ h(x))
	= T(H\circ S^k(x)) 
	= T\circ H(x'),
\]
proving that the maps $S|_{A'}\colon A'\to A'$ and $T|_{B'}\colon B'\to B'$ are semiconjugate as claimed.

To check the claim about the cardinality of preimages, given $x'=S^k(x)$ as above, assume now that there is $y'\in A'$, $y'\ne x'$, such that $H(x')=H(y')$. There are uniquely determined $y\in A$ and $\ell\ge0$ so that $y'=S^\ell(y)$. Then $x'\ne y'$ implies $(x,k)\ne(y,\ell)$. As  $H(x')=H(y')$, it holds $T^k(h(x))=T^\ell(h(y))$ and hence $h(y)=T^{-\ell}(T^k(h(x)))$.
As the point $(x,k)$ was given and as there are at most $\sup R_T$ possible values for $\ell$, there are at most  $\sup R_T$ such points $h(y)$. Since $h$ is at most $K$-to-one, the claim follows.
\end{proof}

\begin{proof}[Proof of Proposition \ref{pro:semiconj}]
Note that 
\[
	(\underline \Cut_n^{-1}\circ\fp_n)\big((\cA_n)^\bZ\times\{0\}\big)
	= (\cB_n)^\bZ
	\spac{and}
	 \Pi_n((\cB_n)^\bZ)
	= \Lambda_n.
\]
Let
\begin{equation}\label{eq:defhn}
	h_n\colon(\cA)^\bZ\times\{0\}\to\Lambda_n,\quad
	h_n(\underline a)
	\eqdef  \left(\Pi_n\circ\underline \Cut_n^{-1}\circ\fp_n\right)
	(\underline a,0)
\end{equation}
We apply Lemma \ref{cla:cunhc} letting 
\[\begin{split}
	&X=\fS_n, \quad S=\Phi_n,  \quad A=(\cA_n)^\bZ\times\{0\},  \quad R_S=\underline R_n\circ\fp_n,\\
	&Y=\Gamma_n,  \quad T=F,  \quad B=\Lambda_n,  \quad R_T= 
	\underline R_n\circ \underline \Cut_n\circ\Pi_n^{-1},\quad h= h_n.
\end{split}\]	 
Recall that by \eqref{eq:simplee}, the function $R_T$ is well-defined.
Recall that, by the definition of the suspension space and with the notation of this lemma,  
\[
	A'
	= \fS_n
	= \bigcup_{k}\Phi_n^k\big((\cA_n)^\bZ\times\{0\}\big)
	\spac{and}
	B'
	= \Gamma_n.
\]	

In the next two claims we check the hypotheses of Lemma \ref{cla:cunhc}.
 
\begin{claim}\label{cla:cunha}
	$(\underline R_n\circ \underline \Cut_n\circ\Pi_n^{-1})\circ h_n 
	= \underline R_n\circ\fp_n$ on $A$.
\end{claim}

\begin{proof}
Using the definition of $h_n$ in \eqref{eq:defhn} and \eqref{eq:simplee}, we get
\[\begin{split}
	(\underline R_n\circ \underline \Cut_n\circ\Pi_n^{-1})\circ h_n
	= (\underline R_n\circ \underline \Cut_n\circ\Pi_n^{-1})
		\circ \left(\Pi_n\circ \underline \Cut_n^{-1}\circ\fp_n\right)
	= \underline R_n\circ \fp_n,
\end{split}\]
proving the claim.	
\end{proof}

\begin{claim}\label{cla:cunhb}
	$h_n\circ \Phi_n^{\underline R_n\circ \fp_n} 
	= F^{\underline R_n\circ \underline \Cut_n\circ\Pi_n^{-1}}\circ h_n$ 
	on $A$.
\end{claim}

\begin{proof}
	Note that on $(\cA_n)^\bZ\times\{0\}$ it holds
\begin{equation}\label{eq:firstformula}
	\Phi_n^{\underline R_n\circ\fp_n}
	= \sigma_{n}\times\id.
\end{equation}
Using the definition of $h_n$ in \eqref{eq:defhn} and \eqref{eq:firstformula}, it follows
\[\begin{split}
	\big(h_n\circ\Phi_n^{\underline R_n\circ \fp_n}\big)(\underline a,0) 
	&= \big(\left(\Pi_n\circ \underline \Cut_n^{-1}\circ\fp_n\right)
		\circ (\sigma_{n}\times\id)\big)(\underline a,0) \\
	&= \big(\Pi_n\circ \underline \Cut_n^{-1}\circ\sigma_n\big)(\underline a)\\
	{\tiny{\text{by Lemma \ref{lem:gnconjugates}}}}\quad
	&= \big(\Pi_n\circ (\underline \Cut_n^{-1}\circ\sigma_n)\big)(\underline a)
	= \big(\Pi_n \circ (\sigma_{\cB_n}\circ \underline \Cut_n^{-1})\big)(\underline a)\\
	{\tiny{\text{by \eqref{eq:conj2}}}}\quad
	&= \big(F^{\underline R_n\circ \underline \Cut_n\circ\Pi_n^{-1}}\circ
		(\Pi_n\circ \underline \Cut_n^{-1})\big)(\underline a)\\
	&= \big(F^{\underline R_n\circ \underline \Cut_n\circ\Pi_n^{-1}}\circ
		(\Pi_n\circ \underline \Cut_n^{-1}\circ\fp_n)\big)(\underline a,0)\\
	&= \big(F^{\underline R_n\circ \underline \Cut_n\circ\Pi_n^{-1}}\circ h_n\big)(\underline a,0)	,
\end{split}\]
proving the claim.
\end{proof}

Claims \ref{cla:cunha} and \ref{cla:cunhb} allow us to apply Lemma \ref{cla:cunhc} as explained above. Hence, there is a continuous surjective map $H_n\colon \fS_n\to\Gamma_n$ which extends $h_n$ and satisfies 
\[
	F\circ H_n
	= H_n\circ \Phi_n.
\]
The property about cardinality of preimages also follows from Lemma \ref{cla:cunhc} together with the fact that the map in \eqref{eq:forget} is at most $R_n$-to-one.  

The factor property and ergodicity are immediate consequences of this semiconjugation and the definition of $\mu_n$. Moreover, as the map $H_n$ is finite-to-one, by \cite{LedWal:77}, it holds
\[
	\sup_{\lambda\colon(H_n)_\ast\lambda=\mu_n}h(\Phi_n,\lambda)
	= h(F,\mu_n)
		+\int_{\Cs_n}h_{\rm top}(\Phi_n,H_n^{-1}(\{X\}))\,d\mu_n(X).
\]	
In the integral, $h_{\rm top}(\cdot)$ denotes the topological entropy, which is zero for every $X$ because $H_n^{-1}(\{X\})$ is finite for every $X$. 
As the entropy of a factor system is always less than or equal to the entropy of its extension, this implies
\[
	h(\Phi_n,\lambda_n)
	= h(F,\mu_n).
\]
The assertion about entropy in the proposition now is a consequence of Lemma \ref{lem:Abramov}.
\end{proof}

\subsection{Entropy and Lyapunov exponents of horseshoes}\label{subsec:entropiiie}

We now put the previous results into the context of the cascade of horseshoes $(\Gamma_n)_n$ in \eqref{eq:defGammanLambdan}. Recall that our construction, in particular those horseshoes, depend on the initially fixed ergodic measure $\mu$ as stated in the beginning of  Section \ref{sec:ncashor}. Recall that $L_1=L_1(F,J)$ is as in Definition \ref{defrem:commonblending}.

\begin{corollary}\label{cor:tailinghorse}
	 For every $\tilde\mu\in\cM_{\rm erg}(F|_{\Gamma_n})$ it holds
\[
	\chi(\tilde\mu)\in\Big(\frac{1}{2^n}(\alpha-\varepsilon),
			\frac{1}{2^n}(\alpha+\varepsilon)\Big).
\]
Moreover, the measure $\mu_n\in\cM_{\rm erg}(F|_{\Gamma_n})$ defined in \eqref{eq:defmun} satisfies
\[
	h(F,\mu_n)
	\ge e^{-L_1(1-2^{-n})\lvert\alpha\rvert} \left(h(F,\mu)-\varepsilon_H\right) .
\]	
In particular, any measure $\mu_\infty$ which is weak$\ast$-accumulated%
\footnote{Indeed, our particular choice of $(m_n)_n$ in Section \ref{sec:choicemn} implies convergence, see Lemma \ref{cla:convergence}.} 
 by $(\mu_n)_n$ satisfies
\[
	\chi(\mu_\infty)=0
	\spac{and}
	h(F,\mu_\infty)
	\ge e^{-L_1\lvert\alpha\rvert} \cdot \left(h(F,\mu)-\varepsilon_H\right) .
\]
\end{corollary}

\begin{proof}
By Corollary \ref{cor:cons1}, every collection of words $\cB_n$ defines a CIFS on $J$ relative to $K$, $2^{-n}\alpha_0$, $2^{-n}\alpha$, and $2^{-n}\varepsilon$. Hence, Proposition \ref{procor:specCIFS} (1) implies the statement about the exponents for any measure in $\cM_{\rm erg}(F|_{\Gamma_n})$.

Recall that, by Corollary \ref{newcor:notormenta}, Assumption \ref{ass:roof} is satisfied with $K=L_1\lvert\alpha\rvert$.
By Proposition \ref{pro:semiconj}, it follows
\[\begin{split}
	h(F,\mu_n)
	= h(\Phi_n,\lambda_n)
	&= \frac{\log M_n}{\frac{1}{M_n}\sum_{a\in\cA_n}R_n(a)}\\
	{\tiny\text{by Proposition \ref{procor:notormenta}}}\quad
	&\ge \frac{m_n\log M_{n-1}}
			{m_n(1+L_12^{-(n-1)}\lvert\alpha\rvert)\max R_{n-1}}\\
	&\ge\ldots
	\ge \frac{\log M_0}{\prod_{k=0}^{n-1}(1+L_12^{-k}\lvert\alpha\rvert)
		 \max R_0}\\
	&\ge e^{-L_1(1-2^{-n})\lvert\alpha\rvert}\cdot
		\frac{\log M_0}{ \max R_0}\\
	{\tiny\text{by \eqref{eq:estentropyy}, also using $R_0(a)=\lvert a\rvert$}}\quad	
	&\ge e^{-L_1(1-2^{-n})\lvert\alpha\rvert}
		\cdot(h(F,\mu)-\varepsilon_H),
\end{split}\]
proving the lower bound for entropy. 

Finally, recalling that $\chi(\mu_n)$ is the integral of a continuous function, we get $\chi(\mu_\infty)=0$ for any limit measure $\mu_\infty$. Further, as the entropy map is upper semi-continuous, the result about entropy follows taking limits as $n\to\infty$. 
\end{proof}

\subsection{Proof of Proposition \ref{pro:main}}

For every $n\in\bN$ consider the numbers 
\[
	\alpha_n
	\eqdef \inf\{\chi(\tilde\mu)\}
	\spac{and}
	\beta_n\eqdef \sup\{\chi(\tilde\mu)\},
\]
where $\inf$ and $\sup$ are taken over all $\tilde\mu\in\cM_{\rm erg}(F|_{\Gamma_n})$. By Corollary \ref{cor:tailinghorse}
\[
	\frac{1}{2^n}(\alpha-\varepsilon)
	<\alpha_n
	<\beta_n
	<\frac{1}{2^n}(\alpha+\varepsilon),
\]
and
\[
	h_{\rm top}(F,\Gamma_n)
	\ge h(F,\mu_n)
	\ge e^{-L_1(1-2^{-n})\lvert\alpha\rvert} \left(h(F,\mu)-\varepsilon_H\right) .
\]
The natural projection of $\Gamma_n$ to its first coordinate is the set $\CS{\cB_n}$ which, by definition, is a coded shift. The fact that $\Lambda_n$ is the attractor of a CIFS implies that every fiber intersecting $\Lambda_n$ contains only one point. Hence the claimed property of the cardinality a fiber intersecting  $\Gamma_n$ follows. This proves the proposition.
\qed

\section{Inherited internal structure of horseshoes}\label{sec:internal}

In this section, we see how the structure of ground and intermediate floors in the suspension space $\fS_n$ described in Section \ref{sec:susmodint} passes on to a corresponding internal structure of the horseshoe $\Gamma_n$ via the factor map $H_n$ in Proposition \ref{pro:semiconj}. 

First recall the definition of intermediate floors $\fG_n^{(j)}=\fG_n^{(n-1,(j))}$, $j\in\{0,\ldots,m_n-1\}$ in \eqref{eq:intfloor-a} and the notation in \eqref{eq:notationintflor}. Let
\begin{equation}\label{eq:defLambdarrr1}
	\Lambda_n^{(n-1,(j))}
	\eqdef H_n(\fG_n^{(n-1,(j))}).
\end{equation}

Analogously, taking into account inductively the spelling of a word in $\cB_n$ in the alphabet $\cB_\ell$, for some $\ell\in\{0,\ldots,n-1\}$, we consider the intermediate floor with $(\ell,n)$-address $\va=(\mathrm a_\ell,\ldots,\mathrm a_{n-1})$. Recalling the definition of the intermediate floor $\fG_n^{(\ell,\va)}$ in \eqref{eq:defintflo}, let 
\[
	\Lambda_n^{(\ell,\va)}
	\eqdef H_n(\fG_n^{(\ell,\va)}).
\]

\begin{remark}[Almost-return maps]
To motivate the above definitions, consider some point $X=(\xi,x(\xi))\in\Lambda_{n}$ and its forward orbit. As $\xi=\pi_1(X)\in\PCS{\cB_n}$, it has a ``spelling in the alphabet of words $\cB_{n}$'' as 
\[
	\xi
	=\underline\iota_{\cB_n}\big(\ldots, w^{(n)}_{i_{-1}}| w^{(n)}_{i_0}, w^{(n)}_{i_1},\ldots\big).
\]	
Recall that by \eqref{eq:conj2} it holds
\[
	F^{\lvert w^{(n)}_{i_0}\rvert}(X)
	\in\Lambda_{n},
\]
that is, this map is a return map on $\Lambda_{n}$.
Recall that, by \eqref{eq:msoqui1}, it holds 
\[
	\Lambda_n
	\subset \PCS{\cB_n}\times J
	\subset \Sigma_N\times J.
\]
We now refine this relation to any address.

For our main argument in Section \ref{sec:proooof} the following observation will be essential. Recall that $ w^{(n)}_{i_0}$ resulted from our repeat-and-tail scheme:
\[
	 w^{(n)}_{i_0}
	= \big( w_{j_1}^{(n-1)},\ldots, w_{j_{m_n}}^{(n-1)},
			\ft_n( w_{j_1}^{(n-1)},\ldots, w_{j_{m_n}}^{(n-1)})\big),
	\spac{where}		
	 w_{j_k}^{(n-1)}\in\cB_{n-1},
\]
analogously for the other elements $ w^{(n)}_{i_k}$. By definition \eqref{eq:defLambdarrr1}, the set $\Lambda_n^{(n-1,(j))}$ contains the point  on the forward orbit of $X$ whose position on the trajectory is determined by  the position of the word $ w^{(n-1)}_j$ in the sequence $\xi$. By Lemma \ref{lemcor:msoqui} below  it holds 
\[
	\Lambda_n^{(n-1,(j))}
	\subset \Sigma_N^-\times[ w_j^{(n-1)}]^+\times J
\]	 
and hence the corresponding part of the orbit of $X$ after few iterations is very close to the corresponding orbit starting in $\Lambda_{n-1}$. Indeed, Lemma \ref{lemcor:msoqui} considers any address $\va$ and provides an even finer description.
In very rough terms, this lemma states 
\[
	\text{``}\,\,\Lambda_n^{(\ell,\va)}
	\subset \Big(\Sigma_N^- \times \big(\cB_\ell\times \Sigma_N^+\big)\Big)\times J\,\,\text{''},
\]
though this formula is not precise for two reasons: first, it does not mark the 0th position of the two-sided sequence and second, $\cB_\ell$ is the union of words which possibly do not have equal length. Let us hence state the precise statement. 
\end{remark}

Recall notations in Lemma \ref{lem:hund} and that $\Cut_\ell^{-1}$ defined in \eqref{secn} ``adds tails" up to level $\ell$. Recall also Notation \ref{eq:notation}.

\begin{lemma}\label{lemcor:msoqui}
	For every $n\in\bN$, $\ell\in\{0,\ldots,n-1\}$, and $(\ell,n)$-address $\va$ and $\underline a\in(\cA_n)^\bZ$, it holds
\[
	H_n(\underline a,s_n^{(\ell,\va)})
	\in\Sigma_N^-\times
	\big[(\Cut_\ell^{-1}\circ\varsigma_n^{(\ell,\va)})(\underline a)\big]^+
	\times J.
\]
\end{lemma}

\begin{proof}
To prepare the proof, recall the terminology in Section \ref{sec:locintflor}.
Without loss of generality, we assume that the address $\va=(\mathrm a_\ell,\ldots,\mathrm a_{n-1})$ has simplified representation, that is, $\mathrm a_\ell\ne0$. 
Consider the sequence of addresses  (each with simplified representation) $\va^{(0)}=0$, $\ldots,$ $\va^{(k)},\ldots,\va^{(\lVert \va\rVert)}=\va$ as in \eqref{eq:serefer} and denote by $\ell_k = w(\va^{(k)})$ the corresponding levels which go from $\ell_0=n-1$ down to $\ell_{\lVert\va\rVert}=\ell$. Let 
\[
	\alpha_k
	\eqdef \varsigma_n^{(\ell_k,\va^{(k)})}(\underline a)
	= \left((\sigma_{{\ell_k}}^{\tau_k} \circ\underline\Sub_{n,\ell_k})(\underline a)\right)_0
	\in\cA_{\ell_k},
	\,\,\text{where}\,\,
	\tau_k
	= \sum_{i=\ell_k}^{n-1} 
		\mathrm a_i \cdot m_{\ell_k+1}\cdots m_{i+1}.
\]
Also consider the corresponding times in the suspension space for $k=0,\ldots,\lVert\va\rVert$
\begin{equation}\label{sell-final}\begin{split}
	t_0
	&=0,\ldots,\\
	t_k
	&= \sum_{i=1}^{\lVert\va^{(k)}\rVert} 
		R_{\ell_i}\big(\varsigma_n^{(\ell_i,\va^{(i-1)})}(\underline a)\big)
	= s_n^{(\ell_k,\va^{(k)})}(\underline a),\ldots,\\
	t_{\lVert\va\rVert}
	&= s_n^{(\ell,\va)}(\underline a)	.
\end{split}\end{equation}
For later reference, also recalling \eqref{eq:bug}, note that for every $k$ it holds 
\begin{equation}\label{eq:skiter}
	t_k-t_{k-1}
	= R_{\ell_k}\big(\varsigma_n^{(\ell_k,\va^{(k-1)})}(\underline a)\big)
	= R_{\ell_k}(\alpha_k)
	= \lvert  \Cut_{\ell_k}^{-1}(\alpha_k)\rvert,
\end{equation}
where $\lvert\cdot\rvert$ is the length of the corresponding word in the alphabet $\{1,\ldots,N\}$. 

We are now prepared to prove the lemma. Let
 \[
	H_n(\underline a, t_k)
	= \big(\eta^{(k)},x_k\big)
	\in\Sigma_N\times\bS^1,	
	\spac{where}
	\eta^{(k)}
	\eqdef \sigma^{t_k}(\eta^{(0)}).
\]
The proof will be by induction on $k=0,\ldots,\lVert\va\rVert$.
We have that
 \[\begin{split}
	(\eta^{(\lVert\va\rVert)},x_{\lVert\va\rVert})
	&= H_n(\underline a,t_{\lVert\va\rVert})\\
	{\tiny\text{(by definition of the suspension space) }\quad}
	&= (H_n\circ\Phi_n^{t_{\lVert\va\rVert}})(\underline a,0)\\
	{\tiny\text{(by semiconjuation in Proposition \ref{pro:semiconj}) }\quad}
	&= \big(F^{t_{\lVert\va\rVert}}\circ H_n\big)(\underline a,0).
 \end{split}\] 
Thus, by definition of the skew product $F$ together with \eqref{sell-final}, it follows 
\[
	H_n(\underline a,0)
	= F^{-t_{\lVert\va\rVert}}(\eta^{(\lVert\va\rVert)},x_{\lVert\va\rVert})
	= ( \sigma^{-t_{\lVert\va\rVert}}(\eta^{(\lVert\va\rVert)}),x_0)
	= (\eta^{(0)},x_0).
\]	
One the one hand, $H_n$ extends $h_n$ and maps $(\cA_n)^\bZ\times\{0\}$ into $\PCS{\cB_n}\times J$ and hence $\eta^{(0)}\in \PCS{\cB_n}=\underline\iota_{\cB_n}(\cB_n^\bZ)$. On the other hand, $\eta^{(0)}=(\underline\iota_{\cB_n}\circ \underline\Cut_n^{-1})(\underline a)$.
As done before, in the following we will use that $\cB_n$ is disjoint and hence uniquely left decipherable.
Hence, using the notation above, for $\underline a=(\ldots,a_{-1}|a_0,a_1,\ldots)$ the associated word $(\iota_{\cB_n}\circ\Cut_n^{-1})(a_0)\in\Sigma_N^\ast$ is uniquely determined by $a_0$. With the notation in \eqref{eq:specific} the corresponding  cylinder is 
\[
	[\Cut_n^{-1}(a_0)]^+
	=[(\iota_{\cB_n}\circ\Cut_n^{-1})(a_0)]^+
\]	 
and we obtain
\[
	H_n(\underline a,0)
	\in \Sigma_N^-\times [\Cut_n^{-1}(a_0)]^+\times J.
\]
This proves the assertion for $k=0$.

Assume that the assertion was shown for $k-1$ for some $k\ge1$. It holds
\begin{equation}\label{eq:indstep}\begin{split}
 	(\eta^{(k)},x_k)
	= H_n(\underline a, t_k) 
	= H_n(\Phi_n^{t_k}(\underline a,0))
	&= H_n\big( \Phi_n^{t_k-t_{k-1}}
		\circ \Phi_n^{t_{k-1}}(\underline a,0)\big)\\
	{\tiny\text{(by the factor property) }\quad}
	&= F^{t_k-t_{k-1}}(H_n( 	\Phi_n^{t_{k-1}}(\underline a,0)))\\
	&= F^{t_k-t_{k-1}}(H_n(\underline a,t_{k-1}))\\
	&= F^{t_k-t_{k-1}}(\eta^{(k-1)},x_{k-1}).
 \end{split}\end{equation}
By induction hypothesis, it holds 
\[
	(\eta^{(k-1)},x_{k-1})
	\in \Sigma_N^-\times[ \Cut_{\ell_{k-1}}^{-1}(\alpha_{k-1})]^+\times J.
\]	 
By definition of the skew product $F$ and \eqref{eq:indstep}, it follows
 \[\begin{split}
	(\eta^{(k)},x_k)
	&= \big(\sigma^{t_k-t_{k-1}}(\eta^{(k-1)}),f_{[\eta^{(k-1)}]}^{t_k-t_{k-1}}(x_{k-1})\big)\\
	\tiny{\text{using \eqref{eq:skiter} \quad}}
	&= \big(\sigma^{R_k}(\eta^{(k-1)}),f_{[\eta^{(k-1)}]}^{R_k}(x_{k-1})\big),
	\spac{where}
	R_k
	\eqdef R_{\ell_k}(\alpha_k).
 \end{split}\] 

To finish the proof, it is enough to check the following.

\begin{claim}
 The first $R_k$ symbols of $\eta^{(k-1)}$ (in the original alphabet $\{1,\ldots,N\}$) form a word in $\cB_{\ell_k}$. In particular, $f_{[\eta^{(k-1)}]}^{R_k}$ is a map of the CIFS on $J$ defined by $\cB_{\ell_k}$. 
\end{claim}
 
\begin{proof} 
By definition of $ \Cut_{\ell_k}$, it holds $ \Cut_{\ell_k}^{-1}(\alpha_k)\in\cB_{\ell_k}$. Note also that it is a subword of $\eta^{(k-1)}$. There are two cases to check. 
 
\smallskip\noindent
\textbf{Case 1: $\ell_k=\ell_{k-1}$.}
In this case, $ \Cut_{\ell_k}^{-1}(\alpha_k)$ is the second element in the bi-infinite concatenation of words (in the alphabet $\cB_{\ell_k}=\cB_{\ell_{k-1}}$) forming $\eta^{(k-1)}$, that is, with
\[
	\eta^{(k-1)}
	= (\ldots |\eta^{(k-1)}_{i_0},\eta^{(k-1)}_{i_1},\ldots)
	\spac{it holds}
	 \Cut_{\ell_k}^{-1}(\alpha_k)
	= \eta^{(k-1)}_{i_1}\in\cB_{\ell_k}
\]
and $\eta^{(k)}=\sigma^{R_k}(\eta^{(k-1)})=(\ldots \eta^{(k-1)}_{i_0}|\eta^{(k-1)}_{i_1},\ldots)$,
and we are done.  

\smallskip\noindent
\textbf{Case 2: $\ell_k<\ell_{k-1}$.}
In this case, we recall that $\eta^{(k-1)}$ is a bi-infinite concatenation of words in $\cB_{\ell_{k-1}}$, 
\[
	\eta^{(k-1)}
	= (\ldots |\eta^{(k-1)}_{i_0},\eta^{(k-1)}_{i_1},\ldots)
\]
where each such subword is a $m_{\ell_k}$-times repeated and tailed version of words in $\cB_{\ell_k}$,
\[
	\eta^{(k-1)}_{i_0}
	= \big( w^{(\ell_k)}_{j_1}, w^{(\ell_k)}_{j_2},
		\ldots, w^{(\ell_k)}_{j_{m_{\ell_{k-1}}}},
		\ft_{\ell_k+1}( w^{(\ell_k)}_{j_1}, w^{(\ell_k)}_{j_2},
			\ldots, w^{(\ell_k)}_{j_{m_{\ell_{k-1}}}})\big).
\]
In particular, $ \Cut_{\ell_k}^{-1}(\alpha_k)= w^{(\ell_k)}_{j_2}\in\cB_{\ell_k}$ and $\eta^{(k)}=(\ldots| w^{(\ell_k)}_{j_2},\ldots)$, and we are done.   
\end{proof}
The proof of the lemma is now complete.
\end{proof}

\section{Core of the proof of Theorem  \ref{teo:2}}\label{sec:proooof}

This section puts together all ingredients developed throughout this paper. 
Let us collects the mains ones to state the key result towards the proof of Theorem \ref{teo:2}.

Given $F\in\mathrm{SP}^1_{\rm shyp}(\Sigma_N\times\bS^1)$ and some $F$-ergodic measure $\mu$ with negative Lyapunov exponent, by Theorem \ref{teo:existenceCIFS} we obtain an initial collection of words $\cB_0$ which defines a CIFS with quantifiers on some interval $J$. Given a sequence $(m_n)_n$, we define a cascade of collections of words $(\cB_n)_n$  where each word in $\cB_n$ is the $m_n$-repeated and tailed version of words in $\cB_{n-1}$, the tailing map $\ft_n$ as in Theorem \ref{thepro:tailing}. By Proposition \ref{pro:defineshorseshoe}, each collection $\cB_n$ has an associated horseshoe $\Gamma_n$ of $F$. The word length defines a corresponding roof function $R_n$. 

Our construction is accompanied by a cascade of abstract alphabets $(\cA_n)_n$.  Each $\cA_n$ defines a shift space, endowed with the shift map $\sigma_n=\sigma_{\cA_n}$ and the Bernoulli measure $\fb_n$, which forms the ground floor $\fG_n$ of the suspension space $\fS_n=\fS_{\cA_n,R_n}$, with associated suspension map $\Phi_n$ and measure $\lambda_n$, see Section \ref{ssec:susmod}. The measure preserving system $(\fS_n,\Phi_n,\lambda_n)$ is an extension of $(\Gamma_n,F,\mu_n)$ by the factor map $H_n$, see Proposition \ref{pro:semiconj}.

\begin{proposition}\label{prothe:suspflow}
	Let $F\in \mathrm{SP}^1_{\rm shyp}(\Sigma_N\times\bS^1)$, $N\ge2$, and consider an $F$-invariant ergodic measure $\mu$ with Lyapunov exponent $\alpha\eqdef\chi(\mu)<0$ and positive entropy $h=h(F,\mu)$. For every $\varepsilon_H\in(0,h)$ and $\varepsilon_E\in(0,\lvert\alpha\rvert/4)$, there are a closed interval $J\subset\bS^1$ and an initial finite disjoint collection of words $\cB_0\subset\Sigma_N^\ast$ defining a CIFS on $J$ relative to some  
$K>1$, $\alpha+\varepsilon_E$, $\alpha$, and $\varepsilon_E$ satisfying
\begin{equation}\label{pro:theoremB1}
	\min_{ w\in\cB_0}\lvert w\rvert(h-\varepsilon_H)
	\le \log\card\cB_0
	\le \max_{ w\in\cB_0}\lvert w\rvert(h+\varepsilon_H).
\end{equation}
Moreover, there is a sufficiently fast growing sequence of natural numbers $(m_n)_n$ such that the measure preserving systems $(\fS_n=\fS_{\cA_n,R_n},\Phi_n,\lambda_n)$ satisfy the following. For every continuous function $\phi\colon\Sigma_N\times\bS^1\to\bR$  and $\varepsilon>0$, there exists $L_0=L_0(\phi,\varepsilon)\in\bN$ such that for every $\ell\ge L_0$ and $n\ge \ell+1$, there exists a subset $\fS_{n,\phi,\varepsilon}$ of $\fS_n$ such that $\lambda_n(\fS_{n,\phi,\varepsilon})>1-\varepsilon$ and for every $(\underline a,s)\in \fS_{n,\phi,\varepsilon}$ it holds
\begin{equation}\label{eq:expectedd}
	\left\lvert\frac{1}{\frakR_\ell}
		\sum_{k=0}^{\frakR_\ell-1}\psi_n(\Phi_n^k(\underline a,s))
		- \int\phi\,d\mu \right\rvert
	<\varepsilon,
\end{equation}
where
\[
	\frakR_n
	\eqdef \int \underline R_n\,d\fb_n
	\spac{and}
	\psi_n\colon\fS_n\to\bR, \quad
	\psi_n(\underline a,s)\eqdef(\phi\circ H_n)(\underline a,s).
\]
\end{proposition}

The proof of the above proposition will be split into subsections.

First, note that by Theorem \ref{teo:existenceCIFS} there exist a closed interval $J\subset\bS^1$ and a finite disjoint collection of words $\cB_0\subset\Sigma_N^\ast$ defining a CIFS on $J$ relative to some constant $K>1$ and $\alpha+\varepsilon_E$, $\alpha$, and $\varepsilon_E$, and also satisfying \eqref{pro:theoremB1}. Let $\cA_0=\cB_0$. 

\subsection{Choice of the fast growing sequence $(m_n)_n$.}\label{sec:choicemn}

We start by fixing a dense sequence of continuous functions $\phi_k \colon \Sigma_N\times\bS^1\to\bR$, $k\in\bN$. 

The sequence $(m_n)_n$ is defined inductively over $n\in\bN_0$. Let $m_0=1$. 
Suppose that for $n\in\bN$ all numbers $\{m_0$, $m_1$, $\ldots,$ $m_{n-1}\}$ are chosen and hence $\cB_{n-1}$ and $\cA_{n-1}$ are defined and verify:
\begin{itemize}
\item the collection $\cB_{n-1}$ defines a CIFS on $J$ relative to $K\ge1$, $2^{-n}(\alpha+\varepsilon_E)$, $2^{-n}\alpha$, and $2^{-n}\varepsilon_E$,
\item there are the associated attractor $\Lambda_{n-1}$ for the CIFS and the horseshoe $\Gamma_{n-1}\supset \Lambda_{n-1}$ generated by it, see Proposition \ref{pro:defineshorseshoe},
\item the word length on $\cB_{n-1}$ defines the roof function $R_{n-1}$,
\item the abstract Bernoulli shift  $((\cA_{n-1})^\bZ,\sigma_{{n-1}},\fb_{n-1})$  forms the ground floor $\fG_{n-1}$ for the measure preserving suspension system $(\fS_{n-1},\Phi_{n-1},\lambda_{n-1})$, see Section \ref{ssec:susmod}. This system is an entropy-preserving extension of $(\Gamma_{n-1},F,\mu_{n-1})$ by the factor map $H_{n-1}$, see Proposition \ref{pro:semiconj}. 
\end{itemize}

To define $m_n$, for $k\in\{1,\ldots,n\}$ consider the auxiliary lifted potentials associated to $\phi_k$,
\[
	\psi_{n-1,k}\colon\fS_{n-1}\to\bR,
	\spac{where}
	\psi_{n-1,k}(\underline a,s)
	\eqdef (\phi_k\circ H_{n-1})(\underline a,s),
\]
and let
\begin{itemize}
\item[(I)] (controlled large deviation) $N_0(\psi_{n-1,k},2^{-n})\in\bN$ be as in Proposition \ref{pro:LD} applied to $\cA_{n-1}$ and $R_{n-1}$,
\item[(II)] (controlled distortion) $N_1(\phi_k,2^{-n})\in\bN$ be as in Proposition \ref{pro:distortion} applied to $\cB_{n-1}$,
\item[(III)] (tailing map) $N_2=N_2(\cB_{n-1})\in\bN$ be as in Theorem \ref{thepro:tailing} applied to $\cB_{n-1}$.
 \end{itemize}
We now define $m_n$ by
\[
	m_n
	\eqdef \max\Big\{
		\max_{k=1,\ldots,n}N_0(\psi_{n-1,k},2^{-n}),
		\max_{k=1,\ldots,n}N_1(\phi_k,2^{-n}),N_2,m_{n-1}\Big\}
	+1.
\]

Finally, we define $\cB_n$ as the $m_n$-times repeated and tailed version of words in $\cB_{n-1}$, 
\[	\cB_n=(\cB_{n-1}^{m_n})_{\ft_n},\] with the tailing map $\ft_n$ as in Theorem \ref{thepro:tailing}. We also let $\cA_n=(\cA_{n-1})^{m_n}$. This finishes the inductive definition.

Note that by Corollary \ref{newcor:notormenta} we have that Assumption \ref{ass:roof} about the roof functions $R_n$ is satisfied taking $K=L_1\lvert\alpha\rvert$.

\subsection{General scheme of the proof of Proposition \ref{prothe:suspflow}}

Let us sketch the steps of the proof and recall the main ingredients which will be implemented, compare also Figure \ref{fig.4}. 

We first show that the sequence $(\mu_n)_n$ converges in the weak$\ast$ topology, see Section \ref{sec:convergence}.

To prove the proposition, we need to show the approximation property \eqref{eq:expectedd} for a sufficiently large set of points. 
To do so, given $\phi$ first find $k_0\in\bN$ such that $\phi_{k_0}$ from our dense family is close to it. Choose large $\ell\ge k_0$ and let $n\ge \ell+1$, see Section \ref{sec:quantifiers}.

By implementing item (I) above, controlled large deviation on level $\ell-1$ provides us a large set of good orbit pieces on $\fS_{\ell-1}$ which (up to $m_\ell$ consecutive  times) run from the $(\ell-1)$st level ground floor to its roof. Here each piece has a close-to-expected length and a close-to-expected finite Birkhoff sum of the lift of the potential $\phi_{k_0}$ to $\fS_{\ell-1}$.  As $\phi_{k_0}$ and $\phi$ are close, these properties extend to the lift of $\phi$. See Section \ref{sec:restate}.

Consider the principal part $\fR_n^{(\ell)}\subset\fS_n$ defined in \eqref{eq:fodase} which decomposes into strips $\fL_n^{(\ell,\va)}$ indexed by $(\ell,\va)$-addresses
\[
	\fR_n^{(\ell)}
	=\cupdot_{\va=(\mathrm a_\ell,\ldots,\mathrm a_{n-1})}\fL_n^{(\ell,\va)}.
\] 
Every strip with $(\ell,n)$-address $\va$ decomposes into $(\ell-1,n)$-substrips
\[
	\fL_n^{(\ell,\va)}
	= \cupdot_{j=0,\ldots,m_\ell-1}\fL_n^{(\ell-1,j\va)}
\]
 which start at the corresponding intermediate floors $\fG_n^{(\ell-1,j\va)}$. Recall the definition of the map $L_n^{(\ell-1,j\va)}$ in \eqref{eq:Lnellalpha} mapping the  ``model suspension space'' $\fS_{\ell-1}$ bijectively onto the substrip $\fL_n^{(\ell-1,j\va)}$. In this way, each good orbit piece obtained by controlled large deviation is sent to its counterpart on the level $n$-suspension space $\fS_n$. This is more precisely stated in Main Lemma \ref{mlemmapro:proofpro} in Section \ref{sec:separating} whose proof is postponed to Section \ref{sec:MainLemma}.

Assuming Main Lemma \ref{mlemmapro:proofpro}, in Section \ref{sec:endofproof} we conclude the proof of the proposition. The following are the main ingredients. The Bernoulli measure $\fb_{\ell-1}$ on the $(\ell-1)$st level ground floor lifts isomorphically to the Bernoulli measure  $\fb_n$ on the $n$th level ground floor which, in turn, lifts naturally to its copy $\fb_n^{(\ell,\va)}$ on the $(\ell,n)$-intermediate floor. This allows us to conclude that the large measure set on the model space has its large measure counterpart on each strip. Stitching together all strips provides a large measure subset of the principal part $\fR_n^{(\ell)}$. Finally we will see that, by construction, the tail part $\fT_n^{(\ell)}$ has comparably small measure.

\bigskip
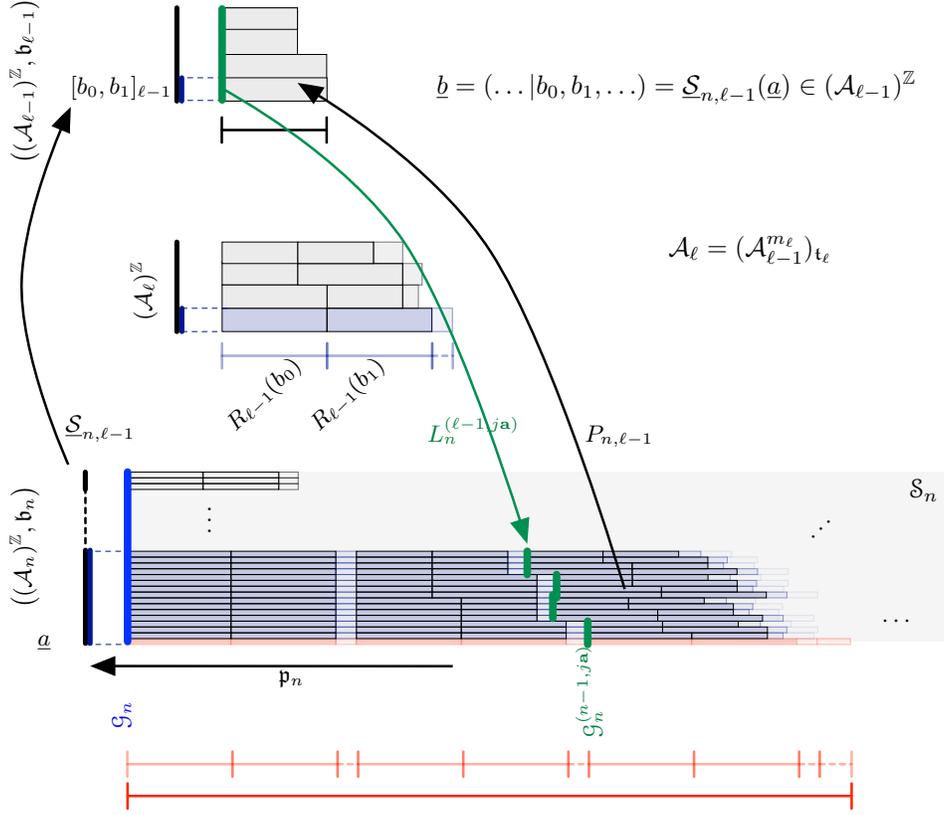
\begin{figure}[h] 
 \begin{overpic}[scale=0.95]{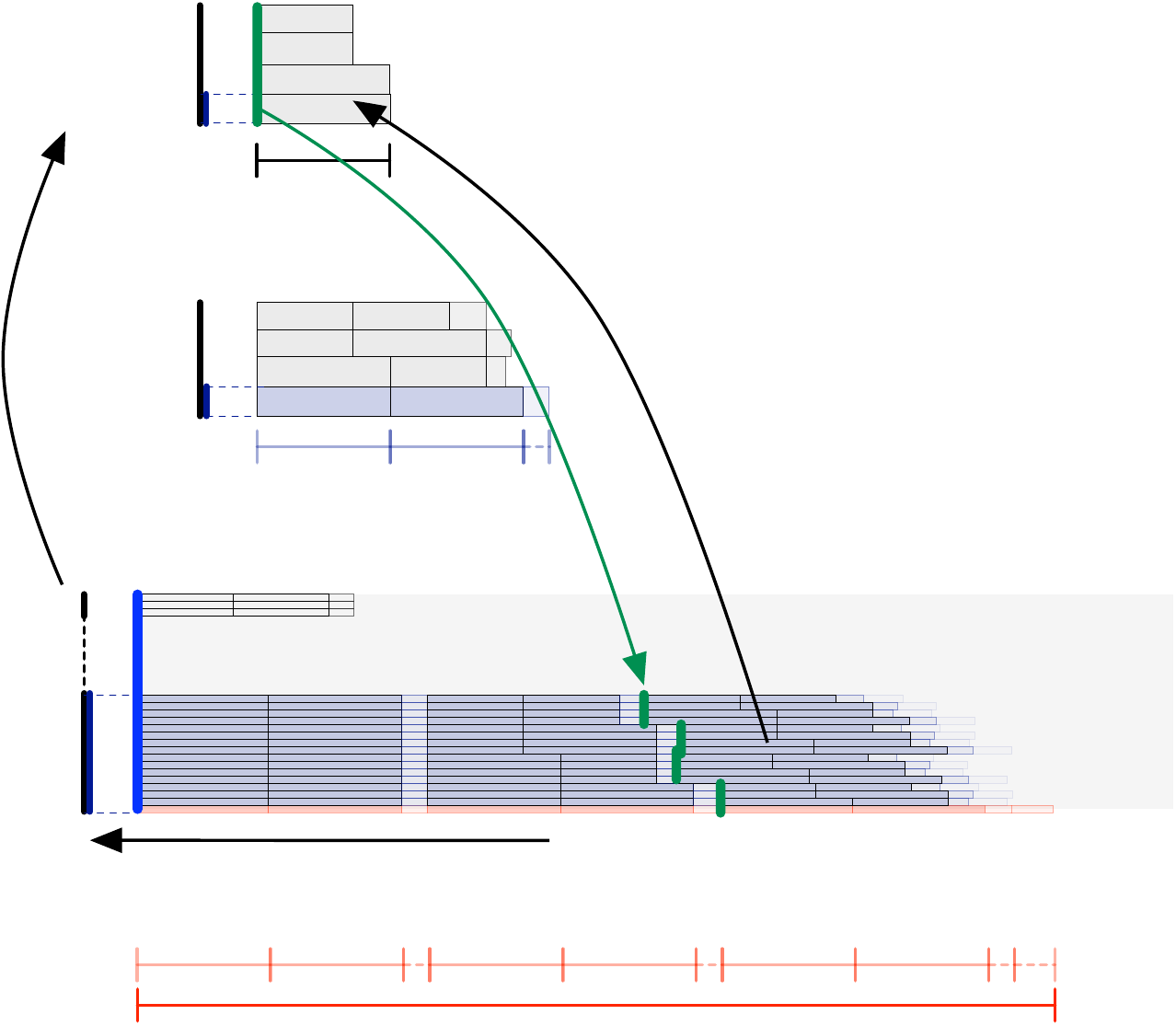}
 	\put(-1,70){{\rotatebox{90}{\small$\big((\cA_{\ell-1})^\bZ,\fb_{\ell-1}\big)$}}}
 	\put(12,53){\small{\rotatebox{90}{$(\cA_\ell)^\bZ$}}}
 	\put(-1,22){\small{\rotatebox{90}{$\big((\cA_n)^\bZ,\fb_n\big)$}}}
 	\put(96,34){{$\fS_n$}}
	\put(28,14){{\small$\fp_n$}}
 	\put(10,9){{\rotatebox{90}{\small\textcolor{blue}{$\fG_n$}}}}
 	\put(60,8){{\rotatebox{90}{\small\textcolor{dgreen}{$\fG_n^{(n-1,j\va)}$}}}}
	\put(5,41){\small{$\underline\Sub_{n,\ell-1}$}}
	\put(44,40){\small{\textcolor{dgreen}{$L_n^{(\ell-1,j\va)}$}}}
	\put(61,40){\small{$P_{n,\ell-1}$}}
	\put(22,41){\small{\rotatebox{45}{$R_{\ell-1}(b_0)$}}}
	\put(31,41){\small{\rotatebox{45}{$R_{\ell-1}(b_1)$}}}
	\put(70,60){{$\cA_\ell=(\cA_{\ell-1}^{m_\ell})_{\ft_\ell}$}}
	\put(5.5,77.5){\small{$[b_0,b_1]_{\ell-1}$}}
	\put(45,77.5){{$\underline b=(\ldots|b_0,b_1,\ldots)=\underline\Sub_{n,\ell-1}(\underline a)\in(\cA_{\ell-1})^\bZ$}}
	\put(2,18){\small{$\underline a$}}
	\put(20,30){{\rotatebox{90}{\small$\cdots$}}}
	\put(85,29){{\rotatebox{45}{\small$\cdots$}}}
	\put(93,20){\small$\cdots$}
 \end{overpic}
  \caption{Cascade of suspension spaces: Ingredients of the proof of Proposition \ref{prothe:suspflow}}
 \label{fig.4}
\end{figure}

\subsection{Weak$\ast$ convergence of the factor measures $\mu_n$}\label{sec:convergence}

Our construction provides sequences of probability measures $(\mu_n)_n=(H_n)_\ast\lambda_n\subset\cM_{\rm erg}(F)$, see \eqref{eq:defmun}. 
We first show that this sequence converges in the weak$\ast$ topology.

\begin{lemma}\label{cla:convergence}
	The sequence $(\mu_n)_n$ converges in the weak$\ast$ topology.
\end{lemma}

Let us first state a preliminary result which is a direct consequence of item (II) above about distortion control, together with Proposition \ref{procor:notormenta}  (1)--(2).
Recall the definition of the sum $\Delta\psi_{n,k}$ in \eqref{eq:defDeltapsi} and in its variation $\var_{\cA_n}$ in \eqref{eq:defAve}. We consider the abstract alphabets $\cA_n$ and, for simplicity, write $[\cdot]_n$ for the cylinder in the sequence space $(\cA_n)^\bZ$.

\begin{claim}\label{clacor:ka3}
	For every $n\in\bN$ and $k\in\{1,\ldots,n\}$ it holds
\[
	\var_{\cA_n}(\Delta\psi_{n,k})
	\le  \frac{1}{2^n}m_n\max R_{n-1}
	\le L_2\frac{1}{2^n}\frakR_n.
\]	
Moreover, reformulating the above taking into account the bijection $\fS_{n,n-1}$ between $\cA_n$ and $(\cA_{n-1})^{m_n}$, for every $a\in\cA_n$ it holds
\[
	\max_{\underline a,\underline a'\in[a]_n}\Big\lvert
	\sum_{i=0}^{m_n-1}
		\Delta \psi_{n-1,k}(\sigma_{{n-1}}^i(\underline\Sub_{n,n-1}(\underline a)))
		- \Delta\psi_{n,k}(\underline a')
	\Big\rvert 	
	\le L_2\frac{1}{2^n}\frakR_n.
\]
\end{claim}

\begin{proof}[Proof of Lemma \ref{cla:convergence}]
It suffices to show that for the dense sequence of continuous functions $(\phi_k)_k$,  
it holds
\begin{equation}\label{eq:foooorr}
	\Big\lvert\int\phi_k\,d\mu_n-\int\phi_k\,d\mu_{n-1}\Big\rvert
	< C(k,n),
\end{equation}
for some summable sequence $(C(k,n))_n$.
Recall that $\fb_n([a]_n)=(m_1\cdots m_n)^{-1}$ for every $a\in\cA_n$. Covering the sequence space $(\cA_n)^\bZ$ by the cylinders $\{[a]_n\colon a\in\cA_n\}$, it holds
\begin{equation}\label{eq:first}
	\frac{1}{\frakR_n}
	\sum_{a\in\cA_n}\min_{\underline a\in[a]_n}\Delta\psi_{n,k}(\underline a)
	\frac{1}{m_1\cdots m_n}
	\le \int\phi_k\,d\mu_n
\end{equation}
together with the analogous upper bound.
Analogously, covering the sequence space $(\cA_{n-1})^\bZ$ by cylinders of length $m_n$, it holds
\begin{equation}\label{eq:firstb}\begin{split}
	 \int& \phi_k \,d\mu_{n-1} \\
	&\le\frac{1}{\frakR_{n-1} } \sum_{(b_1,\ldots,b_{m_n})} 
	\max_{\underline b} \frac{1}{m_n}
		\sum_{i=0}^{m_n-1}
	\Delta \psi_{n-1,k}(\sigma_{{n-1}}^i(\underline b))
	\frac{1}{m_1\cdots m_{n-1}m_n},
\end{split}\end{equation}
where the maximum is taken over all sequences $\underline b$ in the cylinder $ [b_1,\ldots,b_{m_n}]_{n-1}\subset(\cA_{n-1})^\bZ$ and the sum is taken over all $(b_1,\ldots,b_{m_n})\in (\cA _{n-1})^{m_n}$. The lower bound is analogous.  

By Proposition \ref{procor:notormenta} (3), it holds 
\begin{equation}\label{eq:k1}
	\frac{1}{\frakR_n}
	< \frac{1}{m_n \frakR_{n-1}} 
	 <\frac{1}{\frakR_n}+ \frac{1}{\frakR_n}\frac{L_2}{2^n}.
\end{equation}
We now estimate \eqref{eq:foooorr}, let us compare \eqref{eq:firstb} with \eqref{eq:first}.
Applying \eqref{eq:k1} and taking into consideration the bijective map $\Sub_{n,n-1}$  between $\cA_n$ and $(\cA _{n-1})^{m_n}$, together with $m_n\card\cA_{n-1}=\card\cA_n$, $m_n\max R_{n-1}<\max R_n$, and Claim \ref{clacor:ka3}, it follows
\[\begin{split}
	&\left\lvert \frac{1}{m_n\frakR_{n-1}}
		\sum_{b\in (\cA _{n-1})^{m_n}} 
		\max_{\underline b\in [b]_{n-1}} 
		\sum_{i=0}^{m_n-1}\Delta \psi_{n-1,k}(\sigma_{{n-1}}^i(\underline b))
	- \frac{1}{\frakR_n}
	\sum_{a\in\cA_n}\min_{\underline a\in[a]_n}\Delta\psi_{n,k}(\underline a)
	\right\rvert \\
	&\le \frac{1}{\frakR_n}\left\lvert 
		\sum_{b\in (\cA _{n-1})^{m_n}} 
		\max_{\underline b\in [b]_{n-1}} 
		\sum_{i=0}^{m_n-1}\Delta \psi_{n-1,k}(\sigma_{{n-1}}^i(\underline b))
	- \sum_{a\in\cA_n}\min_{\underline a\in[a]_n}\Delta\psi_{n,k}(\underline a)
	\right\rvert \\
	&\phantom{=}
	+ \frac{1}{\frakR_n}\frac{L_2}{2^n}\cdot
		m_n\card\cA_{n-1}\cdot m_n\max R_{n-1}\cdot\lVert\phi_k\rVert\\
	&\le \frac{1}{\frakR_n}
		\sum_{a\in \cA _n} 
	\left\lvert 		
		\max_{\underline a\in [a]_n} 
		\sum_{i=0}^{m_n-1}\Delta \psi_{n-1,k}(\sigma_{{n-1}}^i(\underline\Sub_{n,n-1}(\underline a)))
		- \min_{\underline a\in[a]_n}\Delta\psi_{n,k}(\underline a)
	\right\rvert \\
	&\phantom{=}
	+ \frac{1}{\frakR_n}\frac{L_2}{2^n}\cdot
		\card\cA_n\cdot \max R_n\cdot\lVert\phi_k\rVert\\
	&\le \frac{1}{\frakR_n} \card\cA_n\cdot L_2\frac{1}{2^n}\frakR_n
	+ \frac{\max R_n}{\frakR_n}\frac{L_2}{2^n}\card\cA_n\lVert\phi_k\rVert\\
	&\le m_1\cdots m_n \cdot C(k,n),
		\spac{where} 
	C(k,n)\eqdef\left\{
			L_2\frac{1}{2^n}
			+ L_2^2\frac{1}{2^n}\lVert\phi_k\rVert
		\right\},
\end{split}\]
where for the last estimate we used \eqref{eq:cardAn} $\card\cA_n=m_1\cdots m_n$ and Proposition \ref{procor:notormenta} (2). The analogous estimate holds exchanging $\max$ and $\min$. 

Clearly, $(C(k,n))_n$ is summable. 
Combining the estimates \eqref{eq:first} and \eqref{eq:firstb} of the integrals, it follows \eqref{eq:foooorr}. This finishes the proof.
\end{proof}

\subsection{Choice of quantifiers.}\label{sec:quantifiers}
Fix a continuous function $\phi\colon\Sigma_N\times\bS^1\to\bR$ and $\varepsilon\in(0,1/3)$. Choose $k_0\in\bN$ and $L_0=L_0(\phi,\varepsilon)\ge k_0+2$ such that for all $\ell\ge L_0$ it holds
\begin{equation}\label{eq:gettingclose}\begin{split}
	\lVert\phi_{k_0}-\phi\rVert
	<\varepsilon,\quad
	\frac{L_2}{2^{\ell-2}}\max\left\{1,\lVert\phi\rVert\right\}
	<\varepsilon,\quad
	2\lVert\phi\rVert\frac{1}{\ell}L_2
	<\varepsilon.
\end{split}\end{equation}
Hence, for every $n\ge L_0$ and $k\in\{1,\ldots,n\}$ the assertions in (I)--(III)  apply to the function $\phi_k$ and its lift $\psi_{n-1,k}=\phi_k\circ H_{n-1}$ and $m_n$.
By Lemma \ref{cla:convergence}, we can assume that $L_0\in\bN$ is large enough that for every $\ell\ge L_0$ it holds
\begin{equation}\label{eq:gettingclose-phi}
	\left\lvert\int\phi\,d\mu_{\ell-1}-\int\phi\,d\mu\right\rvert
	<\varepsilon.
\end{equation}

\subsection{Invoking assertions (I)--(II) to restate large deviation control.}\label{sec:restate}

Let us restate the estimate in item (I) in a more convenient way.   

\begin{lemma}[Controlled large deviation]\label{lemcla:invoke}
	For every $\ell\ge L_0+1$ there exists  a set 
\[
	A\subset(\cA_{\ell-1})^\bZ
	\spac{satisfying}
	\fb_{\ell-1}(A)
	>1-\varepsilon	
\]
so that for every $\underline b\in A$, $i=0,\ldots,m_{\ell}-1$, and $k\in\{1,\ldots,m_{\ell}\}$ we have
\begin{equation}\label{eq:applyLD}
	\Big\lvert\sum_{j=i}^{i+k-1}\left(\underline R_{\ell-1}(\sigma_{{\ell-1}}^j(\underline b))
		-\frakR_{\ell-1}\right)\Big\rvert
	<m_\ell\frac{1}{2^\ell}
	< m_\ell\varepsilon
\end{equation}
and for $\psi_{\ell-1}\eqdef \phi\circ H_{\ell-1}$ it holds
\begin{equation}\label{eq:regena}
	\Big\lvert
	\sum_{j=i}^{i+k-1}\left(\Delta\psi_{\ell-1}(\sigma_{{\ell-1}}^j(\underline b))
		-\int \Delta\psi_{\ell-1}\,d\fb_{\ell-1}\right)\Big\rvert
	< 2\varepsilon L_2 \frakR_\ell 
		+2\varepsilon k .
\end{equation}
\end{lemma}

\begin{proof}
As $\ell-1\ge L_0$, by assertion (I), Proposition \ref{pro:LD} applied to $\psi_{\ell-1,k_0}$, $2^{-\ell}<\varepsilon$, $\cA_{\ell-1}$, and $R_{\ell-1}$ let us control large deviation. More precisely it provides a set $A\subset (\cA_{\ell-1})^\bZ$ satisfying $\fb_{\ell-1}(A)>1-\varepsilon$ such that for every $\underline b\in A$, $i=0,\ldots,m_\ell-1$, and $k\in\{1,\ldots,m_\ell-1\}$ we have 
\[
	\Big\lvert
	\sum_{j=i}^{i+k-1}
		\left(\Delta\psi_{\ell-1,k_0}(\sigma_{{\ell-1}}^j(\underline b))
		-\int \Delta\psi_{\ell-1,k_0}\,d\fb_{\ell-1}\right)\Big\rvert
	<m_\ell(2\var_{\cA_{\ell-1}}(\Delta\psi_{\ell-1,k_0})+\varepsilon).	
\]
To estimate the right hand side, we apply Claim \ref{clacor:ka3} and Proposition \ref{procor:notormenta} (3). It follows 
\[\begin{split}
	2m_\ell\var_{\cA_{\ell-1}}(\Delta\psi_{\ell-1,k_0})
	&\le m_\ell L_2\frac{1}{2^{\ell-2}}\frakR_{\ell-1}
	< \frakR_\ell\frac{L_2}{2^{\ell-2}}.
\end{split}\]
Analogously, using also $m_\ell<m_\ell\max R_{\ell-1}<L_2\frakR_\ell$ and then \eqref{eq:gettingclose}, it follows
\[
	\Big\lvert
	\sum_{j=i}^{i+k-1}\left(\Delta\psi_{\ell-1,k_0}(\sigma_{{\ell-1}}^j(\underline b))
		-\int \Delta\psi_{\ell-1,k_0}\,d\fb_{\ell-1}\right)\Big\rvert
	< 2\varepsilon L_2 \frakR_\ell .
\]
Finally, to substitute the approximating function $\psi_{\ell-1,k_0}$ by $\psi_{\ell-1}$, using the first estimate in \eqref{eq:gettingclose} we get
$
	\lVert \psi_{\ell-1,k_0}-\psi_{\ell-1}\rVert
	<\varepsilon
$
and hence we obtain \eqref{eq:regena}.
\end{proof}

\subsection{Transporting good orbits from $\fS_{\ell-1}$ to $\fS_n$.}
\label{sec:separating}
We now study appropriate subsets of the principal 
part $\fR_n$ of the suspension space $\fS_n$. 
We invoke Lemma \ref{lemcla:invoke} on level $\ell-1$ to control large deviations on certain orbits and ``transport'' them to level $n$.
Note that if $\va$ is an $(\ell,n)$-address, then $0\va$ is an $(\ell-1,n)$-address. 
Given $A\subset(\cA_{\ell-1})^\bZ$ as in Lemma \ref{lemcla:invoke}, recalling the definition of the map $L_n^{(\ell-1,0\va)}$ in \eqref{eq:Lnellalpha}, let
\[
	A_n^{(\ell-1,0\va)}
	\eqdef L_n^{(\ell-1,0\va)}(A\times\{0\})
	= P_{n,\ell-1}^{-1}(A\times\{0\})
	\subset \fG_n^{(\ell-1,0\va)}
	=\fG_n^{(\ell,\va)}.
\]
Recall the definition of the measure $\fb_n^{(\ell,\va)}$ in  \eqref{eq:defbnellva}. 
By Lemma \ref{lem:defBerinmflo}, the Bernoulli measure $\fb_{\ell-1}$ lifts naturally to its copy $\fb_n^{(\ell-1,0\va)}$ and it holds
\begin{equation}\label{eq:analogously}\begin{split}
	\fb_n^{(\ell-1,0\va)}(A_n^{(\ell-1,0\va)})
	&= (\fb_{\ell-1}\circ\fp_{\ell-1}\circ P_{n,\ell-1}) (P_{n,\ell-1}^{-1}(A\times\{0\}))
	= \fb_{\ell-1}(A)\\
	&>1-\varepsilon.
\end{split}\end{equation}

Let us consider the following set of addresses
\begin{equation}\label{eq:grasssch}\begin{split}
	\mathfrak A_n^\ell
	\eqdef \big\{\va=(\mathrm a_\ell,\ldots,\mathrm a_{n-1})\colon
	 &\mathrm a_\ell\in\{0,\ldots,m_{\ell+1}-2\},\\
	 &\mathrm a_k\in\{0,\ldots,m_{k+1}-1\}\text{ for }k\ne\ell\big\}
\end{split}\end{equation}	 
to which our following arguments can be applied. This restriction on $\mathrm a_\ell$ will be explained in the beginning of the proof of Main Lemma \ref{mlemmapro:proofpro} in Section \ref{sec:MainLemma}. Given $\va\in\mathfrak A_n^\ell$, consider all points whose orbits start in $A_n^{(\ell-1,0\va)}$ and pass through the corresponding ``good set'' $A_n^{(\ell-1,0\va+1_\ell)}$ in the \emph{adjacent} intermediate floor:
\begin{equation}\label{eq:Iwill}
	B_n^{(\ell-1,0\va)}
	\eqdef \big\{\zeta\in A_n^{(\ell-1,0\va)}\colon \Phi_n^{\underline R_\ell\circ P_{n,\ell}}(\zeta)\in
		A_n^{(\ell-1,0\va+1_\ell)}\big\}
	\subset\fL_n^{(\ell,\va)}	.
\end{equation}
Indeed, by our restriction on the index $\mathrm a_\ell$ in address $\va$ this adjacent address $0\va+1_\ell$ is admissible. Analogously to \eqref{eq:analogously}, replacing $0\va$ by $0\va+1_\ell$, it holds
\[
	\fb_n^{(\ell-1,0\va+1_\ell)}(A_n^{(\ell-1,0\va+1_\ell)})
	>1-\varepsilon.
\]
Applying again Lemma \ref{lem:defBerinmflo}, the following holds.
\begin{claim}\label{cla:groundinte}
	For every $\va\in\mathfrak A_n^\ell$, it holds $\fb_n^{(\ell-1,0\va)}(B_n^{(\ell-1,0\va)})>1-2\varepsilon$.
\end{claim}

Consider the set of points in the substrip $\fL_n^{(\ell,\va)}$ whose orbit starts in $B_n^{(\ell-1,0\va)}$,
\[
	C_n^{(\ell,\va)}
	\eqdef \big\{\Phi_n^k(\zeta)\colon
	\zeta\in B_n^{(\ell-1,0\va)}, k\in\bN_0\big\}
	\cap \fL_n^{(\ell,\va)}.
\]
As the strips $\fL_n^{(\ell,\va)}$ are pairwise disjoint for different $(\ell,n)$-addresses, these sets are also pairwise disjoint.

\begin{mlemma}\label{mlemmapro:proofpro}
There are constants $C=C(\phi)>1$ and $L_0'\ge L_0(\phi,\varepsilon)$ such that for every $\ell\ge L_0'$, $n\ge \ell+1$, and $(\ell,n)$-address $\va\in\mathfrak A_n^\ell$  it holds
\begin{itemize}
\item[(i)] {\rm(Birkhoff averages)} for every $\zeta \in C_n^{(\ell,\va)}$
\[
	\Big\lvert\frac{1}{\frakR_\ell}
		\sum_{s=0}^{\frakR_\ell-1}\psi_n(\Phi_n^s(\zeta))\big)
	- \int \phi\,d\mu\Big\rvert
	<C\varepsilon,
\]
\item[(ii)] {\rm(Expected roof functions)} for every $\zeta\in B_n^{(\ell-1,0\va)}$
\[
	\Big\lvert
	\sum_{i=0}^{m_\ell-1}
	\big(\underline R_{\ell-1}\circ\sigma_{{\ell-1}}^i\circ \fp_{\ell-1}\circ 
			P_{n,\ell-1}\big)(\zeta)
	- m_\ell\frakR_{\ell-1} \Big\rvert
	<m_\ell \varepsilon	.
\]	
\end{itemize}
\end{mlemma}

We postpone the proof of Main Lemma \ref{mlemmapro:proofpro} to Section \ref{sec:MainLemma}.

\subsection{End of the proof of Proposition \ref{prothe:suspflow}.}\label{sec:endofproof}
%
Let $\ell\ge L_0'$, $n\ge\ell+1$, and $\va\in\mathfrak A_n^\ell$. By  Main Lemma \ref{mlemmapro:proofpro},   every point in the set $C_n^{(\ell,\va)}$ satisfies the claimed approximation property of its $\frakR_\ell$-Birkhoff sum.
 To finish the proof of the proposition, we need to show that this set has large $\lambda_n$-measure.  
\begin{lemma}
	For every $\ell\ge L_0'$ sufficiently large and $n\ge\ell+1$ it holds
\[
	\lambda_n\Big(\cupdot_{\va\in\mathfrak A_n^\ell} C_n^{(\ell,\va)}\Big)
	>1-3\varepsilon.
\]	
\end{lemma}

\begin{proof}
In order to estimate the $\lambda_n$-measure of the union of all such points,  first note that
\[\begin{split}
	&(\fb_n\times\mathfrak m)
	\Big(\cupdot_{\va\in\mathfrak A_n^\ell} C_n^{(\ell,\va)}\Big)
	= \sum_{\va\in\mathfrak A_n^\ell} (\fb_n\times\mathfrak m)(C_n^{(\ell,\va)})\\
	{\tiny{\text{using \eqref{eq:defbnellva}}}}\quad
	&= \sum_{\va\in\mathfrak A_n^\ell} 
		\int_{B_n^{(\ell-1,0\va)}}
		\mathfrak m\big(C_n^{(\ell,\va)}\cap (\{\zeta\}\times\bN)\big)\,d\fb_n^{(\ell-1,0\va)}(\zeta)\\
	{\tiny{\text{by Lemma \ref{lem:defBerinmflo}}}}\quad&\\
	= \sum_{\va\in\mathfrak A_n^\ell} 
		\int_{B_n^{(\ell-1,0\va)}}\sum_{i=0}^{m_\ell-1}&
		(\underline R_{\ell-1}\circ \sigma_{{\ell-1}}^i \circ\fp_{\ell-1}\circ 
			P_{n,\ell-1})(\zeta)
		\,d\big(\fb_{\ell-1}\circ(\fp_{\ell-1}\circ 
			P_{n,\ell-1})\big)(\zeta)\\
	{\tiny{\text{Main Lemma \ref{mlemmapro:proofpro}}}}\quad		
	&> \sum_{\va\in\mathfrak A_n^\ell}\fb_n^{(\ell-1,0\va)}(B_n^{(\ell-1,0\va)}) 
		\cdot m_\ell(\frakR_{\ell-1}-\varepsilon)\\
	{\tiny{\text{by Claim	\ref{cla:groundinte}}}}\quad
	&= \card\mathfrak A_n^\ell\cdot
		(1-2\varepsilon)\cdot m_\ell(\frakR_{\ell-1}-\varepsilon)\\
	{\tiny{\text{by \eqref{eq:grasssch}}}}\quad	
	&=  m_n\cdots m_{\ell+2}(m_{\ell+1}-1) \cdot(1-2\varepsilon)\cdot m_\ell
			(\frakR_{\ell-1}-\varepsilon)\\
	&>  m_n\cdots m_\ell \frakR_{\ell-1}
		\big(1-\frac{1}{m_{\ell+1}})\cdot(1-2\varepsilon\big)
		(1-\varepsilon).		
\end{split}\]
On the other hand, by \eqref{eq:fodase},
\[\begin{split}
	(\fb_n\times\mathfrak m)(\fS_n)
	&= (\fb_n\times\mathfrak m)(\fR_n^{(\ell-1)}\,\smallcupdot\, \fT_n^{(\ell-1)})
	= (\fb_n\times\mathfrak m)(\fR_n^{(\ell-1)})+ 
		(\fb_n\times\mathfrak m)(\fT_n^{(\ell-1)})\\
	&= \sum_\va \sum_{j=0}^{m_\ell-1}
		(\fb_n\times\mathfrak m)(\fL_n^{(\ell-1,j\va)})
		+(\fb_n\times\mathfrak m)(\fT_n^{(\ell-1)}),
\end{split}\]
where here the sum is taken over all $(\ell,n)$-addresses $\va$.
To estimate $(\fb_n\times\mathfrak m)(\fT_n^{(\ell-1)})$, recall Remark \ref{rem:tails} about the length of tails added at each step. To simplify the estimate, as the formal localization of the intermediate floors where tails are added is rather involved, we use  again Lemma \ref{lem:defBerinmflo} to ``move between the measures'' on intermediate floors. Together with Proposition \ref{procor:notormenta} (4) we get
\[
	(\fb_n\times\mathfrak m)(\fT_n^{(\ell-1)})
	= \sum_{k=\ell}^nm_n\cdots m_k\int\lvert\underline\ft_{k-1}\rvert\,d\fb_{k-1}
	\le L_2\sum_{k=\ell}^n m_n\cdots m_k\frakR_{k-1}\frac{1}{2^{k-1}}.
\]
Since by Proposition \ref{procor:notormenta} (3) it holds 
\[
	\frakR_{k-1}<(1+L_2\frac{1}{2^{k-1}})m_{k-1}\frakR_{k-2},
\]	
 it follows	
\[\begin{split}
	(\fb_n\times\mathfrak m)(\fT_n^{(\ell-1)})
	&< L_2\sum_{k=\ell}^n m_n\cdots m_\ell\frakR_{\ell-1}\frac{1}{2^{k-1}}
		\cdot \prod_{j=\ell+1}^{k}(1+L_2\frac{1}{2^{j-1}})\\
	&\le \frac{4L_2}{2^\ell} m_n\cdots m_{\ell}\frakR_{\ell-1}
		\cdot e^{2L_2/2^\ell}.
\end{split}\]
On the other hand, using analogous estimates for the principal (that is, nontail) part
\[\begin{split}
	(\fb_n\times\mathfrak m)(\fR_n^{(\ell-1)})
	&= \sum_\va\sum_{j=0}^{m_\ell-1}
		\int (\underline R_{\ell-1}\circ \fp_{\ell-1}\circ P_{n,\ell-1}) \,d\fb_n^{(\ell-1,j\va)}\\
	&\le \sum_\va\sum_{j=0}^{m_\ell-1}\frakR_{\ell-1}
	= ( m_n\cdots m_{\ell+1})\cdot m_\ell \cdot\frakR_{\ell-1}.
\end{split}\]
Putting together the previous estimates, we obtain
\[
	\lambda_n\Big(\bigcup_{\va\in\mathfrak A_n^\ell} C_n^{(\ell,\va)}\Big)
	= \frac{(\fb_n\times\mathfrak m)\big(\bigcup_{\va\in\mathfrak A_n^\ell} C_n^{(\ell,\va)}\big)}
		{(\fb_n\times\mathfrak m)(\fR_n^{(\ell-1)}\cup\fT_n^{(\ell-1)})}
	\ge
	 \frac{(1-\frac{1}{m_{\ell+1}})(1-2\varepsilon)(1-\varepsilon)}
		{1+4L_22^{-\ell}e^{2L_2/2^\ell}}.
\]
To conclude the proof of the lemma, it suffices to take $L_0'\ge L_0$  sufficiently large.
\end{proof}

To conclude the proof of Proposition \ref{prothe:suspflow} it remains to prove Main Lemma \ref{mlemmapro:proofpro}.

\subsection{Proof of Main Lemma \ref{mlemmapro:proofpro}}\label{sec:MainLemma}

Before starting the proof, let us sketch its mains steps.
The elements of the alphabet $\cA_n$ are obtained by concatenating elements on the lower level $\ell-1$. The substitution map $\Sub_{n,\ell-1}$ translates between $\cA_n$ and $\cA_{\ell-1}$.
Accordingly, words in the collection $\cB_n$ are obtained by our repeat-and-tail procedure applied to words on each lower level, in particular on level $\ell-1$. 
By construction, the images of good orbit pieces under the factor $H_{\ell-1}$ are sufficiently close to the images of their counterparts under the factor $H_n$. By implementing controlled distortion, this allows us to compare the finite Birkhoff sums of  $\phi$ of their lifts  on $\fS_{\ell-1}$ with their counterparts on $\fS_n$. 

The large deviation result Proposition \ref{pro:LD}  was obtained for (at most $m_\ell$) consecutive Birkhoff sums on level $\ell-1$. This corresponds to taking concatenated orbit pieces on consecutive substrips $\fL_n^{(\ell-1,j\va)}$ that stretch over two adjacent strips $\fL_n^{(\ell,\va)}$ and $\fL_n^{(\ell,\va+1_\ell)}$. Together they will form an orbit piece of close-to-expected length $\frakR_\ell$. 
This now explains our choice of addresses $\mathfrak A_n^\ell$ in \eqref{eq:grasssch}:  if we started from inside the last strip, $a_\ell=m_{\ell+1}-1$, then what follows after it is not the next strip but the tail.

Fix some $(n,\ell)$-address $\va\in\mathfrak A_n^\ell$. 
Consider a point
\begin{equation}\label{eq:interscetionhyp}
	\gamma\in B_n^{(\ell-1,0\va)}
	\subset\fG_n^{(\ell,\va)}.
\end{equation}	 
We will show that item (ii) in Main Lemma is true for every such $\gamma$ (Lemma \ref{cla:secondestimateML}). We also show that item (i) in Main Lemma holds for every point in the slice which is ``in the same fiber'' of the suspension space as $\gamma$, that is, for every 
\[
	\zeta
	\in \cL_n^{(\ell,\va)}
	\spac{so that}
	\fp_n(\zeta)
	=\fp_n(\gamma).
\]
To prove the lemma, in Step 2 we first consider $\zeta$ in some intermediate floor, see Lemma \ref{cla:secondestimateML1a}. The general case is concluded in Step 3, see Lemma \ref{cla:secondestimateML1b}.  But first in Step 0 we fix some notation and in Step 1 we implement Section \ref{sec:internal}.

\smallskip\noindent\textbf{Step 0: Auxiliary codification of orbits.}
For the following see Figure \ref{fig.40}.
For $\gamma$ as in \eqref{eq:interscetionhyp}, let 
\[
	\underline a
	\eqdef \fp_n(\gamma)
	\in(\cA_n)^\bZ.
\]
Recalling Lemma \ref{lem:hund} which expresses the unique point of the intersection of $\underline a$-fiber with the intermediate floor $\fG_n^{(\ell,\va)}$,  for $j=0,\ldots,m_\ell-1$ write
\begin{equation}\label{eq:defzetai}
	\zeta_j
	\eqdef (\underline a,s_n^{(\ell-1,j\va)}(\underline a))
	\in \fG_n^{(\ell-1,j\va)}
	\subset \fL_n^{(\ell,\va)}
	\subset \fS_n.
\end{equation}
Note that 
\[
	\zeta_0
	= (\underline a,s_n^{(\ell-1,0\va)}(\underline a))
	= (\underline a,s_n^{(\ell,\va)}(\underline a))
	= \gamma.
\]
Note that every $\zeta_j$ is in the orbit of $\zeta_0=\gamma$ (with respect to the suspension map $\Phi_n$). Analogously, choose points on intermediate floors within the adjacent slice addressed by $\va+1_\ell$ by letting
\[
	 w_j
	\eqdef (\underline a,s_n^{(\ell-1,j(\va+1_\ell))}(\underline a))
	\in \fG_n^{(\ell-1,j(\va+1_\ell))}
	\subset\fL_n^{(\ell,\va+1_\ell)}
\]
and note that $ w_j$ is also on the orbit of $\gamma$. 

\bigskip
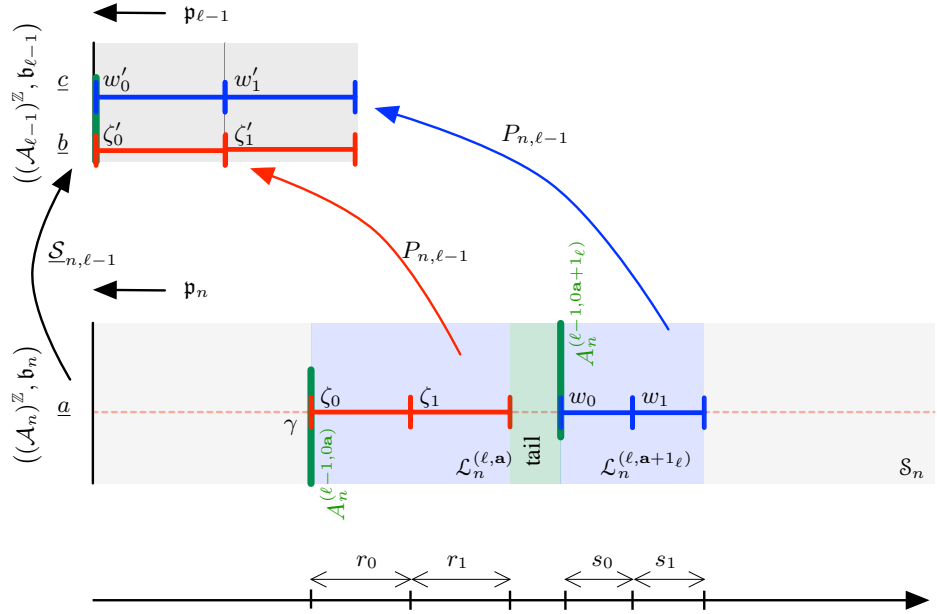
\begin{figure}[h] 
 \begin{overpic}[scale=0.90]{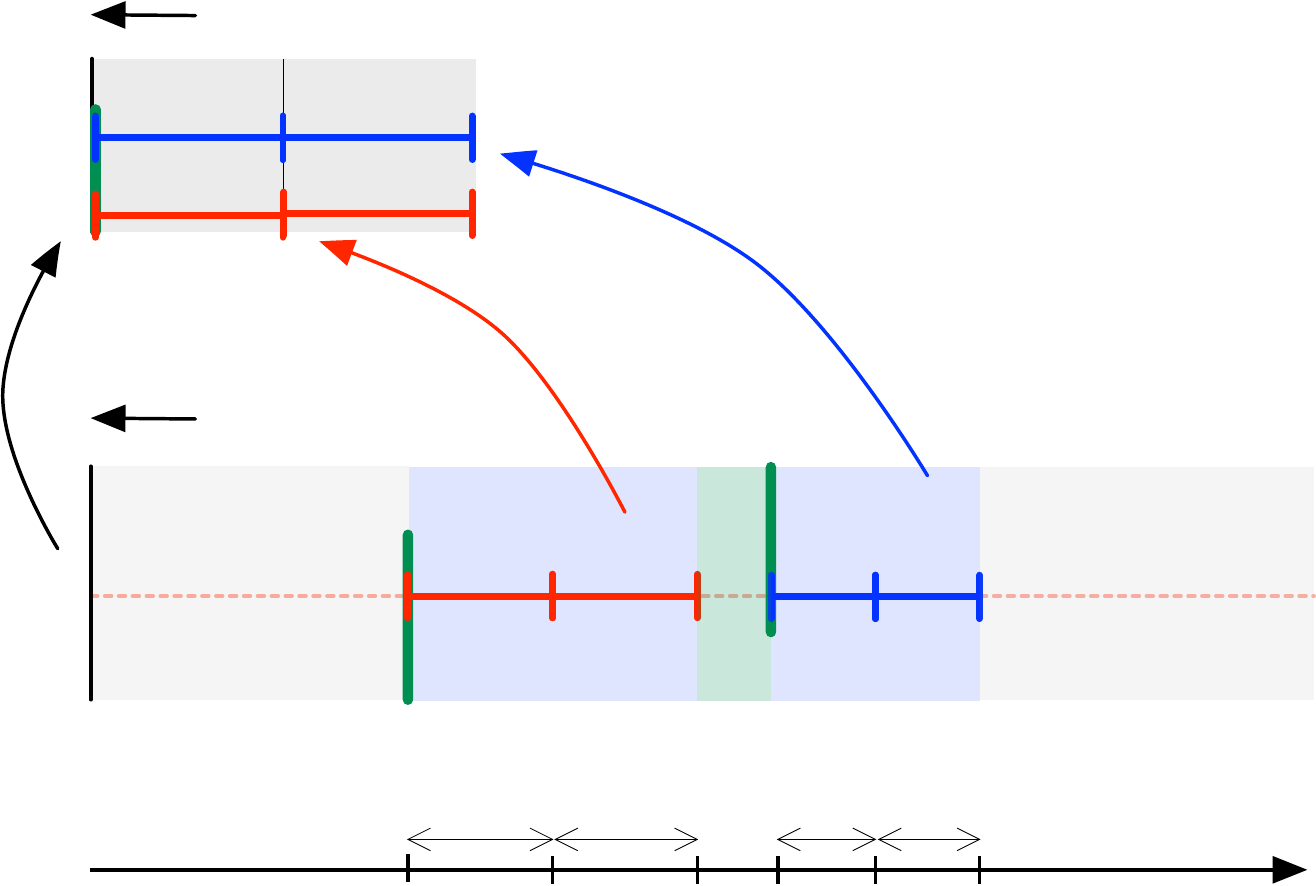}
 	\put(-2,47){{\rotatebox{90}{\small$\big((\cA_{\ell-1})^\bZ,\fb_{\ell-1}\big)$}}}
 	\put(-2,16){\small{\rotatebox{90}{$\big((\cA_n)^\bZ,\fb_n\big)$}}}
	\put(3,22){\small{$\underline a$}}
	\put(3,51){\small{$\underline b$}}
	\put(3,58){\small{$\underline c$}}
	\put(17,65.5){{\small$\fp_{\ell-1}$}}
	\put(17,35){{\small$\fp_n$}}
 	\put(96,15){{\small$\fS_n$}}
 	\put(47,15){{\small$\fL_n^{(\ell,\va)}$}}
 	\put(63,15){{\small$\fL_n^{(\ell,\va+1_\ell)}$}}
	\put(32,10){{\rotatebox{90}{\small\textcolor{green}{$A_n^{(\ell-1,0\va)}$}}}}
	\put(59.5,27){{\rotatebox{90}{\small\textcolor{green}{$A_n^{(\ell-1,0\va+1_\ell)}$}}}}
	\put(54.5,16){{\rotatebox{90}{\small{tail}}}}
	\put(8,52){{\small$\zeta_0'$}}
	\put(22.5,52){{\small$\zeta_1'$}}
	\put(32,23){{\small$\zeta_0$}}
	\put(28,20){{\small$\gamma$}}
	\put(43,23){{\small$\zeta_1$}}
	\put(59.5,23){{\small$ w_0$}}
	\put(67.5,23){{\small$ w_1$}}
	\put(8,58){{\small$ w_0'$}}
	\put(22.5,58){{\small$ w_1'$}}
	\put(2,39){{\small$\underline \Sub_{n,\ell-1}$}}
	\put(41,39){{\small$ P_{n,\ell-1}$}}
	\put(52,52){{\small$ P_{n,\ell-1}$}}
	\put(36,5){{\small$r_0$}}
	\put(46,5){{\small$r_1$}}
	\put(62,5){{\small$s_0$}}
	\put(69,5){{\small$s_1$}}
 \end{overpic}
  \caption{Points ``in the same fiber'' as $\gamma=\zeta_0$ in the suspension space $\fS_n$ project to their counterparts in the model space $\fS_{\ell-1}$}
 \label{fig.40}
\end{figure}

To define the counterparts of $\zeta_0, w_0$ on the model space, let
\[\begin{split}
	\underline b
	&\eqdef \underline\Sub_{n,\ell-1}(\underline a)
	= (\fp_{\ell-1}\circ P_{n,\ell-1})(\zeta_0)
	\in(\cA_{\ell-1})^\bZ\\
	\underline c
	 &\eqdef (\fp_{\ell-1}\circ P_{n,\ell-1})( w_0)
	 \in(\cA_{\ell-1})^\bZ
\end{split}	\]
(recall the definition of $\underline\Sub_{n,\ell-1}$ in \eqref{def:Tnell} and compare Figure \ref{fig.40}). 
Recalling the choice of $A\subset(\cA_{\ell-1})^\bZ$ in Lemma \ref{lemcla:invoke} and the definition of $B_n^{(\ell-1,0\va)}$ in \eqref{eq:Iwill},
 it holds $\underline b,\underline c\in A$.
For $j=0,\ldots,m_\ell-1$ let
\begin{equation}\label{eq:defzetai-new}\begin{split}
	\zeta_j'
	&\eqdef (\sigma_{{\ell-1}}^j(\underline b),0)
	\in \fS_{\ell-1},\\
	 w_j'
	&\eqdef \big(\sigma_{{\ell-1}}^j(\underline c),0)
	\in\fS_{\ell-1}.
\end{split}\end{equation}
One checks that, using the notation \eqref{eq:posintfloooor}, it holds
\begin{equation}\label{eq:defzetai-old}
\begin{split}
	\varsigma_n^{(\ell-1,j\va)}(\underline a)
	&= b_j^{(\ell-1)},
	\spac{where}
	\underline b=(\ldots,b^{(\ell-1)}_{-1}|b^{(\ell-1)}_0,\ldots),\\
	\varsigma_n^{(\ell-1,j(\va+1_\ell))}(\underline a)
	&= c_j^{(\ell-1)},
	\spac{where}
	\underline c=(\ldots,c^{(\ell-1)}_{-1}|c^{(\ell-1)}_0,\ldots).
\end{split}	
\end{equation}
Also let
\begin{equation}\label{eq:defrisi}
	r_j
	\eqdef R_{\ell-1}(b_j^{(\ell-1)})
	=\underline R_{\ell-1}(\sigma_{{\ell-1}}^j(\underline b))
\end{equation}
and note that for $j=0,\ldots,m_\ell-2$ this number is the length of the orbit segment (with respect to the suspension map) between $\zeta_j$ and $\zeta_{j+1}$, that is
\[
	r_j= s_n^{(\ell-1,(j+1)\va)}(\underline a)-s_n^{(\ell-1,j\va)}(\underline a).
\]
For $j=0,\ldots,m_\ell-1$ also let
\begin{equation}\label{eq:defrisi-s}
	s_j
	\eqdef R_{\ell-1}(c_j^{(\ell-1)})
	= \underline R_{\ell-1}(\sigma_{{\ell-1}}^j(\underline c)).
\end{equation}
Finally, by the estimate of  the maximal length of a tail added at level $\ell$ in Proposition \ref{procor:notormenta} (4) and using \eqref{eq:gettingclose}, it holds
\begin{equation}\label{eq:estitails}
	s_n^{(\ell-1,0\va+1_\ell)}(\underline a)
	-\big(s_n^{(\ell-1,(m_\ell-2)\va)}(\underline a)+r_{m_\ell-2}\big)
	\le \max \lvert \ft_\ell\rvert
	\le L_2\frac{1}{2^\ell}\frakR_\ell
			<\varepsilon\cdot\frakR_\ell.
\end{equation}

\smallskip\noindent\textbf{Step 1: Implementing the internal structure of horseshoes.}
Let us now use Section \ref{sec:internal}. Recall the factor map $H_{\ell-1}\colon \fS_{\ell-1}\to\Gamma_{\ell-1}$. By definition of $\zeta_j', w_j'$ together with Lemma \ref{lemcor:msoqui}, we have
\[\begin{split}
	H_{\ell-1}(\zeta_j')
	&\in\Sigma_N^-\times[\Cut_{\ell-1}^{-1}(b_j^{(\ell-1)})]^+\times J,\\
	H_{\ell-1}( w_j')
	&\in\Sigma_N^-\times[\Cut_{\ell-1}^{-1}(c_j^{(\ell-1)})]^+\times J.
\end{split}\]
By definition, the above points are in $\Lambda_{\ell-1}$.
By definition of $\zeta_j$ in \eqref{eq:defzetai} together with Lemma \ref{lemcor:msoqui} it follows, it holds
\[\begin{split}
	H_n(\zeta_j)
	=H_n\big(\underline a,s_n^{(\ell-1,j\va)}(\underline a)\big)
	&\in\Sigma_N^-\times[(\Cut_{\ell-1}^{-1}\circ \varsigma_n^{(\ell-1,j\va)})(\underline a)]^+\times J \\
	{\tiny{\text{using \eqref{eq:defzetai-old}}}}\quad
	&=\Sigma_N^-\times[\Cut_{\ell-1}^{-1}(b_j^{(\ell-1)})]^+\times J.
\end{split}\]
Analogously,
\[	H_n( w_j)
	\in\Sigma_N^-\times[\Cut_{\ell-1}^{-1}(c_j^{(\ell-1)})]^+\times J.
\]	
Note that $H_n(\zeta_j),H_n( w_j)$ are points in $\Lambda_n$. 

Proposition \ref{pro:distortion} implies the following key distortion estimate.
For its statement and proof we use the usual short notation for a Birkhoff sum $S_n\varphi=\varphi+\varphi\circ G+\ldots+\varphi\circ G^{n-1}$; the map $G$ is given by the context. Recall that $\psi_n$ is the lift of $\phi$ to $\fS_n$.

\begin{claim}\label{cla:distortion}
	With the notation above, for every $j=0,\ldots,m_\ell-1$ it holds
\[\begin{split}
	&\Big\lvert S_{r_j}\psi_n(\zeta_j)
		- S_{r_j}\psi_{\ell-1}(\zeta_j')\Big\rvert
	= \Big\lvert S_{r_j}\phi(H_n(\zeta_j))
		-S_{r_j}\phi(H_{\ell-1}(\zeta_j'))\Big\rvert	
	<\varepsilon r_j,\\
	&\Big\lvert S_{s_j}\psi_n( w_j)
		- S_{s_j}\psi_{\ell-1}( w_j')\Big\rvert
	= \Big\lvert S_{s_j}\phi(H_n( w_j))
		-S_{s_j}\phi(H_{\ell-1}( w_j'))\Big\rvert	
	<\varepsilon s_j.
\end{split}\]	
\end{claim}

\smallskip\noindent\textbf{Step 2: Birkhoff sums for $\zeta_j$.}

 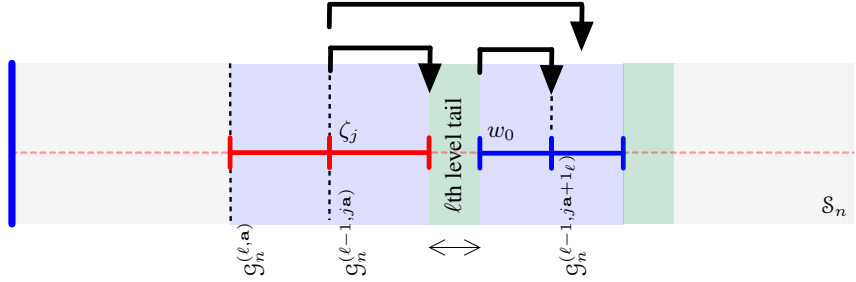
\begin{figure}[h] 
 \begin{overpic}[scale=.9]{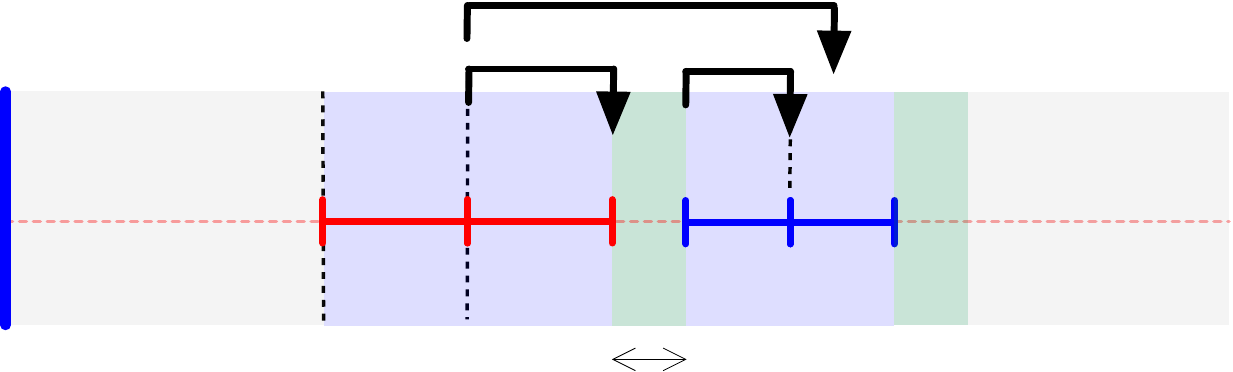}
 	\put(96,5){{\small$\fS_n$}}
 	\put(39,14){{\small$\zeta_j $}}
 	\put(56.5,14){{\small$ w_0$}}
	\put(51.5,5){{\rotatebox{90}{\small{$\ell$th level tail}}}}
	\put(27,-2.5){{\rotatebox{90}{\small$\fG_n^{(\ell,\va)}$}}}
	\put(39,-2.5){{\rotatebox{90}{\small$\fG_n^{(\ell-1,j\va)}$}}}
	\put(65,-2.5){{\rotatebox{90}{\small$\fG_n^{(\ell-1,j\va+1_\ell)}$}}}
 \end{overpic}
  \caption{$\ell$th level Birkhoff sums starting at some intermediate floor with $(\ell-1,n)$-address $j\va$  are split into $(\ell-1)$st level Birkhoff sums, but ignoring any tail of level $\ell$.}
 \label{fig.2}
\end{figure}

\begin{lemma}\label{cla:secondestimateML1a}
	Item (i) in Main Lemma \ref{mlemmapro:proofpro} holds for every $\zeta_j$, $j\in\{0,\ldots,m_\ell-1\}$, taking $C=C_0\eqdef 5\lVert\phi\rVert+8+4L_2$.
\end{lemma}


\begin{proof}
Given $j\in\{0,\ldots,m_\ell-1\}$, to estimate the Birkhoff sum 
\begin{equation}\label{eq:almostBirkhoff}
	\sum_{s=0}^{\frakR_\ell-1}\psi_n(\Phi_n^s(\zeta_j)),
\end{equation}
 we separate $m_\ell$ (disjoint) Birkhoff sums at level $(\ell-1)$ which start at  intermediate floors with $(\ell-1,n)$-addresses 
 \[	
 j\va ,(j+1)\va,\ldots, (m_\ell-1)\va  
 \spac{and}
 0(\va+1_\ell), 1(\va+1_\ell),\ldots,(j-1)(\va+1_\ell),
\] 
 respectively. Compare Figure \ref{fig.2}.
Let
\[
	S_1
	=S_1(j)
	\eqdef \sum_{i=j}^{m_\ell-1}
		S_{r_i}\psi_n(\zeta_i),
		\quad
	S_2
	=S_2(j)
	\eqdef \sum_{i=0}^{j-1}
		S_{s_i}\psi_n( w_i),
\]
where $S_1$ corresponds to the first collection of addresses and $S_2$ to the second one. Note that $S_1+S_2$ is  almost equal to the Birkhoff sum \eqref{eq:almostBirkhoff} except for the following two facts:  
\begin{itemize}
\item[(a)] the sum $S_1+S_2$  takes into account (disjoint) orbit pieces whose total length is in general not equal to the ``expected'' value  $\frakR_\ell$, 
\item[(b)] the sum $S_1+S_2$ ignores all values of $\psi_n$ at points of the tail added at the $\ell$th level.
\end{itemize}

To address item (a), first note that calculating $S_1+S_2$ we sum over $m_\ell$ orbit pieces each having a length very close to the expected one $\frakR_{\ell-1}$. Let us estimate this deviation:
\begin{equation}\label{eq:stranho}\begin{split}
	D_1
	&\eqdef \Big\lvert 
		m_\ell\frakR_{\ell-1}
		- \sum_{i=j}^{m_\ell-1}r_i
		- \sum_{i=0}^{j-1}s_i
	\Big\rvert\\
	\tiny{\text{recalling  \eqref{eq:defrisi}, \eqref{eq:defrisi-s} }}\quad
	&=\Big\lvert 
		m_\ell\frakR_{\ell-1}
		- \sum_{i=j}^{m_\ell-1}\underline R_{\ell-1}(\sigma_{{\ell-1}}^i(\underline b))
		- \sum_{i=0}^{j-1}\underline R_{\ell-1}(\sigma_{{\ell-1}}^i(\underline c))
	\Big\rvert\\
	&\le \Big\lvert 
		(m_\ell-j)\frakR_{\ell-1}
		- \sum_{i=j}^{m_\ell-1}\underline R_{\ell-1}(\sigma_{{\ell-1}}^i(\underline b))
		\Big\rvert	\\
	&\phantom{\le \left\lvert \frakR_\ell - m_\ell\frakR_{\ell-1}\right\rvert}
		+ \Big\lvert 
		j\frakR_{\ell-1}
		- \sum_{i=0}^{j-1}\underline R_{\ell-1}(\sigma_{{\ell-1}}^i(\underline c))
		\Big\rvert	\\
	{\tiny \text{applying \eqref{eq:applyLD} twice and using $m_\ell<\frakR_\ell$}}\quad
	&\le  	 2m_\ell \varepsilon
	<2\varepsilon\frakR_\ell
\end{split}\end{equation}

To address now item (b), first recall that by \eqref{eq:estitails} the length of the ``tail between the orbit pieces'' where the Birkhoff sums $S_1$ and $S_2$ are taken  is at most $\varepsilon\frakR_\ell$.
Further, by Proposition \ref{procor:notormenta} (3), the estimate \eqref{eq:stranho} of $D_1$, and \eqref{eq:gettingclose} and the choice of $\ell$, it holds
\begin{equation}\label{eq:caiu}
	\Big\lvert \frakR_{\ell} -  \sum_{i=j}^{m_\ell-1}r_i- \sum_{i=0}^{j-1}s_i\Big\rvert
	\le  \left\lvert 	\frakR_{\ell} - m_\ell \frakR_{\ell-1}\right\rvert
		+ D_1
	\le 	L_2\frac{1}{2^\ell}\frakR_\ell
		+ 2\varepsilon\frakR_\ell
	\le 3\varepsilon\frakR_\ell	.
\end{equation}

The Birkhoff sum  in \eqref{eq:almostBirkhoff} takes values over the same collection of points as in the sum $S_1+S_2$, except for two blocks of points. The first block consists of points on the tail (at most $\varepsilon\frakR_\ell$ points). The second block  consists of points at the end of the orbit piece in \eqref{eq:almostBirkhoff} (at most the difference between $\frakR_\ell$ and the sum of terms in  the tail, $S_1$, and $S_2$; that is, at most $\varepsilon\frakR_\ell+3\varepsilon\frakR_\ell$ terms). 
Hence, it follows
\begin{equation}\label{eq:proof1}
\begin{split}
	\left\lvert 
		\sum_{s=0}^{\frakR_{\ell} -1}\psi_n(\Phi_n^s(\zeta_j))
		- (S_1+S_2)
	\right\rvert
	\le 
	(\varepsilon\frakR_\ell+(\varepsilon\frakR_\ell+3\varepsilon\frakR_\ell)) \lVert\psi_n\rVert
	\le 
	5\varepsilon  \lVert\phi\rVert \frakR_\ell,
\end{split}\end{equation}
where for the second inequality we also used  $\lVert\psi_n\rVert\le\lVert\phi\rVert$.
This concludes the discussion of the obstructions (a) and (b).

As next step let us estimate $S_1$ and $S_2$.  First note that
\[\begin{split}
	&\Big\lvert S_1 
		- (m_\ell -j)\frakR_{\ell-1}\int\psi_{\ell-1}\,d\lambda_{\ell-1}\Big\rvert\\
	&= \Big\lvert \sum_{i=j}^{m_\ell-1}S_{r_i}\psi_n(\zeta_i)
		- (m_\ell -j)\frakR_{\ell-1}\int\psi_{\ell-1}\,d\lambda_{\ell-1}\Big\rvert\\	
	&\le \Big\lvert 
		\sum_{i=j}^{m_\ell-1}\left(S_{r_i}\psi_n(\zeta_i)
		- S_{r_i}\psi_{\ell-1}(\zeta_i')\right)\Big\rvert
		+\Big\lvert 
		\sum_{i=j}^{m_\ell-1}\Big(S_{r_i}\psi_{\ell-1}(\zeta_i')
		-\frakR_{\ell-1} \int\psi_{\ell-1}\,d\lambda_{\ell-1}\Big)\Big\rvert.
\end{split}\]
To estimate the first term, by Claim \ref{cla:distortion}, we obtain
\[
	\Big\lvert 
		\sum_{i=j}^{m_\ell-1}\left(S_{r_i}\psi_n(\zeta_i)
		- S_{r_i}\psi_{\ell-1}(\zeta_i')\right)\Big\rvert
	\le \varepsilon\sum_{i=j}^{m_\ell-1}r_i.
\]
To estimate the second term, note that  
\[\begin{split}
	\Big\lvert
	&\sum_{i=j}^{m_\ell-1}\left(S_{r_i}\psi_{\ell-1}(\zeta_i')
		-\frakR_{\ell-1} \int\psi_{\ell-1}\,d\lambda_{\ell-1}
		\right)\Big\rvert\\
	{\tiny{\text{by definition of $\zeta_i'$ in \eqref{eq:defzetai-new}}}}\quad	
	&= \Big\lvert	
	\sum_{i=j}^{m_\ell-1}\left(\Delta\psi_{\ell-1}(\sigma_{{\ell-1}}^i(\underline b))
		-\frakR_{\ell-1} \int\psi_{\ell-1}\,d\lambda_{\ell-1}	
		\right)\Big\rvert	\\
	{\tiny{\text{by Abramov's formula in Lemma \ref{lem:Abramov}}}}\quad
	&= \Big\lvert	
	\sum_{i=j}^{m_\ell-1}\left(\Delta\psi_{\ell-1}(\sigma_{{\ell-1}}^i(\underline b))
		- \int\Delta\psi_{\ell-1}\,d\fb_{\ell-1}	
		\right)\Big\rvert\\
	{\tiny{\text{by \eqref{eq:regena}	}}}\quad
	&<2\varepsilon L_2\frakR_\ell+2\varepsilon (m_\ell-i).
\end{split}\]
Putting the previous estimates together, we get
\[
	\Big\lvert S_1 
		- (m_\ell -j)\frakR_{\ell-1}\int\psi_{\ell-1}\,d\lambda_{\ell-1}\Big\rvert
	\le 	 \varepsilon\sum_{i=j}^{m_\ell-1}r_i
		+2\varepsilon L_2\frakR_\ell
		+2\varepsilon(m_\ell-i).
\]
We get the analogous estimate for $S_2$,
\[
	\Big\lvert S_2 
		- j\frakR_{\ell-1}\int\psi_{\ell-1}\,d\lambda_{\ell-1}\Big\rvert
	\le 	 \varepsilon\sum_{i=0}^{j-1}s_i	
		+2\varepsilon L_2\frakR_\ell
		+2\varepsilon j.
\]
This implies
\begin{equation}\label{eq:proof2}
\begin{split}
	\Big\lvert (S_1+S_2) 
	&- m_\ell \frakR_{\ell-1}\int\psi_{\ell-1}\,d\lambda_{\ell-1}\Big\rvert\\
	&\le \varepsilon\Big(\sum_{i=j}^{m_\ell-1}r_i
		+\sum_{i=0}^{j-1}s_i\Big)
		+ 4\varepsilon L_2\frakR_\ell 
		+2\varepsilon m_\ell\\
	{\tiny{\text{by \eqref{eq:caiu} and $m_\ell<\frakR_\ell$}}}	\,\,
	&< \varepsilon\frakR_\ell
		+\varepsilon 3\varepsilon\frakR_\ell
		+ 4\varepsilon L_2\frakR_\ell 
		+2\varepsilon \frakR_\ell
	= \varepsilon
		\Big( 3+3\varepsilon+4L_2\Big)\frakR_\ell .
\end{split}\end{equation}
This finishes the estimate of $S_1$ and $S_2$.

Note that by assumption $\ell\ge L_0$ with  \eqref{eq:gettingclose-phi} and using $\psi_{\ell-1}=\phi\circ H_{\ell-1}$ and $\mu_{\ell-1}=(H_{\ell-1})_\ast\lambda_{\ell-1}$ we  have
\[
	\Big\rvert\int\psi_{\ell-1}\,d\lambda_{\ell-1}
		-\int\phi\,d\mu\Big\rvert
	= \Big\rvert\int\phi\,d\mu_{\ell-1}
		-\int\phi\,d\mu\Big\rvert	
	<\varepsilon.	
\]
Finally, with the above, we conclude
\begin{equation}\label{eq:proof3}
\begin{split}
	\Big\lvert m_\ell \frakR_{\ell-1}&\int\psi_{\ell-1}\,d\lambda_{\ell-1}
		- \frakR_{\ell}\int\phi\,d\mu \Big\rvert\\
	&\le 	\Big\lvert m_\ell \frakR_{\ell-1}- \frakR_{\ell}\Big\rvert \cdot
		\lVert\phi\rVert
	+ \frakR_\ell\cdot\Big\rvert\int\psi_{\ell-1}\,d\lambda_{\ell-1}
		-\int\phi\,d\mu\Big\rvert\\
	{\tiny{\text{by Proposition \ref{procor:notormenta} (3)
	}}}\quad
	&\le 	L_2\frac{1}{2^\ell}\frakR_\ell \cdot\lVert\phi\rVert
		+\frakR_\ell\cdot\varepsilon
		= \Big(L_2\frac{1}{2^\ell}\lVert\phi\rVert+\varepsilon\Big)\frakR_\ell\\
	{\tiny\text{by \eqref{eq:gettingclose}}}\quad
	&\le 2\varepsilon\cdot\frakR_\ell.
\end{split}\end{equation}
Hence, \eqref{eq:proof1},  \eqref{eq:proof2}, and \eqref{eq:proof3} together imply
\[\begin{split}
	\left\lvert
		\frac{1}{\frakR_\ell}
			\sum_{s=0}^{\frakR_{\ell}-1}\psi_n(\Phi_n^s(\zeta_j))
		- \int\phi\,d\mu\right\rvert
	&\le (5\lVert\phi\rVert+8+4L_2)\varepsilon.
\end{split}\]
This proves the lemma.
\end{proof}

\smallskip\noindent\textbf{Step 3: Birkhoff sums for any other $\zeta\in C_n^{(\ell,\va)}$.}

\begin{lemma}\label{cla:secondestimateML1b}
	Item (i) in Main Lemma \ref{mlemmapro:proofpro} holds for every $\zeta\in C_n^{(\ell,\va)}$ taking $C=C_0+1$.
\end{lemma}

\begin{proof}
Choose $j\in\{0,\ldots,m_\ell-1\}$ which addresses the previous intermediate floor of level $(\ell-1)$, that is, using notation \eqref{eq:defzetai}, choose the minimal index $j$ for which there is $r\ge0$ such that $\zeta=\Phi_n^r(\zeta_j)$ (compare Figure \ref{fig.3}). 

\begin{figure}[h] 
 \begin{overpic}[scale=.9]{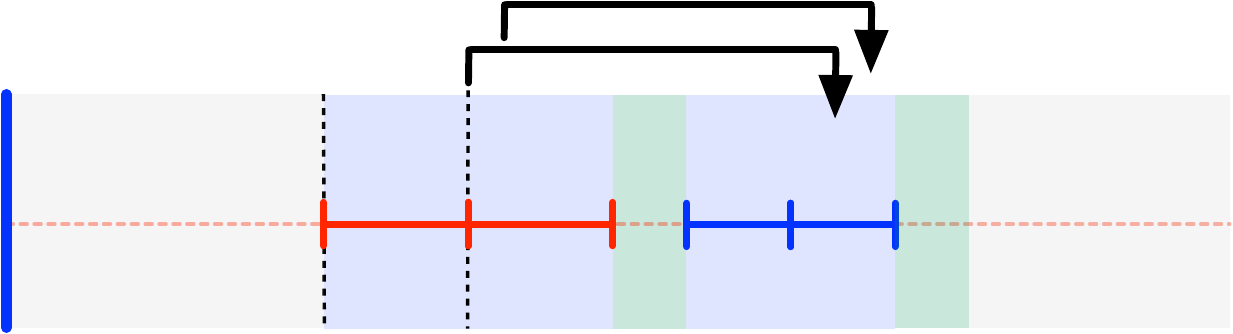}
 	\put(96,1){{\small$\fS_n$}}
 	\put(35,11){{\small$\zeta_j $}}
 	\put(40,11){{\small$\zeta$}}
	\put(51.5,5){{\rotatebox{90}{\small{$\ell$th level tail}}}}
	\put(27,-2.5){{\rotatebox{90}{\small$\fG_n^{(\ell,\va)}$}}}
	\put(39,-2.5){{\rotatebox{90}{\small$\fG_n^{(\ell-1,j\va)}$}}}
 \end{overpic}
  \caption{$\ell$th level Birkhoff sums starting at any point in a slice (shaded region) are replaced by one starting at some $(\ell-1)$st level intermediate floor}
 \label{fig.3}
\end{figure}
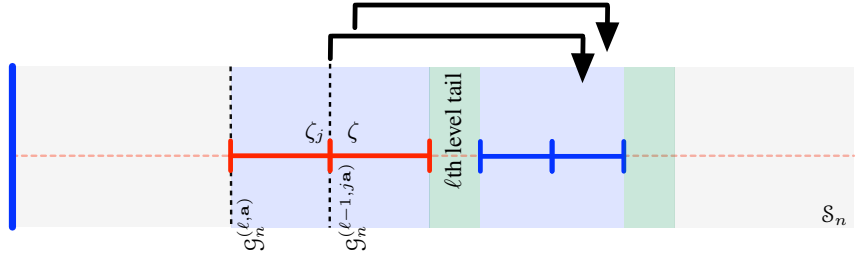
To estimate the differences between the Birkhoff sums along the orbit segments of length $\frakR_\ell$ starting at $\zeta$ and $\zeta_j$, respectively, just note that both  are on the same orbit and both have the same  length  and hence share most of its terms $r+1,\ldots,\frakR_\ell-r$.  Hence, 
\begin{equation}\label{eq:lunch}
	\Big\lvert \sum_{s=0}^{\frakR_\ell-1}\psi_n(\Phi_n^s(\zeta))
		- \sum_{s=0}^{\frakR_\ell-1}\psi_n(\Phi_n^s(\zeta_j))
	\Big\rvert
	\le 2r\lVert\psi_n\rVert
	\le 2r\lVert\phi\rVert.
\end{equation}
Note that together with Proposition \ref{procor:notormenta} (1)--(2) it holds
\begin{equation}\label{eq:lunch2}
	r
	\le \max R_{\ell-1}
	\le \frac{1}{m_\ell}\max R_\ell
	\le \frac{1}{m_\ell}L_2\frakR_\ell.
\end{equation}
Hence
\[\begin{split}
	\Big\lvert 
	\frac{1}{\frakR_\ell}\sum_{s=0}^{\frakR_\ell-1}\psi_n(\Phi_n^s(\zeta))
	- \int\phi\,d\mu\Big\rvert
	&\le \frac{1}{\frakR_\ell}
		\Big\lvert \sum_{s=0}^{\frakR_\ell-1}\psi_n(\Phi_n^s(\zeta))
		- \sum_{s=0}^{\frakR_\ell-1}\psi_n(\Phi_n^s(\zeta_j))
	\Big\rvert\\
	&\quad\phantom{\le}
	+\Big\lvert 
	\frac{1}{\frakR_\ell}\sum_{s=0}^{\frakR_\ell-1}\psi_n(\Phi_n^s(\zeta_j))
	- \int\phi\,d\mu\Big\rvert\\
	{\tiny{\text{using \eqref{eq:lunch}, \eqref{eq:lunch2},  \eqref{eq:gettingclose}, and Lemma \ref{cla:secondestimateML1a}}}}\quad
	&\le 2\lVert\phi\rVert\frac{1}{m_\ell}L_2
		+(5\lVert\phi\rVert+8+4L_2)\varepsilon
	\le (1+C_0)\varepsilon,
\end{split}\]
proving the lemma.
\end{proof}

\begin{lemma}\label{cla:secondestimateML}
	Item (ii) in Main Lemma \ref{mlemmapro:proofpro} is true for every $\gamma\in B_n^{(\ell-1,0\va)}$.
\end{lemma}

\begin{proof}
It suffices to take $j=0$ in \eqref{eq:stranho} to recall that $\zeta_0=\gamma$. 
\end{proof}

As $\gamma\in B_n^{(\ell,\va)}$ was arbitrary, Main Lemma \ref{mlemmapro:proofpro} (i) now is a consequence of Lemmas \ref{cla:secondestimateML1a} and \ref{cla:secondestimateML1b}. Main Lemma \ref{mlemmapro:proofpro}  (ii)  follows from Lemma \ref{cla:secondestimateML}. This finishes the proof.
\qed

\section{Proof of Theorem \ref{teo:2}}\label{sec:proooofb}

Our construction provides the sequences of horseshoes $(\Gamma_n)_n$, as in \eqref{eq:defGammanLambdan}, and Borel probability measures $(\mu_n)_n\subset\cM_{\rm erg}(F)$, as in \eqref{eq:defmun}. By Lemma \ref{cla:convergence}, the sequence $(\mu_n)_n$ weak$\ast$ converges to some probability measure $\mu_\infty$ as $n\to\infty$. By Corollary \ref{cor:tailinghorse}, it holds
\[
	\chi(\mu_\infty)=0
	\spac{and}
	h(F,\mu_\infty)
	\ge e^{-L_1\lvert\alpha\rvert}(h(F,\mu)-\varepsilon_H).
\]	
It remains to show that $\mu_\infty$ is ergodic. For that we will use Proposition \ref{pro:mainrussian} below that is a minor extension of \cite[Lemma 2]{GorIlyKleNal:05}%
\footnote{In \cite{GorIlyKleNal:05}, it is assumed that every measure in the sequence is uniformly distributed on a periodic orbit}. 
For completeness, we prove it in the Appendix.

For every continuous $\phi\colon\Sigma_N\times\bS^1\to\bR$ and $\varepsilon>0$ let $L_0=L_0(\phi,\varepsilon)\in\bN$ as in Proposition \ref{prothe:suspflow}. Hence, for every $\ell\ge L_0$ and $n\ge \ell+1$ the subset 
\[
	\Gamma_{n,\phi,\varepsilon}
	\eqdef H_n(\fS_{n,\phi,\varepsilon})
	\subset \Gamma_n
\]
satisfies 
\[
	\mu_n(\Gamma_{n,\phi,\varepsilon})
	= (H_n)_\ast\lambda_n(\Gamma_{n,\phi,\varepsilon})
	= \lambda_n(\fS_{n,\phi,\varepsilon})
	>1-\varepsilon	.
\]
It also follows that for every $X=H_n(\underline a,s) \in H_n(\fS_{n,\phi,\varepsilon})=\Gamma_{n,\phi,\varepsilon}$ it holds
\[
	\left\lvert\frac{1}{\frakR_\ell}
		\sum_{k=0}^{\frakR_\ell-1}\phi(F^k(X))
		- \int\phi\,d\mu \right\rvert
	= \left\lvert\frac{1}{\frakR_\ell}
		\sum_{k=0}^{\frakR_\ell-1}\psi_n(\Phi_n^k(\underline a,s))
		- \int\phi\,d\mu \right\rvert	
	<\varepsilon,
\]
where we also used the fact that by Proposition \ref{pro:semiconj} the maps $\Phi_n$ and $F|_{\Gamma_n}$ are semi-conjugate by $H_n$.  We now use the following result.

\begin{proposition}\label{pro:mainrussian}
Let $G\colon X\to X$ be a homeomorphism of a compact metric space. Consider sequences of Borel measurable subsets $(\Upsilon_n)_n$ of $X$, Borel measures $(\varrho_n)_n$ on $X$ weak$\ast$ converging to some Borel measure $\varrho$, and positive integers $(T(n))_n$ tending to $\infty$.
	Assume that for every $\phi\colon X\to\bR$ continuous and $\varepsilon>0$, there exists $L=L(\phi,\varepsilon)\in\bN$ such that for every $\ell\ge L$ there exists $N=N(\ell)\ge \ell$ such that for every $n\ge N$ there exists a measurable subset $\Upsilon_{n,\phi,\varepsilon}\subset\Upsilon_n$ with $\varrho_n(\Upsilon_{n,\phi,\varepsilon})>1-\varepsilon$ such that  
\[
	\left\lvert\frac{1}{T(\ell)}\sum_{k=0}^{T(\ell)-1}\phi(G^k(x))-\int\phi\,d\varrho\right\rvert
	<\varepsilon
	\spac{for every}
	x\in\Upsilon_{n,\phi,\varepsilon}.
\]
Then $\varrho$ is $G$-ergodic.	
\end{proposition}

The comments above imply that we can apply Proposition \ref{pro:mainrussian}   with $G=F$, $X=\Sigma_N\times\bS^1$, $\Upsilon_n=\Gamma_n$, $\varrho_n=\mu_n$, $T(n)=\frakR_n$, $L=L_0$, $N=L_0$, $\Upsilon_{n,\phi,\varepsilon}=\Gamma_{n,\phi,\varepsilon}$. Therefore, the limit measure $\mu_\infty$ is ergodic.
\qed

\section*{Appendix: Ergodicity of limit measures (Proof of Proposition \ref{pro:mainrussian})}

Given $\phi\colon X\to\bR$, denote
\[
	\underline\phi(x)\eqdef \liminf_{n\to\infty}\frac1n\sum_{k=0}^{n-1}\phi(G^k(x))
	\spac{and}
	\overline\phi(x)\eqdef \limsup_{n\to\infty}\frac1n\sum_{k=0}^{n-1}\phi(G^k(x)).
\]
Given $\varepsilon>0$, denote the upper topological limit of $(\Upsilon_{n,\phi,\varepsilon})_n$ by $\Upsilon_{\phi,\varepsilon}$, that is,
\[
	\Upsilon_{\phi,\varepsilon}
	\eqdef {\bigcap_{k=1}^\infty\overline{\bigcup_{n=k}^\infty \Upsilon_{n,\phi,\varepsilon}}}
	=\big\{y\in X\colon
		\exists n_k\to\infty, 
		y_k\in \Upsilon_{n_k,\phi,\varepsilon},y=\lim_{k\to\infty}y_k\big\}.
\]
We use the following fact that is straightforward to check.

\begin{claim}\label{lempro:1}
	For every continuous $\phi\colon X\to\bR$  and $\varepsilon>0$ it holds 
\[
	\varrho(\Upsilon_{\phi,\varepsilon})\ge \limsup_{n\to\infty}\varrho_n(\Upsilon_{n,\phi,\varepsilon}).
\]	
\end{claim}

\begin{lemma}\label{lemcla:1}
	For every continuous $\phi\colon X\to\bR$  and $\varepsilon>0$ there exists a set $\Xi_{\phi,\varepsilon}\subset \Upsilon_{\phi,\varepsilon}$ such that $\varrho(\Xi_{\phi,\varepsilon})>1-\varepsilon$ and 
\[
	\int\phi\,d\varrho-\varepsilon
	<\overline\phi(x)=\underline\phi(x)
	<\int\phi\,d\varrho+\varepsilon
\]	 
for every $x\in\Xi_{\phi,\varepsilon}$.
\end{lemma}

\begin{proof}
	Given $\phi$ and $\varepsilon$, let $L=L(\phi,\varepsilon)$ and for $\ell\ge L$ let $N=N(\ell)$ be as in the hypothesis of the proposition. 
By Claim \ref{lempro:1} and our hypothesis, 
\[
	\varrho(\Upsilon_{\phi,\varepsilon})\ge\limsup_n\varrho_n(\Upsilon_{n,\phi,\varepsilon})>1-\varepsilon.
\]	 

Every $x\in\Upsilon_{\phi,\varepsilon}$ is the limit of some sequence of points $x_i$  in $\Upsilon_{n_i,\phi,\varepsilon}$, $n_i\ge \ell$. Hence, for $n_i\ge\ell$ sufficiently large, it holds
\[
	\left\lvert \frac{1}{T(\ell)}\sum_{k=0}^{T(\ell)-1}\phi(G^k(x))
			-\frac{1}{T(\ell)}\sum_{k=0}^{T(\ell)-1}\phi(G^k(x_i))\right\rvert
	<\varepsilon.
\]
By our hypothesis on $x_i\in\Upsilon_{n_i,\phi,\varepsilon}$, it holds
\[
	\left\lvert \frac{1}{T(\ell)}\sum_{k=0}^{T(\ell)-1}\phi(G^k(x_i))
			- \int\phi\,d\varrho\right\rvert
	<\varepsilon.		
\]
Hence, for every $x\in\Upsilon_{\phi,\varepsilon}$ and $\ell\ge1$ sufficiently large it holds 
\[
	\left\lvert \frac{1}{T(\ell)}\sum_{k=0}^{T(\ell)-1}\phi(G^k(x))
	- \int\phi\,d\varrho\right\rvert
	<2\varepsilon.
\]
Therefore, with the notation above, for every $x\in\Upsilon_{\phi,\varepsilon}$
\[
	 \overline\phi(x) > \int\phi\,d\varrho -2\varepsilon
	 \spac{and}
	 \underline\phi(x) < \int\phi\,d\varrho+2\varepsilon.
\]
Applying the Birkhoff theorem to the invariant measure $\varrho$, we get a set $Z$ with $\varrho(Z)=1$ so that at every $z\in Z$ it holds $\overline\phi(z)=\underline\phi(z)$. 
By the above, for every $z\in \Xi_{\phi,\varepsilon}\eqdef Z\cap \Upsilon_{\phi,\varepsilon}$ it holds $\lvert\overline\phi(z)-\phi(\varrho)\rvert<3\varepsilon$ and $\varrho(\Xi_{\phi,\varepsilon})=\varrho(\Upsilon_{\phi,\varepsilon})>1-\varepsilon$. This proves the lemma.
\end{proof}

Let us now prove that $\varrho$ is ergodic. 
Take a dense set of continuous functions $\{\phi_k\}_k$ and a summable sequence of positive numbers $(\varepsilon_k)_k$. As 
\[
	\sum_k\varrho(\Upsilon_{\phi_k,\varepsilon_k}^c)
	\le \sum_k\varepsilon_k<\infty,
\]	 
by the Borel-Cantelli lemma, there is a set $\Upsilon$ satisfying $\varrho(\Upsilon)=1$ such that every $x\in\Upsilon$ is contained in only finitely many sets $\Upsilon_{\phi_k,\varepsilon_k}^c$. It follows that for every continuous $\phi$ and  $x\in\Upsilon$ Birkhoff averages of $\phi$ converge to $\int\phi\,d\varrho$. This implies that $\varrho$ is $G$-ergodic.
\qed

\bibliographystyle{siam}

\end{document}